\documentclass{amsart}
\usepackage{geometry}
\usepackage{newlfont}
\usepackage{amssymb}
\usepackage{amsmath}
\usepackage{amsfonts}
\usepackage{latexsym}
\usepackage{amsthm}
\usepackage{mathrsfs}
\usepackage{stmaryrd}
\usepackage[all,cmtip]{xy}
\usepackage{caption}
\usepackage{enumerate}
\usepackage{hyperref}
\usepackage{xcolor}

\theoremstyle{plain}                    
\newtheorem{teo}{Theorem}[subsection]      
\newtheorem{theoremalpha}{Theorem}

\newtheorem{prop}[teo]{Proposition}
\newtheorem{cor}[teo]{Corollary}       
\newtheorem{lem}[teo]{Lemma}            
\theoremstyle{definition}               
\newtheorem{notations}{}[subsection]   
\newtheorem{defin}[teo]{Definition}
\theoremstyle{remark}
\newtheorem{rmk}[teo]{Remark}
\newenvironment{dimo}
         {\textit{Proof of Theorem}}
     {\hspace*{\fill}\hspace*{\fill}\mbox{$\Box$}}

\theoremstyle{plain}                    
\newtheorem{teoa}{Theorem}[section]      
\newtheorem{propa}[teoa]{Proposition}
\newtheorem{cora}[teoa]{Corollary}       
\newtheorem{lema}[teoa]{Lemma}            
\theoremstyle{definition}               
\newtheorem{defina}[teoa]{Definition}
\theoremstyle{remark}
\newtheorem{rmka}[teoa]{Remark}

\numberwithin{equation}{subsection}

\newcommand{\oo}{\mathcal{O}}
\newcommand{\mt}{\mathcal}

\newcommand{\cod}{\text{cod}}
\newcommand{\jc}{\mathcal J ac_{d,g}^o}
\newcommand{\vc}{\mathcal V ec_{r,d,g}^o}
\newcommand{\jr}{\mathcal J_{d,g}^o}
\newcommand{\vr}{\mathcal V_{r,d,g}^o}

\newcommand{\lra}{\longrightarrow}

\newcommand{\Vc}{\mathcal V ec_{r,d,g}}
\newcommand{\Jc}{\mathcal J ac_{d,g}}
\newcommand{\Vr}{\mathcal V_{r,d,g}}
\newcommand{\Jr}{\mathcal J_{d,g}}

\newcommand{\Pic}{\text{Pic}}
\newcommand{\CVc}{\overline{\mathcal Vec}_{r,d,g}}
\newcommand{\CJc}{\overline{\mathcal Jac}_{d,g}}
\newcommand{\CVr}{\overline{\mathcal V}_{r,d,g}}
\newcommand{\CU}{\overline{U}_{r,d,g}}

\newcommand{\Aut}{\text{Aut}}
\newcommand{\Spec}{\text{Spec}}
\newcommand{\Spf}{\text{Spf}}
\newcommand{\Mg}{\mathcal M_g}
\newcommand{\CMg}{\overline{\mathcal{M}}_g}
\newcommand{\ch}{\text{ch}}
\newcommand{\Td}{\text{Td}}
\newcommand{\TF}{\mt TF_{r,d,g}^{ss}}
\newcommand{\Vl}{\mathcal Vec_{r,=\mathcal L,C}}

\newcommand{\CVtc}{\overline{\mathcal Vec}_{r,d,2}}
\newcommand{\CVtr}{\overline{\mathcal V}_{r,d,2}}
\newcommand{\Vtc}{\mathcal Vec_{r,d,2}}

\begin{document}

\title{The Picard group of the universal moduli space\\ of vector bundles on stable curves.}
\author{Roberto Fringuelli}
\maketitle

\begin{abstract}We construct the moduli stack of properly balanced vector bundles on semistable curves and we determine explicitly its Picard group. As a consequence, we obtain an explicit description of the Picard groups of the universal moduli stack of vector bundles on smooth curves and of the Schmitt's compactification over the stack of stable curves. We prove some results about the gerbe structure of the universal moduli stack over its rigidification by the natural action of the multiplicative group. In particular, we give necessary and sufficient conditions for the existence of Poincar\'e bundles over the universal curve of an open substack of the rigidification, generalizing a result of Mestrano-Ramanan.
\end{abstract}

\tableofcontents

\section*{Introduction.}

Let $\Vc^{(s)s}$ be the moduli stack of (semi)stable vector bundles of rank $r$ and degree $d$ on smooth curves of genus $g$ over an algebraically closed field $k$ of characteristic $0$. It turns out that the forgetful map $\Vc^{ss}\to\Mg$ is universally closed, i.e. it satisfies the existence part of the valuative criterion of properness. Unfortunately, if we enlarge the moduli problem, adding slope-semistable (with respect to the canonical polarization) vector bundles on stable curves, the morphism to the moduli stack $\CMg$ of stable curves is not universally closed anymore. There are two natural ways to make it universally closed. The first one is adding slope-semistable torsion free sheaves and this was done by Pandharipande in \cite{P96}. The disadvantage is that such stack, as Faltings has shown in \cite{Fa96}, is not regular if the rank is greater than one. The second approach, which is better for our purposes, is to consider vector bundles on semistable curves: see \cite{Gie84}, \cite{K}, \cite{NS} in the case of a fixed irreducible curve with one node, \cite{Cap94}, \cite{Mel09} in the rank one case over the entire moduli stack $\CMg$ or \cite{Sc}, \cite{T98} in the higher rank case over $\CMg$. The advantages are that such stacks are regular and the boundary has normal-crossing singularities. Unfortunately, for rank greater than one, we do not have an easy description of the objects at the boundary. We will overcome the problem by constructing a non quasi-compact smooth stack $\CVc$, parametrizing properly balanced vector bundles on semistable curves (see \S\ref{pbvb} for a precise definition). When $r=1$, it coincides with the Caporaso's compactification $\CJc$ of the universal Jacobian stack, constructed by Melo in \cite{Mel09}, based upon the results of Caporaso in \cite{Cap94}. Moreover, it contains some interesting open substacks, like:
\begin{itemize}
\item[-] The moduli stack $\Vc$ of (not necessarily semistable) vector bundles over smooth curves.
\item[-] The moduli stack $\CVc^{P(s)s}$ of vector bundles such that their push-forwards in the stable model of the curve are slope-(semi)stable torsion free sheaves.
\item[-] The moduli stack $\CVc^{H(s)s}$ of H-(semi)stable vector bundles constructed by Schmitt in \cite{Sc}.
\item[-] The moduli stack of Hilbert-semistable vector bundles (see \cite{T98}).
\end{itemize}

The main result of this paper is computing and giving explicit generators for the Picard groups of the moduli stacks $\Vc$ and $\CVc$ for rank greater than one, generalizing the results in rank one obtained by Melo-Viviani in \cite{MV}, based upon a result of Kouvidakis (see \cite{Kou91}). As a consequence, we will see that there exist natural isomorphisms of Picard groups between $\Vc^s$, $\Vc^{ss}$ and $\Vc$, between $\CVc^{Hss}$, $\CVc^{Pss}$ and $\CVc$ and between $\CVc^{Ps}$ and $\CVc^{Hs}$.\\ 

The motivation for this work comes from the study of the modular compactifications of the moduli stack $\Vc^{ss}$ and the moduli space $U_{r,d,g}$ of semistable vector bundles on smooth curves from the point of view of the log-minimal model program (LMMP). One would like to mimic the so called Hassett-Keel program for the moduli space $\overline{M}_g$ of stable curves, which aims at giving a modular interpretation to every step of the LMMP for $\overline{M}_g$. In other words, the goal is to construct compactifications of the universal moduli space of semistable vector bundles over each step of the minimal model program for $\overline{M}_g$. In the rank one case, the conjectural first two steps of the LMMP for the Caporaso's compactification $\overline{J}_{d,g}$ have been described by Bini-Felici-Melo-Viviani in \cite{BFMV}. From the stacky point of view, the first step (resp. the second step) is constructed as the compactified Jacobian over the Schubert's moduli stack $\CMg^{ps}$ of pseudo-stable curves (resp. over the Hyeon-Morrison's moduli stack $\CMg^{wp}$ of weakly-pseudo-stable curves). In higher rank, the conjectural first step of the LMMP for the Pandharipande's compactification $\widetilde{U}_{r,d,g}$ has been described by Grimes in \cite{Gr14}: using the torsion free approach, he constructs a compactification $\widetilde{U}^{ps}_{r,d,g}$ of the moduli space of slope-semistable vector bundles over $\overline{M}^{ps}_g$. 
In order to construct birational compact models for the Pandharipande compactification of $U_{r,d,g}$, it is useful to have an explicit description of its rational Picard group which naturally embeds into the rational Picard group of the moduli stack $\TF$ of slope-semistable torsion free sheaves over stable curves. Indeed our first idea was to study directly the Picard group of $\TF$ . For technical difficulties due to the fact that this stack is not smooth, we have preferred to study first $\CVc$, whose Picard group contains $\Pic(\TF)$, and we plan to give a description of $\TF$ in a subsequent paper.\vspace{0.2cm}

In Section \ref{cvc}, we introduce and study our main object: \emph{the universal moduli stack $\CVc$ of properly balanced vector bundles of rank $r$ and degree $d$ on semistable curves of arithmetic genus $g$}. We will show that it is an irreducible smooth Artin stack of dimension $(r^2+3)(g-1)$. The stacks of the above list are contained in $\CVc$ in the following way
\begin{equation}\label{chainopen}
\begin{array}{lllll}
\CVc^{Ps}\subset&\CVc^{Hs}\subset&\CVc^{Hss}\subset&\CVc^{Pss}\subset&\CVc\\
\cup & &\cup & &\cup\\
\Vc^s&\subset&\Vc^{ss}&\subset&\Vc.
\end{array}
\end{equation}
The stack $\CVc$ is endowed with a morphism $\overline{\phi}_{r,d}$ to the stack $\CMg$ which forgets the vector bundle and sends a curve to its stable model. Moreover, it has a structure of $\mathbb G_m$-stack, since the group $\mathbb G_m$ naturally injects into the automorphism group of every object as multiplication by scalars on the vector bundle. Therefore, $\CVc$ becomes a $\mathbb G_m$-gerbe over the $\mathbb G_m$-rigidification $\CVr:=\CVc\fatslash \mathbb G_m$. Let $\nu_{r,d}:\CVc\to\CVr$ be the rigidification morphism. Analogously, the open substacks in (\ref{chainopen}) are $\mathbb G_m$-gerbes over their rigidifications
\begin{equation}\label{chainopenrig}
\begin{array}{lllll}
\CVr^{Ps}\subset&\CVr^{Hs}\subset&\CVr^{Hss}\subset&\CVr^{Pss}\subset&\CVr\\
\cup & &\cup & &\cup\\
\Vr^s&\subset&\Vr^{ss}&\subset&\Vr.
\end{array}
\end{equation}
The inclusions (\ref{chainopen}) and (\ref{chainopenrig} give us the following commutative diagram of Picard groups:
\begin{equation}\label{compic}
\xymatrix @C=0.01em @R=0.7em{
& \Pic\left(\CVc\right)\ar@{-}[d]\ar@{->>}[rr] & & \Pic\left(\CVc^{Pss}\right)\ar@{-}[d]\ar@{->>}[rr] & & \Pic\left(\CVc^{Hs}\right)\ar@{->>}[dd]\\
\Pic\left(\CVr\right)\ar@{^(->}[ur]\ar@{->>}[rr]\ar@{->>}[dddd] &\ar@{->>}[ddd] & \Pic\left(\CVr^{Pss}\right)\ar@{^(->}[ur]\ar@{->>}[rr]\ar@{->>}[dd] &\ar@{->>}[d] & \Pic\left(\CVr^{Hs}\right)\ar@{^(->}[ur]\ar@{->>}[dd]\\
& & & \Pic\left(\CVc^{Hss}\right) \ar@{-}[d]\ar@{-}[r]&\ar@{->>}[r] & \Pic\left(\CVc^{Ps}\right)\ar@{->>}[dd]\\
 & & \Pic\left(\CVr^{Hss}\right)\ar@{^(->}[ur]\ar@{->>}[rr]\ar@{->>}[dd] & \ar@{->>}[d]& \Pic\left(\CVr^{Ps}\right)\ar@{^(->}[ur]\ar@{->>}[dd]\\
& \Pic\left(\Vc\right)\ar@{-}[r] &\ar@{->>}[r] & \Pic\left(\Vc^{ss}\right)\ar@{-}[r] &\ar@{->>}[r] & \Pic\left(\Vc^{s}\right)\\
\Pic\left(\Vr\right)\ar@{^(->}[ur]\ar@{->>}[rr] & & \Pic\left(\Vr^{ss}\right)\ar@{^(->}[ur]\ar@{->>}[rr] & & \Pic\left(\Vr^{s}\right)\ar@{^(->}[ur]
}
\end{equation}
where the diagonal maps are the inclusions induced by the rigidification morphisms, while the vertical and horizontal ones are the restriction morphisms, which are surjective because we are working with smooth stacks. We will prove that the Picard groups of diagram (\ref{compic}) are generated by the boundary line bundles and the tautological line bundles, which are defined in Section \ref{linechow}.\\
In the same section, we also describe the irreducible components of the boundary divisor $\widetilde\delta:=\CVc\backslash\Vc$. The boundary is the pull-back via the morphism $\overline{\phi}_{r,d}:\CVc\to\CMg$ of the boundary of $\CMg$. It is known that  $\CMg\backslash \Mg=\bigcup_{i=0}^{\lfloor g/2\rfloor}\delta_i$,  where $\delta_0$ is an irreducible divisor whose generic point is an irreducible curve with just one node and, for $i\neq 0$, $\delta_i$ is an irreducible divisor whose generic point is a stable curve with two irreducible smooth components of genus $i$ and $g-i$ meeting in one point. In Proposition \ref{boundary}, we will prove that $\widetilde\delta_i:=\overline{\phi}^*_{r,d}\left(\delta_i\right)$ is irreducible if $i=0$ and, otherwise, decomposes as $\bigcup_{j\in J_i} \widetilde\delta_i^j$, where $J_i$ is a set of integers depending on $i$ and $\widetilde\delta_i^j$ are irreducible divisors. Such $\widetilde\delta_i^j$ will be called \emph{boundary divisors}. For special values of $i$ and $j$, the corresponding boundary divisor will be called \emph{extremal boundary divisor}. The boundary divisors which are not extremal will be called \emph{non-extremal boundary divisors} (for a precise description see \S\ref{boundiv}). By smoothness of $\CVc$, the divisors $\{\widetilde\delta_i^j\}$ give us line bundles. We will call them \emph{boundary line bundles} and we will denote them with $\{\oo(\widetilde\delta_i^j)\}$. We will say that $\oo(\widetilde\delta_i^j)$ is a \emph{(non)-extremal boundary line bundle} if $\widetilde\delta_i^j$ is a (non)-extremal boundary divisor. 
The irreducible components of the boundary of $\CVr$ are the divisors $\nu_{r,d}(\widetilde\delta_i^j)$. The associated line bundles are called boundary line bundles of $\CVr$. We will denote by the the same symbols used for $\CVc$ the boundary divisors and the associated boundary line bundles on $\CVr$. \\
In \S\ref{tautbun} we define the \emph{tautological line bundles}. They are defined as determinant of cohomology and as Deligne pairing (see \S\ref{dcdp} for the definition and basic properties) of particular line bundles along the universal curve $\overline{\pi}:\overline{Vec}_{r,d,g,1}\to\CVc$. More precisely, they are
$$
\begin{array}{lcl}
K_{1,0,0} &:=&\langle \omega_{\overline{\pi}},\omega_{\overline{\pi}}\rangle,\\
K_{0,1,0}&:=&\langle \omega_{\overline{\pi}},\det\mt E\rangle,\\
K_{-1,2,0}&:=&\langle \det\mt E,\det\mt E\rangle,\\
\Lambda(m,n,l) & :=& d_{\overline{\pi}}(\omega_{\overline{\pi}}^m\otimes (\det \mt E)^n\otimes \mathcal E^l),
\end{array}
$$
where $\omega_{\overline{\pi}}$ is the relative dualizing sheaf for $\overline\pi$ and $\mt E$ is the universal vector bundle on $\overline{Vec}_{r,d,g,1}$. Following the same strategy as Melo-Viviani in \cite{MV}, based upon the work of Mumford in \cite{Mum83}, we apply Grothendieck-Riemann-Roch theorem to the morphism $\pi:\overline{Vec}_{r,d,g,1}\to\CVc$ in order to compute the relations among the tautological line bundles in the rational Picard group. In particular, in Theorem \ref{relations} we prove that all tautological line bundles can be expressed in the (rational) Picard group of $\CVc$ in terms of $\Lambda(1,0,0)$, $\Lambda(0,1,0)$, $\Lambda(1,1,0)$, $\Lambda(0,0,1)$ and the boundary line bundles.\vspace{0.2cm}

Finally, we can now state the main results of this paper. In Section \ref{robba}, we prove that all Picard groups on the diagram (\ref{compic}) are free and generated by the tautological line bundles and the boundary line bundles. More precisely, we have the following.

\begin{theoremalpha}\label{pic}Assume $g\geq 3$ and $r\geq 2$.
\begin{enumerate}[(i)]
\item The Picard groups of $\Vc$, $\Vc^{ss}$, $\Vc^s$ are freely generated by $\Lambda(1,0,0)$, $\Lambda(1,1,0)$, $\Lambda(0,1,0)$ and $\Lambda(0,0,1)$.
\item The Picard groups of $\CVc$, $\CVc^{Pss}$, $\CVc^{Hss}$ are freely generated by $\Lambda(1,0,0)$, $\Lambda(1,1,0)$, $\Lambda(0,1,0)$, $\Lambda(0,0,1)$ and the boundary line bundles.
\item The Picard groups of $\CVc^{Ps}$, $\CVc^{Hs}$ are freely generated by $\Lambda(1,0,0)$, $\Lambda(1,1,0)$, $\Lambda(0,1,0)$, $\Lambda(0,0,1)$ and the non-extremal boundary line bundles.
\end{enumerate}
\end{theoremalpha}

Let $v_{r,d,g}$ and $n_{r,d}$ be the numbers defined in the Notations \ref{notations} below. Let $\alpha$ and $\beta$ be (not necessarily unique) integers such that $\alpha(d+1-g)+\beta(d+g-1)=-\frac{1}{n_{r,d}}\cdot\frac{v_{1,d,g}}{v_{r,d,g}}(d+r(1-g))$. Set
$$\Xi:=\Lambda(0,1,0)^{\frac{d+g-1}{v_{1,d,g}}}\otimes\Lambda(1,1,0)^{-\frac{d-g+1}{v_{1,d,g}}},\quad
\Theta:=\Lambda(0,0,1)^{\frac{r}{n_{r,d}}\cdot \frac{v_{1,d,g}}{v_{r,d,g}}}\otimes \Lambda(0,1,0)^{\alpha}\otimes \Lambda (1,1,0)^{\beta}.$$

In the same section, we will show 
\begin{theoremalpha}\label{picred}Assume $g\geq 3$ and $r\geq2$.
\begin{enumerate}[(i)]
\item The Picard groups of $\Vr$, $\Vr^{ss}$, $\Vr^{s}$ are freely generated by $\Lambda(1,0,0)$, $\Xi$ and $\Theta$.
\item The Picard groups of $\CVr$, $\CVr^{Pss}$, $\CVr^{Hss}$ are freely generated by $\Lambda(1,0,0)$, $\Xi$, $\Theta$ and the boundary line bundles.
\item The Picard groups of $\CVr^{Ps}$ and $\CVr^{Hs}$ are freely generated by $\Lambda(1,0,0)$, $\Xi$, $\Theta$ and the non-extremal boundary line bundles.
\end{enumerate}
\end{theoremalpha}

The above theorems are a collection of the major results of this work. For the proof of Theorem \ref{pic}, resp. \ref{picred}, see page \pageref{proofpic}, resp. \pageref{proofpicred}. If we remove the word ''freely'', they hold also in the genus two case for some of the open substacks in the assertions. This will be shown in appendix \ref{genus2}, together with an explicit description of the relations among the generators.\vspace{0.2cm}

We sketch the strategy of the proofs of the Theorems \ref{pic} and \ref{picred}. First, in \S\ref{indipendece}, we will prove that the boundary line bundles are linearly independent. Since the stack $\CVc$ is smooth and contains quasi-compact open substacks which are "large enough" and admit a presentation as quotient stacks, we have a natural exact sequence of groups
\begin{equation}\label{quasiex}
\bigoplus_{i=0,\ldots,\lfloor g/2\rfloor}\oplus_{j\in J_i}\langle\oo(\widetilde\delta_i^j)\rangle\longrightarrow \Pic(\CVc)\rightarrow \Pic(\Vc)\longrightarrow 0.
\end{equation}
In Theorem \ref{indbou}, we show that this sequence is also left exact. The strategy that we will use is the same as the one of Arbarello-Cornalba for $\CMg$ in \cite{AC87} and the generalization for $\CJc$ done by Melo-Viviani in \cite{MV}. More precisely, we will construct morphisms $B\to\CVc$ from irreducible smooth projective curves $B$ and we show that the intersection matrix between these test curves and the boundary line bundles on $\CVc$ is non-degenerate.\\
Furthermore, since the homomorphism of Picard groups induced by the rigidification morphism $\nu_{r,d}:\CVc\to\CVr$ is injective and it sends the boundary line bundles  of $\CVr$ to the boundary line bundles of $\CVc$, we see that also the boundary line bundles in the rigidification $\CVr$ are linearly independent (see Corollary \ref{boundrig}). In other words we have an exact sequence:
\begin{equation}\label{exrig}
0\longrightarrow\bigoplus_{i=0,\ldots,\lfloor g/2\rfloor}\oplus_{j\in J_i}\langle\oo(\widetilde\delta_i^j)\rangle\longrightarrow \Pic(\CVr)\longrightarrow \Pic(\Vr)\longrightarrow 0.
\end{equation}
We will show that the sequence (\ref{quasiex}), (resp. (\ref{exrig})), remains exact if we replace the middle term with the Picard group of $\CVc^{Pss}$ (resp. $\CVr^{Pss}$) or $\CVc^{Hss}$ (resp. $\CVr^{Hss}$). This reduces the proof of Theorem \ref{pic}(ii) (resp. Theorem \ref{picred}(ii)) to proving the Theorem \ref{pic}(i) (resp. Theorem \ref{picred}(i)). While for the stacks $\CVc^{Ps}$ and $\CVc^{Hs}$ (resp. $\CVr^{Ps}$ and $\CVr^{Hs}$) the sequence (\ref{quasiex}) (resp. (\ref{exrig})) is exact, if we remove the extremal boundary line bundles. This reduces the proof of Theorem \ref{pic}(iii)  (resp. of Theorem \ref{picred}(iii)) to proving the Theorem \ref{pic}(i) (resp. the Theorem \ref{picred}(i)).\vspace{0.2cm}

The stack $\Vc$ admits a natural map $det$ to the universal Jacobian stack $\Jc$, which sends a vector bundle to its determinant line bundle. The morphism is smooth and the fiber over a polarized curve $(C,\mt L)$ is the irreducible moduli stack $\Vl$ of pairs $(\mt E,\varphi)$, where $\mt E$ is a vector bundle on $C$ and $\varphi$ is an isomorphism between $\det\mt E$ and $\mt L$ (for more details see \S\ref{fibre}). Hoffmann in \cite{H} showed that the pull-back to $\Vl$ of the tautological line bundle $\Lambda(0,0,1)$ on $\CVc$ freely generates $\Pic\left(\Vl\right)$ (see Theorem \ref{fibers}). 
Moreover, as Melo-Viviani have shown in \cite{MV}, the tautological line bundles $\Lambda(1,0,0)$, $\Lambda(1,1,0)$, $\Lambda(0,1,0)$ freely generate the Picard group of $\Jc$ (see Theorem \ref{picjac}). Since the Picard groups of $\Vc$, $\Vc^{ss}$, $\Vc^{s}$ are isomorphic (see Lemma \ref{redsemistable}), Theorem \ref{pic}(i) (and so Theorem \ref{pic}) is equivalent to prove that we have an exact sequence of groups
\begin{equation}\label{theoAi}
0\lra \Pic(\Jc)\lra \Pic(\Vc^{ss})\lra \Pic(\mathcal Vec^{ss}_{=\mathcal L,C})\lra 0,
\end{equation}
where the first map is the pull-back via the determinant morphism and the second one is the restriction along a fixed geometric fiber. We will prove this in \S\ref{liscio}. If we were working with schemes, this would follow from the so-called seesaw principle: if we have a proper flat morphism of varieties with integral geometric fibers then a line bundle on the source is the pull-back of a line bundle on the target if and only if it is trivial along any geometric fiber. We generalize this principle to stacks admitting a proper good moduli space (see Appendix \ref{App}) and we will use this fact to prove the exactness of (\ref{theoAi}).\\
In \S\ref{comparing}, we use the Leray spectral sequence for the lisse-\'etale sheaf $\mathbb{G}_m$ with respect to the rigidification morphism $\nu_{r,d}:\Vc\lra\Vr$, in order to conclude the proof of Theorem \ref{picred}. Moreover, we will obtain, as a consequence, some interesting results about the properties of $\CVr$ (see Proposition \ref{poincare}). In particular we will show that the rigidified universal curve $\mt V_{r,d,g,1}\to\Vr$ admits a universal vector bundle over an open substack of $\Vr$ if and only if the integers $d+r(1-g)$, $r(d+1-g)$ and $r(2g-2)$ are coprime, generalizing the result of Mestrano-Ramanan (\cite[Corollary 2.9]{MR85}) in the rank one case.\vspace{0.2cm}

The paper is organized in the following way. In Section \ref{cvc}, we define and study the moduli stack $\CVc$ of properly balanced vector bundles on semistable curves. In \S\ref{pbvb}, we give the definition of a properly balanced vector bundle on a semistable curve and we study the properties. In \S\ref{univ} we prove that the moduli stack $\CVc$ is algebraic. In \S\ref{prop}, we list some properties of our stacks and we introduce the rigidified moduli stack $\CVr$. We will use the deformation theory of vector bundles on nodal curves to study the local structure of $\CVc$ (see \S\ref{locstr}).In \S\ref{Sc}, we focus on the existence of good moduli spaces for an open substack of $\CVc$, following the Schmitt's construction. In Section \ref{linechow}, we recall some basic facts about the Picard group of a stack. In \S\ref{boh}, we explain the relations between the Picard group and the Chow group of divisors of stacks. We illustrate how to construct line bundles on moduli stacks using the determinant of cohomology and the Deligne pairing (see \S\ref{dcdp}). Then we recall the computation of the Picard group of the stack $\CMg$, resp. $\Jc$, resp. $\Vl$ (see \S\ref{cmg}, resp. \S\ref{jc}, resp. \S\ref{fibre}). In \S\ref{boundiv} we describe the boundary divisors of $\CVc$, while in \S\ref{tautbun} we define the tautological line bundles and we study the relations among them. Finally, in Section \ref{robba}, as explained before, we prove Theorems \ref{pic} and \ref{picred}. The genus two case will be treated separately in the Appendix \ref{genus2}. In Appendix \ref{App}, we recall the definition of a good moduli space for a stack and we develop, following the strategy adopted by Brochard in \cite[Appendix]{Br2}, a base change cohomology theory for stacks admitting a proper good moduli space.\\

\textbf{Acknowledgements:} The author would like to thank his advisor Filippo Viviani, for introducing the author to the problem, for his several suggests and comments without which this work would not have been possible. I thank the anonymous referees for carefully reading the paper and the several comments and corrections.

\subsection*{Notations.}
\begin{notations}\label{notations}
Let $r,g,d$ integers such that $g\geq 2$, $r\geq 1$. Denote by $g$ the arithmetic genus of the curves, $d$ the degree of the vector bundles and $r$ their rank. Given two integers $s$, $t$ we denote by $(s,t)$ the greatest common divisor of $s$ and $t$. We will set
$$
n_{r,d}:=(r,d),\; v_{r,d,g}:=\left(\frac{d}{n_{r,d}}+\frac{r}{n_{r,d}}(1-g),\,d+1-g,\,2g-2\right),\; k_{r,d,g}:=\frac{2g-2}{\left(2g-2,d+r(1-g)\right)}.
$$
Given a rational number $q$, we denote by $\lfloor q\rfloor$ the greatest integer such that $\lfloor q\rfloor\leq q$ and by $\lceil q\rceil$ the lowest integer such that $q\leq\lceil q\rceil$.
\end{notations}
\begin{notations}
We will work with the category $Sch/k$ of schemes over an algebraically closed field $k$ of characteristic $0$. When we say commutative, resp. cartesian, diagram of stacks we will intend in the $2$-categorical sense. We will implicitly assume that all the sheaves are sheaves for the site lisse-\'etale, or equivalently for the site lisse-lisse champ\^etre (see \cite[Appendix A.1]{Br1}).

The choice of the characteristic is due to the fact that the explicit computation of the integral Picard group of $\CMg$ is known to be true only in characteristic $0$ (if $g\geq 3$). Also the computation of $\Jc$ in \cite{MV} is unknown in positive characteristic, because it is based upon a result of Kouvidakis in \cite{Kou91}, which is proved over the complex numbers. If these two results could be extended to arbitrary characteristics, then also our results would automatically extend.
\end{notations}
\section{The universal moduli space $\CVc$.}\label{cvc}
Here we introduce the moduli stack of properly balanced vector bundles on semistable curves. Before giving the definition, we need to define and study the objects which are going to be parametrized.

\begin{defin} A \emph{family of curves over a scheme $S$} is a proper and flat morphism $C\to S$ of finite presentation and of relative dimension $\leq 1$. A \emph{stable} (resp. \emph{semistable}) \emph{curve $C$ over $k$} is a family of curves over $\Spec k$ such that it is connected, 1-dimensional, it has at-worst nodal singularities and any rational smooth component intersects the rest of the curve in at least $3$ (resp. $2$) points.  A \emph{family of (semi)stable curves over a scheme $S$} is a family of curves whose geometric fibers are (semi)stable curves. A \emph{vector bundle on a family of (semistable) curves $C\rightarrow S$} is a coherent $S$-flat sheaf on $C$, which is a vector bundle on any geometric fiber.
\end{defin}

To any family $C\rightarrow S$ of semistable curves, we can associate a new family $C^{st}\rightarrow S$ of stable curves and an $S$-morphism $\pi:C\rightarrow C^{st}$, which, for any geometric fiber over $S$, is the stabilization morphism, i.e. it contracts the destabilizing chains. We can construct this taking the $S$-morphism $\pi: C\rightarrow \mathbb{P}(\omega_{C/S}^{\otimes 3})$ associated to the relative dualizing sheaf of $C\rightarrow S$ and calling $C^{st}$ the image of $C$ through $\pi$.
\begin{defin}
Let $C$ be a semistable curve over $k$ and $Z$ be a non-trivial subcurve. We set $Z^c:=\overline{C\backslash Z}$ and $k_Z:=|Z\cap Z^c|$. Let $\mathcal E$ be a vector bundle over $C$. If $C_1,\ldots, C_n$ are the irreducible components of $C$, we call \emph{multidegree} of $\mathcal E$ the $n$-tuple $(\deg\mt E_{C_1},\ldots,\deg\mt E_{C_n})$ and \emph{total degree} of $\mt E$ the integer $d:=\sum \deg\mt E_{C_i}$.
\end{defin}
With abuse of notation, we will write $\omega_Z:=\deg(\omega_C|_{Z})=2g_Z-2+k_Z$, where $\omega_C$ is the dualizing sheaf and $g_Z:=1-\chi(\oo_Z)$. If $\mt E$ is a vector bundle over a family of semistable curves $C\to S$, we will set $\mathcal E(n):=\mathcal E\otimes\omega^n_{C/S}$. By the projection formula we have
$$
R^i\pi_*\mathcal E(n):=R^i\pi_*(\mathcal E\otimes\omega^n_{C/S})\cong R^i\pi_*(\mathcal E)\otimes\omega^n_{C^{st}/S},
$$
where $\pi$ is the stabilization morphism.

\subsection{Properly balanced vector bundles.}\label{pbvb}

We recall some definitions and results from \cite{K}, \cite{Sc} and \cite{NS}.
\begin{defin} A \emph{chain of rational curves} (or \emph{rational chain}) $R$ is a connected projective nodal curve over $k$ whose associated graph is a path and whose irreducible components are smooth and rational. The $\emph{length}$ of $R$ is the number of irreducible components.
\end{defin}

Let $R_1,\ldots,R_m$ be the irreducible components of a chain of rational curves $R$, labelled in the following way: $R_i\cap R_j\neq \emptyset$ if and only if $|i-j|\leq 1$. For $1\leq i\leq m-1$ let $x_i:=R_i\cap R_{i+1}$ be the nodal points and $x_0\in R_1$, $x_m\in R_m$ closed points different from $x_1$ and $x_{m-1}$. Let $\mathcal E$ be a vector bundle on $R$ of rank $r$. By \cite[Proposition 3.1]{T91b}, any vector bundle $\mt E$ over a chain of rational curves $R$ decomposes in the following way
\begin{equation}\label{decomp}
\mt E\cong \bigoplus_{j=1}^r\mt L_j,\text{ where } \mt L_j \text{ is a line bundle for any }j=1,\ldots,r.
\end{equation}
Using these notations we can give the following definitions.

\begin{defin} Let $\mt E$ be a vector bundle of rank $r$ on a rational chain $R$ of length $m$. If there exists a decomposition as in (\ref{decomp}) such that
\begin{itemize}
\item $\mt E$ is \emph{positive} if $\deg \mt L_{j|R_i}\geq 0$ for any $j\in\{1,\ldots,r\}$ and $i\in\{1,\ldots,m\}$,
\item $\mt E$ is \emph{strictly positive} if $\mt E$ is positive and for any $i\in\{1,\ldots,m\}$ there exists $j\in\{1,\ldots,r\}$ such that $\deg \mt L_{j|R_i}> 0$,
\item $\mt E$ is \emph{stricly standard} if $\mt E$ is strictly positive and $\deg \mt L_{j|R_i}\leq 1$ for any $j\in\{1,\ldots,r\}$ and $i\in\{1,\ldots,m\}$.
\end{itemize}
\end{defin}

\begin{defin}Let $R$ be a chain of rational curves over $k$ and $R_1,\ldots,R_m$ its irreducible components. A strictly standard vector bundle $\mt E$ of rank $r$ over $R$ is called \emph{admissible}, if one of the following equivalent conditions (see \cite[Lemma 2]{NS} or \cite[Lemma 3.3]{K}) holds:
\begin{itemize}
\item $h^0(R,\mathcal E(-x_0))=\sum \deg\mathcal E_{R_i}=\deg \mt E$,
\item $H^0(R,\mathcal E(-x_0-x_m))=0$,
\item $\mt E=\bigoplus_{j=1}^r\mt L_j$, where $\mt L_j$ is a line bundle of total degree $0$ or $1$ for $j=1,\ldots,r$.
\end{itemize}
\end{defin}

\begin{defin} Let $C$ be a semistable curve over $k$. A \emph{destabilizing chain} of $C$ is a subcurve of $C$, which is a rational chain intersecting the rest of the curve in exactly two points. The subcurve of all destabilizing chains will be called \emph{exceptional curve} and will be denoted with $C_{\text{exc}}$ and we set $\widetilde C:=C_{\text{exc}}^c$. A connected component of $C_{\text{exc}}$ will be called \emph{maximal destabilizing chain}.
\end{defin}

\begin{defin}Let $C$ be a semistable curve and $\mathcal E$ be a vector bundle of rank $r$ over $C$. $\mathcal E$ is \emph{(strictly) positive}, resp. \emph{strictly standard}, resp. \emph{admissible vector bundle} if the restriction to any destabilizing chain is (strictly) positive, resp. strictly standard, resp. admissible.
Let $C\rightarrow S$ be a family of semistable curves with a vector bundle $\mathcal E$ of rank $r$. $\mathcal E$ is called \emph{(strictly) positive}, resp. \emph{strictly standard}, resp. \emph{admissible vector bundle} if it is (strictly) positive, resp. strictly standard, resp. admissible for any geometric fiber. 
\end{defin}

\begin{rmk}\label{length} Let $(C,\mt E)$ be a semistable curve with a vector bundle. We remark that if $\mt E$ satisfies one of the conditions above along one chain $R$, then it satisfies it for any subchain $R'\subset R$. In particular, it is enough to check the condition over the maximal destabilizing chains.

These properties are related by a sequence of implications: $\mt E$ is admissible $\Rightarrow$ $\mt E$ is strictly standard $\Rightarrow$ $\mt E$ is strictly positive $\Rightarrow$ $\mt E$ is positive.
 
Furthermore, if $\mt E$ is admissible of rank $r$, then any destabilizing chain must be of length $\leq r$. Indeed, since $\mt E$ is strictly positive, its restriction to each destabilizing chain of length $1$ must be of positive degree. In particular, if $R$ is any destabilizing chain, then
$$
\text{length}R\leq \deg\mt E_R=\sum^r_{i=1}\deg\mt L_j,\text{ where }\mt E=\bigoplus_{j=1}^r\mt L_j.
$$
By definition of admissibility, we have $0\leq \deg\mt L_j\leq 1$. So $\text{length}R\leq r$.
\end{rmk}

The role of positivity is summarized in the next two propositions.

\begin{prop}\label{stabil}\cite[Prop 1.3.1(ii)]{Sc} Let $\pi:C'\rightarrow C$ be a morphism between semistable curves which contracts only some destabilizing chains. Let $\mathcal E$ be a vector bundle on $C'$ positive on the contracted chains. Then $R^i\pi_*(\mathcal E)=0$ for $i>0$. In particular $H^j(C',\mathcal E)=H^j(C,\pi_*\mathcal E)$ for all $j$.
\end{prop}

\begin{prop}\label{NSlemma4} Let $C\rightarrow S$ be a family of semistable curves, $S$ locally noetherian scheme and consider the stabilization morphism
$$
\xymatrix{
C\ar[rd]\ar[rr]^{\pi} & & C^{st}\ar[ld]\\
& S &}
$$
Suppose that $\mathcal E$ is a positive vector bundle on $C\rightarrow S$ and for any point $s\in S$ consider the induced morphism $\pi_{s}:C_s\rightarrow C^{st}_s$. Then
$
\pi_*(\mathcal E)_{C^{st}_s}=\pi_{s*}(\mathcal E_{C_s}).
$
Moreover, $\pi_*\mathcal E$ is $S$-flat.
\end{prop}

\begin{proof}It follows from \cite[Lemma 4]{NS} and \cite[Remark 1.3.6]{Sc}.
\end{proof}

The next result gives us a useful criterion to check if a vector bundle is strictly positive or not.



\begin{prop}\label{ampleness} Let $\mathcal E$ be a vector bundle over a semistable curve $C$. $\mathcal E$ is strictly positive if and only if there exists $n$ big enough such that the vector bundle $\mt E(n)$ is generated by global sections and the induced morphism in the Grassmannian $C\to Gr(H^0(C,\mathcal E(n)),r)$ is a closed embedding.
\end{prop}

\begin{proof}The implication $\Rightarrow$ is given by \cite[Proposition 1.3.3 and Remark 1.3.4]{Sc}. The other one comes from \cite[Remark 1]{NS}.	
\end{proof}

\begin{rmk} Let $\mathcal F$ be a torsion free sheaf over a nodal curve $C$. By \cite[Huitieme Partie, Proposition 2 and 3]{Se82}, the stalk of $\mt F$ over a nodal point $x$ is of the form
\begin{itemize}
\item $\oo_{C,x}^{r_0}\oplus\oo_{C_1,x}^{r_1}\oplus\oo_{C_2,x}^{r_2}$, if $x$ is a meeting point of two irreducible curves $C_1$ and $C_2$.
\item $\oo_{C,x}^{r-a}\oplus m_{C,x}^{a}$, if $x$ is a nodal point belonging to a unique irreducible component.
\end{itemize}
If $\mathcal F$ has uniform rank $r$ (i.e. it has rank $r$ on any irreducible component of $C$), we can always write the stalk at $x$ in the form $\oo_{C,x}^{r-a}\oplus m_{C,x}^{a}$ for some $a$. In this case we will say that $\mathcal F$ is \emph{of type $a$ at $x$}.
\end{rmk}

Now we are going to describe the properties of an admissible vector bundle. The following proposition (and its proof) is a generalization of \cite[Proposition 5]{NS}.

\begin{prop}\label{prop5NS}Let $\mathcal E$ be a vector bundle of rank $r$ over a semistable curve $C$, and $\pi:C\rightarrow C^{st}$ the stabilization morphism, then:
\begin{enumerate}[(i)]
\item $\mathcal E$ is admissible if and only if $\mt E$ is strictly positive and $\pi_*\mt E$ is torsion free.
\item Let $R$ be a maximal destabilizing chain and $x:=\pi(R)$. If $\mathcal E$ is admissible then $\pi_*\mathcal E$ is of type $\deg\mathcal E_R$ at $x$.
\end{enumerate}
\end{prop}

\begin{proof}
Part (i). Let $\widetilde{C}$ be the subcurve of $C$ complementary to the exceptional one. Consider the exact sequence: $$0\lra\mt I_{\widetilde C}\mathcal E\lra \mathcal E\lra \mathcal E_{\widetilde C}\lra 0.$$
Observe that $\mt I_{\widetilde C}\mathcal E$ is the sheaf of sections $\mt E$ vanishing at $\widetilde C$. So, the restriction to $C_{\text{exc}}$ induces an injective morphism of sheaves
$$
\mt I_{\widetilde C}\mathcal E\hookrightarrow \mt E_{C_{\text{exc}}},
$$
which factors through $\mt I_{D}\mathcal E_{C_{\text{exc}}}$, the sheaf of sections of $\mt E_{C_{\text{exc}}}$ vanishing at $D:=C_{\text{exc}}\cap \widetilde C$ with its reduced scheme structure. Since, $\chi(\mt I_{\widetilde C}\mathcal E)=\chi(\mt I_{D}\mathcal E_{C_{\text{exc}}})$, then $\mt I_{\widetilde C}\mathcal E=\mt I_{D}\mathcal E_{C_{\text{exc}}}$. We have:
$$0\lra \pi_*(\mt I_D\mathcal E_{C_{\text{exc}}}) \lra \pi_*\mathcal E\lra \pi_*(\mathcal E_{\widetilde C}).$$
Now $\pi_*(\mathcal E_{\widetilde C})$ is a torsion-free sheaf and $\pi_*(\mt I_D\mathcal E_{C_{\text{exc}}})$ is a torsion sheaf, because its support is $\pi(D)$. So, $\pi_*\mathcal E$ is torsion free if and only if $\pi_*(\mt I_D\mathcal E_{C_{\text{exc}}})=0$. Let $R$ be a maximal destabilizing chain, which intersects the rest of the curve in $p$ and $q$ and $x:=\pi(R)$. By definition the stalk of the sheaf $\pi_*(\mt I_D\mathcal E_{C_{\text{exc}}})$ at $x$ is the $k$-vector space $H^0(R,\mt I_{p,q}\mathcal E_R)=H^0(R,\mathcal E_R(-p-q))$. So, we have that $\pi_*\mathcal E$ is torsion free if and only if for any destabilizing chain $R$ if a global section $s$ of $\mathcal E_R$ vanishes on $R\cap R^c$ then $s\equiv 0$. In particular, if $\mathcal E$ is admissible then $\pi_*\mathcal E$ is torsion free and $\mathcal E$ is strictly positive.

Conversely, assume that $\pi_*\mathcal E$ is torsion free and $\mathcal E$ is strictly positive. If $\mt E$ were not strictly standard then there would exist a destabilizing chain $R=\mathbb P^1$ of length one such that $\mt E_R$ contains $\oo_{\mathbb P^1}(2)$ as direct summand. In particular, if $R'$ is the maximal destabilizing chain containing $R=\mathbb P^1$, then $H^0(R',\mt I_{p,q}\mathcal E_{R'})\neq 0$. So  $\pi_*\mt E$ cannot be torsion free at $\pi(R')$, giving a contradiction. In other words, if $\pi_*\mathcal E$ is torsion free and $\mathcal E$ is strictly positive then $\mt E$ is strictly standard. By the above considerations, the assertion follows.\\
Part (ii). Let $R$ be a maximal destabilizing chain.  By hypothesis and part (i), $\pi_*\mt E$ is torsion free and we have an exact sequence:
$$
0\lra\pi_*\mt E\lra \pi_*(\mt E_{\widetilde C})\lra R^1\pi_*(\mt I_D\mt E_{C_{\text{exc}}})\lra 0.
$$
The sequence is right exact by Proposition \ref{stabil}. Using the notation of part (i), we have that the stalk of the sheaf $R^1\pi_*(\mt I_D\mt E_{C_{\text{exc}}})$ at $x$ is the $k$-vector space $H^1(R,\mt E_{R}(-p-q))$.
If $\deg(\mt E_R)=r$, then $\chi(R,\mt E_R(-p-q))=0$. Since $\mt E$ is admissible, we have $H^0(R,\mt E_R(-p-q))=H^1(R,\mt E_R(-p-q))=0$. Thus, $\pi_*\mt E$ is isomorphic to $\pi_*(\mt E_{\widetilde C})$ locally at $x$. The assertion follows by the fact that $ \pi_*(\mt E_{\widetilde C})$ is a torsion free sheaf of type $\deg(\mt E_R)=r$ at $x$. Suppose that $\deg(\mt E_R)=r-s<r$. Then we must have that $\mt E_R=\oo_R^s\oplus\mt F$. Using the sequence
$$0\lra \mathcal E\lra \mathcal E_{R^c}\oplus \mathcal E_R\lra \mt E_{\{p\}}\oplus \mt E_{\{q\}}\lra 0$$
we can find a neighbourhood $U$ of $x$ in $C^{st}$ such that $\mt E_{\pi^{-1}(U)}=\oo_{\pi^{-1}(U)}^s\oplus\mt E'$ reducing to the case $\deg(\mt E_R)=r$.
\end{proof}

The proposition above has some consequences, which will be useful later. The following result is a generalization of \cite[Remark 4]{NS}.

\begin{cor}\label{NSrmk}Let $C$, resp. $C'$, be a semistable curve with an admissible vector bundle $\mt E$, resp. $\mt E'$. Let $\pi:C\rightarrow C^{st}$, resp. $\pi: C'\to C'^{st}$, be the stabilization morphism of $C$, resp. $C'$.

Suppose there exists an isomorphism of curves $\psi:C^{st}\cong C'^{st}$ and an isomorphism $\pi_*\mt E\cong \psi^*\pi'_*\mt E'$ of sheaves  on $C^{st}$. Then there exists an isomorphism of vector bundles $\mt E_{C^c_{\text{exc}}}\cong \widetilde\psi^*\mt E'_{C'^c_{\text{exc}}}$ on $ C^c_{\text{exc}}$, where $\widetilde \psi: C^c_{\text{exc}}\cong C'^c_{\text{exc}}$ is the isomorphism of curves induced by $\psi$.
\end{cor}

\begin{proof}Let $R$ be a subcurve composed only by maximal destabilizing chains. We set $\widetilde C:=R^c$, $D$ the reduced subscheme $R\cap \widetilde C$, $D^{st}$ the reduced subscheme $\pi(D)\subset C^{st}$. Consider the following exact sequence
$$
0\lra \mt I_{R}\mt E\lra \mt E\lra \mt E_R\lra 0.
$$
With the same argument used in the proof of Proposition \ref{prop5NS}, we can identify $\mt I_R\mt E$ with $\mt I_{D}\mt E_{\widetilde C}$. Applying the left exact functor $\pi_*$, we have
$$
0\lra \pi_*\left(\mt I_{D}\mt E_{\widetilde C}\right)\lra \pi_*(\mt E)\lra \pi_*(\mt E_R)\lra 0.
$$
The sequence is right exact because $\mt E$ is positive. Moreover, $\pi_*(\mt E_R)$ is supported at $D^{st}$ and annihilated by $\mt I_{D^{st}}$. By Proposition \ref{prop5NS}(ii), the morphism $\pi_*(\mt E)\lra \pi_*(\mt E_R)$ induces an isomorphism of vector spaces at the restriction to $D^{st}$. This means that $\pi_*\left(\mt I_D\mt E_{\widetilde C}\right)=\mt I_{D^{st}}\left(\pi_*\mt E\right)$.

For the rest of the proof, $R$, resp. $R'$, will be the union of all destabilizing chains $C_{exc}$, resp. $C'_{\text{exc}}$. Suppose that $(C,\mt E)$ and $(C',\mt E')$, such that there exist $\psi:C^{st}\rightarrow C'^{st}$ and $\phi:\pi_*\mt E\cong \psi^*\pi'_*\mt E'$. By the above observations, we have
$$
\pi_*(\mt I_D\mt E_{\widetilde C})\cong \psi^*\pi'_*(\mt I_{D'}\mt E'_{\widetilde C'}).
$$
We remark that $\widetilde C$ and $\widetilde C'$ are isomorphic and $\psi$ induces an isomorphism $\widetilde\psi$ between them, such that
$$
\psi^*\pi'_*(\mt I_{D'}\mt E'_{\widetilde C'}).\cong \pi_*(\mt I_D\widetilde\psi^*\mt E'_{\widetilde C'}).
$$
Adapting the proof of \cite[Remark 4(ii)]{NS} to our more general case, we can show that, 
$$
\pi_*(\mt I_D\mt E_{\widetilde C})\cong \pi_*(\mt I_D\widetilde\psi^*\mt E'_{\widetilde C'})\Longleftrightarrow \mt I_D\mt E_{\widetilde C}\cong \mt I_D\widetilde\psi^*\mt E'_{\widetilde C'}.
$$
Twisting by $\mt I_D^{-1}$, we have the assertion.
\end{proof}

\begin{defin}Let $\mathcal E$ be a vector bundle of rank $r$ and degree $d$ on a semistable curve $C$. $\mathcal E$ is \emph{balanced} if for any subcurve $Z\subset C$ it satisfies the \emph{basic inequality}:
$$
\left|\deg \mt E_Z-d\frac{\omega_Z}{\omega_C}\right|\leq r\frac{k_Z}{2}.
$$
$\mathcal E$ is \emph{properly balanced} if it is balanced and admissible. If $C\rightarrow S$ is a family of semistable curves and $\mt E$ is a vector bundle of rank $r$ and degree $d$ for this family, we will call it \emph{(properly) balanced} if it is (properly) balanced on any geometric fiber.
\end{defin}
There are several equivalent definitions of balanced vector bundle. We list some which will be useful later. 
\begin{lem}\label{balcon} Let $\mt E$ be a vector bundle of rank $r$ and degree $d$ over a semistable curve $C$. The following conditions are equivalent:
\begin{enumerate}[(i)]
\item $\mt E$ is balanced,
\item for any subcurve $Z\subset C$, we have
\begin{equation}\label{second}\omega_C\cdot \chi(\mt F_Z)\leq \omega_Z \cdot \chi(\mt E),\end{equation}
where $\mt F_Z$ is the subsheaf of $\mt E_Z$ of sections vanishing on $Z\cap Z^c$. 
\item for any subcurve $Z\subset C$, we have \begin{equation}\label{third}\chi(\mathcal G_Z)\geq 0,\end{equation}
where $\mt G$ is the vector bundle $\left(\det\mathcal E\right)^{\otimes 2-2g}\otimes\omega_{C}^{\otimes d+r(g-1)}\oplus\oo_C^{\oplus r(2g-2)-1}$.\end{enumerate}
Moreover, it is enough to check the above (equivalent) conditions for those subcurves $Z\subset C$ such that $Z$ and $Z^c$ are connected.
\end{lem}

\begin{proof}A direct computation shows that $\mt E$ satisfies the basic inequality for a subcurve $Z$ if and only if
$$
\deg \mt E_Z-d\frac{\omega_Z}{\omega_C}\leq r\frac{k_Z}{2}\text{ and } \deg \mt E_{Z^c}-d\frac{\omega_{Z^c}}{\omega_C}\leq r\frac{k_Z}{2}.
$$

$(i)\Leftrightarrow (ii)$. Let $Z\subset C$ be a (not necessarily connected) subcurve. We have the following exact sequence
$$
0\to\mt F_Z\to \mt E_Z\to \mt E_D\to 0,
$$
where $D:=Z\cap Z^c$, with its reduced scheme structure. In particular, the inequality (\ref{second}) becomes
$$
\omega_C\left(\deg \mt E_Z-\frac{r}{2}(k_Z+\omega_Z)\right)=\omega_C\left(\deg\mt E_Z+r(1-g_Z)-r\cdot k_Z\right)\leq \omega_Z\left(d+r(1-g)\right)=\omega_Z(d-\frac{r}{2}\omega_C).
$$
The first equality comes from $\omega_Z=2g_Z-2+k_Z$. If we divide both sides by the positive integer $\omega_C$, we obtain $$
\deg \mt E_Z-d\frac{\omega_Z}{\omega_C}\leq r\frac{k_Z}{2}.$$
So, $\mt E$ satisfies the basic inequality at $Z$ if and only if $\mt F_Z$ and $\mt F_{Z^c}$ satisfy the inequality (\ref{second}), and we have the proof.

$(i)\Leftrightarrow (iii)$. Let $Z\subset C$ be a (not necessarily connected) subcurve. By direct computation:
\begin{align}
\nonumber\chi(\mt G_Z)&=\chi\left((\left(\det\mathcal E_Z\right)^{\otimes 2-2g}\otimes\omega_{Z}^{\otimes d+r(g-1)}\right)+\left( r(2g-2)-1\right)\cdot\chi\left(\oo_Z\right)=\\
\nonumber&=-\omega_C\cdot \deg\mt E_Z+\omega_Z\left(d+r\frac{\omega_C}{2}\right)+\chi(\oo_Z)+(r\cdot \omega_C-1)\chi(\oo_Z)=\\
\nonumber&=-\omega_C\cdot \deg\mt E_Z+\omega_Z\cdot d + r\cdot \omega_C \frac{k_Z}{2}=\omega_C\left(-\deg\mt E_Z+d\frac{\omega_Z}{\omega_C} + r\frac{k_Z}{2}\right).
\end{align}
Since $\omega_C$ is a positive integer, we have that $\mt E$ satisfies the basic inequality for $Z$ if and only if $\chi(\mt G_Z)$ and $\chi(\mt G_{Z^c})$ are $\geq 0$. And we obtain the proof.

It remains to show the last assertion. Assume the $\mt E$ satisfies the inequality (\ref{third}) for any connected subcurve such that the complementary one is still connected. We first observe the followings:
\begin{enumerate}[(a)]
\item $\chi(\mt G_Z)=\sum_{i=1}^{m}\chi(\mt G_{Z_i})$ and $k_Z=\sum_{i=1}^mk_{Z_i}$, for any subcurve $Z\subset C$ with  $Z_1,\ldots, Z_m$ as connected components, 
\item $0=\chi(\mt G_C)=\chi(\mt G_Z)+\chi(\mt G_{Z^c})-r\cdot\omega_C\cdot k_Z$ for any subcurve $Z\subset C$,
\item if $Z$ is connected and $(Z^c)_1,\ldots (Z^c)_m$ are the connected components of $Z^c$, we have that $((Z^c)_i)^c=\cup_{j\neq i}((Z^c)_j)^c\cup Z$ is connected,
\item if $Z$ and $Z^c$ are connected, then 
$$
0\leq \chi(\mt G_{Z^c})=\omega_C\left(-\deg\mt E_{Z^c}+d\frac{\omega_{Z^c}}{\omega_C} + r\frac{k_Z}{2}\right)\leq r\cdot\omega_C\cdot k_Z.
$$
\end{enumerate}
Let $Z\subset C$ be any subcurve. By (a), we can assume that $Z$ is connected. Then we have
$$
\chi(\mt G_Z)\stackrel{(b)}{=}-\chi(\mt G_{Z^c})+r\cdot\omega_C\cdot k_Z\stackrel{(a)}{=}-\sum_{i=1}^m\chi(\mt G_{(Z^c)_i})+r\cdot\omega_C\cdot k_{Z_i}\stackrel{(c)+(d)}{\geq} 0.
$$
This concludes the proof.
\end{proof}

\begin{lem}\label{opencond}Let $(p:C\rightarrow S, \mathcal E)$ be a vector bundle of rank $r$ and degree $d$ over a family of curves. Suppose that $S$ is locally noetherian. The locus where $C$ is a semistable curve and $\mt E$ is strictly positive, resp. admissible, resp. properly balanced, is open in $S$.
\end{lem}

\begin{proof}We can suppose that $S$ is noetherian and connected. Suppose that there exists a point $s\in S$ such that the geometric fiber is a properly balanced vector bundle over a semistable curve. It is known that the locus of semistable curves is open on $S$ (see \cite[\href{https://stacks.math.columbia.edu/tag/0E6X}{Tag 0E6X}]{stacks-project}). So we can suppose that $C\rightarrow S$ is a family of semistables curves of genus $g$. Up to twisting by a suitable power of $\omega_{C/S}$ we can assume, by Proposition \ref{ampleness}, that the rational $S$-morphism
$$
i:C\dashrightarrow Gr(p_*\mathcal E,r)
$$
is a closed embedding over $s$. By \cite[Lemma 3.13]{K}, there exists an open neighborhood $S'$ of $s$ such that $i$ is a closed embedding. Equivalently $\mathcal E_{S'}$ is strictly positive by Proposition \ref{ampleness}. We denote as usual with $\pi:C\rightarrow C^{st}$ the stabilization morphism. By Proposition \ref{NSlemma4}, the sheaf $\pi_*(\mt E_{S'})$ is flat over $S'$ and the push-forward commutes with the restriction on the fibers. In particular, it is torsion free at the fiber $s$, and so there exists an open subset $S''$ of $S'$ where $\pi_*(\mt E_{S''})$ is torsion-free over any fiber (see \cite[Proposition 2.3.1]{HL}). By Proposition \ref{prop5NS}, $\mt E_{S''}$ is admissible. Putting everything together, we obtain an open neighbourhood $S''$ of $s$ such that over any fiber we have an admissible vector bundle over a semistable curve. Let $0\leq k\leq d$, $0\leq i\leq g$ be integers. Consider the relative Hilbert scheme
$$
Hilb_{C/S''}^{\oo_{C}(1),P(m)=km+1-i},
$$
where $\oo_{C}(1)$ is the line bundle induced by the embedding $i$. We call $H_{k,i}$ the closure of the locus of semistable curves in $Hilb_{C/S''}^{\oo_{C}(1),P(m)=km+1-i}$ and let $Z_{k,i}\hookrightarrow C\times_{S''}H_{k,i}$ be the universal curve. Consider the vector bundle $\mathcal G$ over $C\rightarrow S''$ as in Lemma \ref{balcon}(iii). Let $\mathcal G^{k,i}$ be its pull-back on $Z_{k,i}$. The function
$$\chi: h\mapsto \chi(\mathcal G^{k,i}_h)$$
is locally constant on $H_{k,i}$. Now $\pi: H_{k,i}\rightarrow S''$ is projective. So the projection on $S''$ of the connected components of
$$
\bigsqcup_{0\leq k\leq d\atop 0\leq i\leq g}\\ H_{k,i}
$$
such that $\chi$ is negative is a closed subscheme. Its complement in $S$ is open and, by Lemma \ref{balcon}(iii), it contains $s$ and defines a family of properly balanced vector bundles over semistable curves.
\end{proof}

\subsection{The moduli stack of properly balanced vector bundles $\CVc$.}\label{univ}

Now we will introduce our main object of study: \emph{the universal moduli stack $\CVc$ of properly balanced vector bundles of rank $r$ and degree $d$ on semistable curves of arithmetic genus $g$}. Roughly speaking, we want a space such that its points are in bijection with the pairs $(C,\mt E)$, where $C$ is a semistable curve on $k$ and $\mt E$ is a properly balanced vector bundle on $C$. This subsection is devoted to the construction of this space as an Artin stack.

\begin{defin}Let $r$, $d$, $g$ integers such that $r\geq 1$, $g\geq 2$. Let $\CVc$ be the category fibered in groupoids over $Sch/k$ whose objects over a scheme $S$ are the families of semistable curves of genus $g$ with a properly balanced vector bundle of total degree $d$ and rank $r$. 

A morphism between two objects $(C\to S, \mt E)$, $(C'\to S',\mt E')$ is the datum of a cartesian diagram
$$
\xymatrix{
C\ar[r]^{\varphi}\ar[d]& C'\ar[d]\\
S\ar[r]&S'
}
$$
and an isomorphism $\varphi^*\mt E'\cong \mt E$ of vector bundles over $C$.
\end{defin}

The aim of this subsection is proving the following

\begin{teo}\label{teo1}$\CVc$ is an irreducible smooth Artin stack of dimension $(r^2+3)(g-1)$. Furthermore, it admits an open cover $\{\overline{\mathcal U}_n\}_{n\in\mathbb Z}$ such that $\overline{\mathcal U}_n$ is a quotient stack of a smooth noetherian scheme by a suitable general linear group.
\end{teo}

\begin{rmk}\label{capo} In the case $r=1$, $\overline{\mt Vec}_{1,d,g}$ is quasi compact and it corresponds to the compactification of the universal Jacobian over $\CMg$ constructed by Caporaso \cite{Cap94} and later generalized by Melo \cite{Mel09}. Following the notation of \cite{MV}, we will set $\CJc:=\overline{\mt Vec}_{1,d,g}$.
\end{rmk}

The proof consists in several steps, following the strategies adopted by Kausz \cite{K} and Wang \cite{W}. First, we observe that $\CVc$ is clearly a stack for the Zariski topology. We now prove that it is a stack also for the fpqc topology (defined in \cite[Section 2.3.2]{FAG}) with representable diagonal. 

\begin{prop}\label{diagonalrep} $\CVc$ is a stack for the fpqc topology and its diagonal is representable, quasi-compact and separated.
\end{prop}

\begin{proof}We need to show the followings
\begin{enumerate}[(i)]
\item let $S'\rightarrow S$ be an fpqc morphism of schemes, set $S'\times_S S'$ and $\pi_i$ the natural projections. Let $(C'\rightarrow S',\mathcal E')\in\CVc(S')$. Then every descent data
$$
\varphi:\pi_1^*(C'\rightarrow S',\mt E')\cong\pi_2^*(C'\rightarrow S',\mt E')
$$
is effective.
\item Let $S$ be an affine scheme. Let $(C\rightarrow S,\mt E),\;(C'\rightarrow S,\mt E')\in\CVc(S)$. The contravariant functor
$$
(T\rightarrow S)\mapsto \text{Isom}_T((C_T,\mt E_{T}),(C'_T,\mt E'_{T}))
$$
is representable by a quasi-compact separated $S$-scheme.
\end{enumerate}
We first prove (i). By standard limit argument, we can assume that $S$ and $S'$ are noetherian. Up to twisting by a power of the dualizing sheaf, we can assume that the restriction of $\mt E$ to each irreducible component of any geometric fibre has positive degree if the component is not contained in some destabilizing chain. The positivity on the destabilizing components is guaranteed by the admissibility of $\mt E$. In particular, we can suppose that $\det\mt E'$ is relatively ample on $S'$, in particular $\varphi$ induces a descent data for the polarized family of curves $(C'\rightarrow S',\det\mt E')$ and this is effective by \cite[Theorem. 4.38]{FAG}. So there exists a family of curves $C\rightarrow S$ such that its pull-back via $S'\rightarrow S$ is $C'\rightarrow S'$. In particular, $C'\rightarrow C$ is an fpqc cover and $\varphi$ induces a descent data for $\mt E'$ on $C'\rightarrow C$, which is effective by \cite[Theorem. 4.23]{FAG}.

It remains to prove (ii). Using again the standard limit argument, we can restrict to the category of locally noetherian schemes. Suppose that $S$ is an affine connected noetherian scheme. Consider the contravariant functor
$$
\left(T\rightarrow S\right)\mapsto \text{Isom}\left(C_T,C'_T\right).
$$
This functor is represented by a scheme $B$ (see \cite[pp. 47-48]{ACG11}). More precisely: let $Hilb_{C\times_S C'/S}$ be the Hilbert scheme which parametrizes closed subschemes of $C\times_S C'$ flat over $S$. $B$ is the open subscheme of $Hilb_{C\times_S C'/S}$ with the property that a morphism $f:T\rightarrow Hilb_{C\times_S C'/S}$ factorizes through $B$ if and only if the projections $\pi: Z_T\rightarrow C_T$ and $\pi': Z_T\rightarrow C'_T$ are isomorphisms, where $Z_T$ is the closed subscheme of $C\times_S C'$ represented by $f$. Consider the universal pair
$$
\left(Z_B,\varphi:=\pi'\circ \pi^{-1}: C_B\cong Z_B\cong C'_B\right).
$$
Now we prove that $B$ is quasi-projective. By construction it is enough to show that $B$ is contained in $Hilb_{C\times_S C'/S}^{\mt L,P(m)}$, which parametrizes closed subschemes of $C\times_S C'/S$ with Hilbert polynomial $P(m)$ respect to the relatively ample line bundle $\mt L$ on $C\times_S C'/S$. Let $\mt L$ (resp. $\mt L'$) be a relatively very ample line bundle on $C/S$ (resp. $C'/S$). We can take $\mt L=(\det\mt E)^n$ and $\mt L'=(\det\mt E')^n$ for $n$ big enough. Then the sheaf $\mt L\boxtimes_S\mt L'$ is relatively very ample on $C\times_S C'/S$. Using the projection $\pi$ we can identify $Z_B$ and $C_B$. The Hilbert polynomial of $Z_B$ with respect to the polarization $\mt L\boxtimes_S\mt L'$ is
$$
P(m)=\chi((\mt L\boxtimes_S\mt L')^m)=\chi(\mt L^m\otimes \varphi^*\mt L'^m)=\deg(\mt L^m)+\deg(\mt L'^m)+1-g.
$$
It is clearly independent from the choice of the point in $B$ and from $Z_B$, proving the quasi-projectivity. In particular, $B$ is quasi-compact and separated over $S$. The proposition follows from the fact that the contravariant functor
$$
\left(T\rightarrow B\right)\mapsto\text{Isom}_{C_T}\left(\mt E_T,\varphi^*\mt E'_T\right)
$$
is representable by a quasi-compact separated scheme over $B$ (see the proof of \cite[Theorem 4.6.2.1]{LMB}).

\end{proof}

We now introduce a useful open cover of the stack $\CVc$. We will prove that any open subset of this cover has a presentation as quotient stack of a scheme by a suitable general linear group. In particular, $\CVc$ admits a smooth surjective representable morphism from a locally noetherian scheme. Putting together this fact with Proposition \ref{diagonalrep}, we get that $\CVc$ is an Artin stack locally of finite type.

\begin{prop}For any scheme $S$ and any $n\in \mathbb Z$, consider the subgroupoid $\overline{\mathcal U}_n(S)$ in $\CVc (S)$ of pairs $(p:C\rightarrow S,\mathcal E)$ such that
\begin{enumerate}[(i)]
\item $R^ip_*\mathcal E(n)=0$ for any $i>0$,
\item $\mathcal E(n)$ is relatively generated by global sections, i.e. the canonical morphism $p^*p_*\mathcal E(n)\rightarrow\mathcal E(n)$ is surjective, and the induced morphism $C\rightarrow Gr(p_*\mathcal E(n),r)$ is a closed embedding.
\end{enumerate}
Then the sheaf $p_*\mathcal E(n)$ is flat on $S$ and $\mathcal E(n)$ is cohomologically flat over $S$. In particular, for any base change $T\to S$, the pull-back $(p_T:C_T\to T,\mt E_T)$ is an object of $\overline{\mt U}_n(T)$, i.e. $\overline{\mathcal U}_n$  is a fibered full subcategory.
\end{prop}

\begin{proof}We set $\mt F:=\mt E(n)$. By \cite[Proposition 4.1.3]{W}, we know that $p_*\mathcal F$ is flat on $S$ and $\mathcal F$ is cohomologically flat over $S$. Consider the following cartesian diagram
$$\xymatrix{
C_T \ar[d]^{p_T}\ar[r] & C\ar[d]^p\\
T\ar[r]^{\phi} & S
}$$
By \emph{loc. cit.}, we have that $R^ip_{T*}(\mathcal F_T)=0$ for any $i>0$ and that $\mathcal F_T$ is relatively generated by global sections. It remains to prove that the induced $T$-morphism $C_T\rightarrow Gr(p_{T*}\mathcal F_T,r)$ is a closed embedding. By cohomological flatness and the base change property of the Grassmannian, this is nothing but the pull-back of the closed embedding $C\to Gr(p_*\mt F,r)$ along $T\to S$.
\end{proof}

\begin{lem} \label{subcat}The subcategories $\left\{ \overline{\mt U}_n\right\}_{n\in\mathbb Z}$ form an open cover of $\CVc$.
\end{lem}

\begin{proof}Let $S$ be a scheme, $(p:C\rightarrow S,\mathcal E)$ an object of $\CVc(S)$ and $n$ an integer. We must prove that there exists an open $U_n\subset S$ with the universal property that $T\rightarrow S$ factorizes through $U_n$ if and only if $\mathcal E_T$ is an object of $\overline{\mathcal U}_n(T)$.\\
We can assume $S$ affine and noetherian. Let $\mathcal F:=\mathcal E(n)$ and $U_n$ the subset of points of $S$ such that:
\begin{enumerate}[(i)]
\item $H^i(C_s,\mathcal F_s)=0$ for $i>0$,
\item $H^0(C_s,\mathcal F_s)\otimes\oo_{C_s}\rightarrow \mathcal F_s$ is surjective,
\item the induced morphism in the Grassmannian $C_s\rightarrow Gr(H^0(C_s,\mathcal F_s),r)$ is a closed embedding.
\end{enumerate}
We must prove that $U_n$ is open and satisfies the universal property. As in the proof of \cite[Lemma 4.1.5]{W}, consider the open subscheme $W_n\subset S$ satisfying the first two conditions above. By definition it contains $U_n$ and it satisfies the universal property that any morphism $T\rightarrow S$ factorizes through $W_n$ if and only if $R^ip_{T*}\mathcal F_T=0$ for any $i>0$ and $\mathcal F_T$ is relatively generated by global sections. By \cite[Proposition 4.1.3]{W}, $\mt F_{W_n}$ is cohomologically flat over $W_n$. This implies that the fiber over a point $s$ of the morphism
$$
C_{V_n}\rightarrow Gr(p_{W_n*}\mathcal F_{W_n},r)
$$
is exactly $C_s\rightarrow Gr(H^0(C_s,\mathcal F_s),r)$. Since the property of being  a closed embedding for a morphism of proper $W_n$-schemes is an open condition (see \cite[Lemma 3.13]{K}), it follows that $U_n$ is an open subscheme and $\mathcal F_{U_n}\in\overline{\mathcal U}_n(U_n)$.\\
Viceversa, suppose now that $\phi:T\rightarrow S$ is such that $\mathcal F_T\in\overline{\mathcal U}_n(T)$. The morphism factors through $V_n$ and for any $t\in T$
$$
C_t\rightarrow Gr(H^0(C_t,\mathcal F_t),r)
$$
is a closed embedding. Since the morphism $\phi$ restricted to a point $t\in T$ onto is image $\phi(t)$ is fpqc, by descent the morphism $C_{\phi(t)}\rightarrow Gr(H^0(C_{\phi(t)},\mathcal F_{\phi(t)}),r)$ is a closed embedding, or in other words $\phi(t)\in U_n$.\\
It remains to prove that $\{\mathcal U_n\}$ is a covering. It is sufficient to prove that for any point $s$, there exists $n$ such that $\mathcal E_s(n)$ satisfies the conditions (i), (ii) and (iii). By Proposition \ref{NSlemma4}, $H^i(C_s,\mt F_s)=H^i(C^{st}_s,\pi_*(\mt E)\otimes\omega^n_{C_s^{st}})$, where $\pi:C_s\to C_s^{st}$ is the stabilization morphism. Since $\omega_{C_s^{st}}$ is ample, there exists $n$ big enough such that (i) is satisfied, and by Proposition \ref{ampleness} the same holds for (ii) and (iii).
\end{proof}

\begin{rmk}\label{wang}
As in \cite[Remark 4.1.7]{W} for a scheme $S$ and a pair $(p:C\rightarrow S,\mathcal E)\in\overline{\mathcal U_n}(S)$, the direct image $p_*(\mathcal E(n))$ is locally free of rank $d+r(2n-1)(g-1)$. By cohomological flatness, locally on $S$ the morphism in the Grassmannian becomes $
C\hookrightarrow Gr(V_n,r)\times S
$, where $V_n$ is a $k$-vector space of dimension $P(n):=d+r(2n-1)(g-1)$.
\end{rmk}

We are now going to obtain a presentation of $\overline{\mt U}_n$ as a quotient stack. As in the above remark, let $V_n$ be a $k$-vector space of dimension $P(n):=d+r(2n-1)(g-1)$. Consider the Hilbert scheme of closed subschemes on the Grassmannian $Gr(V_n,r)$
$$
Hilb_n:=Hilb_{Gr(V_n,r)/\Spec k}^{\oo_{Gr(V_n,r)}(1), Q(m)},
$$
with Hilbert polynomial $Q(m)=m(d+nr(2g-2))+1-g$ relative to the Plucker line bundle $\oo_{Gr(V_n,r)}(1)$. Let $\mathscr{C}_{(n)}\hookrightarrow Gr(V_n,r)\times Hilb_n$ be the universal curve. The Grassmannian is equipped with a universal quotient $V_n\otimes \oo_{Gr(V_n,r)}\to \mt E$, where $\mt E$ is the universal vector bundle. If we pull-back this morphism on the product $Gr(V_n,r)\times Hilb_n$ and we restrict to the universal curve, we obtain a surjective morphism of vector bundles $q:V_n\otimes\oo_{\mathscr {C}_{(n)}}\rightarrow \mathscr E_{(n)}$. We will call $\mathscr E_{(n)}$ (resp. $q:V_n\otimes\oo_{\mathscr {C}_{(n)}}\rightarrow \mathscr E_{(n)}$) the \emph{universal vector bundle (resp. universal quotient) on $\mathscr {C}_{(n)}$}. \label{Hn}Let $H_n$ be subset of $Hilb_n$ consisting of points $h$ such that:
\begin{enumerate}[(i)]
\item $\mathscr {C}_{(n)h}$ is semistable,
\item $\mathscr {E}_{(n)h}$ is properly balanced,
\item $H^i(\mathscr {C}_{(n)h},\mathscr {E}_{(n)h})=0$ for $i>0$,
\item $H^0(q_h)$ is surjective, (or equivalently an isomorphism, since $\dim V_n=H^0(\mathscr {C}_{(n)h},\mathscr {E}_{(n)h})$).
\end{enumerate}
This subset is an open subscheme of $Hilb_n$. Indeed, the conditions (i) and (ii) are open by Lemma \ref{opencond} and (iii) is open by \cite[Lemma 4.1.5]{W} (see also proof of Lemma \ref{subcat}). Consider the morphism $f_*q:V_n\otimes\oo_{Hilb_n}\to f_*\mathscr {E}_{(n)}$, where $f:\mathscr C_{(n)}\to Hilb_n$ is the universal curve. Then $H_n$ is the intersection of the open subset where $(i)$, $(ii)$ and $(iii)$ hold and the complement of the support of the cokernel of $ f_*q$, for more details see proof of \cite[Lemma 4.1.10]{W}.\vspace{0.2cm}

We denote by the same symbols the restriction over $H_n$ of the universal pair. In particular, the twisted pair $$\left(\mathscr C_{(n)}\to H_n, \mathscr E_{(n)}(-n)\stackrel{def}{=}\mathscr E_{(n)}\otimes \omega^{-n}_{\mathscr C_{(n)}}\right)$$ defines an object in $\CVc(H_n)$. By definition of $H_n$, it must factors through $\overline{\mt U}_n$. In other words, the restriction of the universal curve and of the universal vector bundle on $H_n$ defines a morphism of stacks
$$
\begin{array}{llll}
\Theta: &H_n&\rightarrow &\overline{\mathcal U}_n\\
& h&\mapsto & \left(\mathscr {C}_{(n)h}, \mathscr {E}_{(n)h}(-n)\right).
\end{array}
$$
Moreover, the Hilbert scheme $Hilb_n$ is equipped with a natural action of $GL(V_n)$: if $A\in GL(V_n)$ we denote with the same symbol the induced automorphism $A:Gr(V_n,r)\cong Gr(V_n,r)$ of the Grassmannian. The action of a matrix $A\in GL(V_n)$ is given by $A.[C\hookrightarrow Gr(V_n,r)]:=[C\hookrightarrow Gr(V_n,r)\xrightarrow{A^{-1}}Gr(V_n,r)]$. The open $H_n$ is stable with respect to this action.

\begin{prop}The morphism of stacks
$$
\Theta: H_n\rightarrow \overline{\mathcal U}_n
$$
is a $GL(V_n)$-bundle (in the sense of \cite[2.1.4]{W}).
\end{prop}

\begin{proof}We set $GL:=GL(V_n)$. First we prove that $\Theta$ is $GL$-invariant, i.e.
\begin{enumerate}[(a)]
\item the diagram
$$
\xymatrix{
H_n\times GL\ar[r]^m\ar[d]^{pr_{1}} & H_n\ar[d]^{\Theta}\\
H_n\ar[r]^{\Theta} & \overline{\mathcal U}_n
}
$$
where $pr_1$ is the projection on $H_n$ and $m$ is the multiplication map is commutative. Equivalently, there exists a natural transformation $\rho:pr_1^*\Theta\rightarrow m^*\Theta$.
\item $\rho$ satifies an associativity condition (see \cite[2.1.4]{W}).
\end{enumerate}
As we will see below, the diagram in (a) is cartesian. So, there exists a canonical natural transformation $\rho$ and it can be shown that it satisfies the associativity condition. We will fix a pair $(p:C\rightarrow S,\mathcal E)\in\overline{\mathcal U}_n(S)$ and let $f:S\rightarrow\overline{\mathcal U}_n$ be the associated morphism. We prove that the morphism $f^*\Theta$ is a principal $GL$-bundle. More precisely, we will prove that there exists a $GL$-equivariant isomorphism over $S$
$$H_n\times_{\overline{\mathcal U}_n}S\cong Isom( V_n\otimes\oo_S,p_{*}\mathcal E(n)).$$
For any $S$-scheme $T$, a $T$-valued point of $H_n\times_{\overline{\mathcal U}_n}S$ corresponds to the following data:
\begin{enumerate}[(1)]
\item a morphism $T\rightarrow H_n$,
\item a $T$-isomorphism of schemes $\psi:C_T\cong \mathscr {C}_{(n)T}$,
\item an isomorphism of vector bundles $\psi^*\mathscr {E}_{(n)T}\cong\mathcal E_T (n)$.
\end{enumerate}
Consider the pull-back of the universal quotient of $H_n$ through $T\rightarrow H_n$
$$
q_T:V_n\otimes\oo_{\mathscr {C}_{(n)T}}\rightarrow \mathscr {E}_{(n)T}.
$$
If we pull-back by $\psi$ and compose with the isomorphism of (3), we obtain a surjective morphism
$$
V_n\otimes\oo_{C_T}\rightarrow \mathcal E_T(n).
$$
We claim that the push-forward $V_n\otimes \oo_T\rightarrow p_{T*}(\mathcal E_T(n))\cong p_*(\mathcal E(n))_T$ is an isomorphism, or in other words it defines a $T$-valued point of $Isom( V_n\otimes \oo_S,p_{*}\mathcal E(n))$. As explained in Remark \ref{wang}, the sheaf $p_{T*}(\mathcal E (n)_T)$ is a vector bundle of rank $P(n)$, so it is enough to prove the surjectivity. We can suppose that $T$ is noetherian and by Nakayama's lemma it suffices to prove the surjectivity on the fibers. On a fiber the morphism is
$$
V_n\otimes\oo_{C_t}\rightarrow H^0(\mathscr {E}_{(n)t})\cong H^0(\mathcal E_t(n))
$$
which is an isomorphism by the definition of $H_n$.\\
Conversely, let $T$ be a scheme and $V_n\otimes\oo_T\rightarrow p_*(\mathcal E(n))_T$ a $T$-isomorphism of vector bundles.
By hypothesis, $\mathcal E_T (n)$ is relatively generated by global sections and the induced morphism in the Grassmannian is a closed embedding. Putting everything together, we obtain a surjective map
$$
V_n\otimes\oo_{C_T}\cong p_T^*p_{T*}\mathcal E_T(n)\rightarrow \mathcal E_T(n)
$$
and a closed embedding $C_T\hookrightarrow Gr(V_n,r)\times T$ which defines a morphism $T\rightarrow H_n$. By the representability of the Hilbert functor and of the Grassmannian, there exists a unique isomorphism of schemes $\psi:C_T\cong\mathscr {C}_{(n)T}$ and a unique isomorphism of vector bundles $\psi^*\mathscr {E}_{(n)T}\cong\mathcal E_T(n)$. Then we have obtained a $T$-valued point of $H_n\times_{\overline{\mathcal U}_n} S$. The two constructions above are inverses of each other, concluding the proof.
\end{proof}

\begin{prop}\label{qts}The map $\Theta: H_n\rightarrow \overline{\mathcal U}_n$ gives an isomorphism of stacks
$$
\overline{\mathcal U}_n\cong [H_n/GL(V_n)]
$$
\end{prop}

\begin{proof}This follows from \cite[Lemma 2.1.1.]{W}.
\end{proof}

\begin{rmk}An interesting consequence of this presentation is that the diagonal of $\CVc$ is affine (see \cite[Theorem 2.0.2]{W}).
\end{rmk}

This will conclude the proof of Theorem \ref{teo1}, except for the irreducibility and smoothness of $\CVc$, which will be proved in Propositions \ref{finteo1} and \ref{finteo2}, respectively.
\subsection{Properties and the rigidified moduli stack $\CVr$.}\label{prop}

The stack $\CVc$ admits a \emph{universal curve} $\overline{\pi}: \overline{\mt Vec}_{r,d,g,1}\rightarrow \CVc$, i.e. a stack $\overline{\mt Vec}_{r,d,g,1}$ and a representable morphism $\overline{\pi}$ with the property that for any morphism from a scheme $S$ to $\CVc$ associated to a pair $(C\rightarrow S,\mt E)$, there exists a morphism $C\rightarrow \overline{\mt Vec}_{r,d,g,1}$ such that the diagram
$$
\xymatrix{
	C\ar[r]\ar[d] &\overline{\mt Vec}_{r,d,g,1}\ar[d]^{\overline{\pi}}\\
	S\ar[r] & \CVc}
$$
is cartesian. Furthermore, the universal curve admits a \emph{universal vector bundle}, i.e. for any morphism from a scheme $S$ to $\CVc$ associated to a pair $(C\rightarrow S,\mt E)$, we associate the vector bundle $\mt E$ on $C$. This allows us to define a coherent sheaf for the site lisse-\'etale on $\overline{\mt Vec}_{r,d,g,1}$ flat over $\CVc$.
The stabilization morphism induces a morphism of stacks
$$\overline{\phi}_{r,d}:\CVc\lra \CMg,$$
which forgets the vector bundle and sends the curve to its stabilization. We will denote by $\Vc$ (resp. $\mt U_n$) the open substack of $\CVc$ (resp. $\overline{\mt U}_n)$ of pairs $(C,\mt E)$, where $C$ is a smooth curve. In the next sections we will often need the restriction of $\overline{\phi}_{r,d}$ to the open locus of smooth curves $$\phi_{r,d}:\Vc\lra \Mg.$$

The group $\mathbb G_m$ is contained in a natural way in the automorphism group of any object of $\CVc$, as multiplication by scalars on the vector bundle. There exists a procedure for removing these automorphisms, called \emph{$\mathbb{G}_m$-rigidification} (see \cite[Section 5]{ACV}). We obtain an irreducible smooth Artin stack $\CVr:=\CVc\fatslash \mathbb{G}_m$ of dimension $(r^2+3)(g-1)+1$, with a surjective smooth morphism $\nu_{r,d}:\CVc\rightarrow\CVr$. Furthermore, the morphism $\nu_{r,d}$ makes $\CVc$ into a gerbe banded by $\mathbb G_m$ (a $\mathbb G_m$-gerbe in short) over $\CVr$ (see \cite{Gi} for the theory of gerbes).
The forgetful morphism $\overline{\phi}_{r,d}:\CVc\lra \CMg$ factors through the forgetful morphism $\overline{\phi}_{r,d}:\CVr\rightarrow\CMg$. Over the locus of smooth curves we have the following diagram
$$
\xymatrix{
	\Vc\ar[d]^{det}\ar[rr]^{\nu_{r,d}} & & \Vr\ar[d]^{\widetilde{det}}\\
	\Jc\ar[rd]\ar[rr]^{^{\nu_{1,d}}} & & \Jr\ar[ld]\\
	& \Mg
}
$$
where $det$ (resp. $\widetilde{det}$) is the determinant morphism, which send an object $(C\rightarrow S,\mt E)\in\Vc(S)$ (resp. $\in\Vr(S)$) to $(C\rightarrow S,\det\mt E)\in\Jc(S)$ (resp. $\in\Jr(S)$). Observe that the obvious extension on $\CVc$ of the determinant morphism does not map to the compactified universal Jacobian $\CJc$, because in general, if $\mt E$ is a balanced vector bundle over $C$, then $\det\mt E$ does not satisfy the basic inequality for $\CJc$.

\subsection{Local structure.}\label{locstr}
In this section, we will describe the local structure of $\CVc$ in terms of deformation theory of the pairs $(C,\mt E)$, where $C$ is a semistable curve and $\mt E$ is a properly balanced vector bundle. Therefore we are going to review the necessary facts. First of all, the deformation functor $\text{Def}_C$ of a semistable curve $C$ is smooth (see \cite[Proposition 2.2.10(i), Proposition 2.4.8]{Ser06}) and it admits a miniversal deformation ring (see \cite[Theorem 2.4.1]{Ser06}), i.e. there exists a formally smooth morphism of functors of local Artin $k$-algebras
\begin{equation}\label{formsmooth}
\text{Spf }k\llbracket x_1,\ldots,x_{N}\rrbracket\rightarrow \text{Def}_C, \text{ where } N:=\text{ext}^1(\Omega_C,\oo_C)
\end{equation}
inducing an isomorphism between the tangent spaces. Moreover, if $C$ is stable its deformation functor admits a universal deformation ring (see \cite[Corollary 2.6.4]{Ser06}), i.e. the morphism of functors above is an isomorphism. Let $x$ be a singular point of $C$ and $\hat{\oo}_{C,x}$ the completed local ring of $C$ at $x$. The deformation functor $\text{Def}_{\text{Spec}\hat{\oo}_{C,x}}$ admits a miniversal deformation ring $k\llbracket t\rrbracket$ (see \cite[pag. 81]{DM69}). Let $\Sigma$ be the set of singular points of $C$. The morphism of Artin functors
\begin{equation}\label{}
loc:\text{Def}_C\rightarrow \prod_{x\in\Sigma}\text{Def}_{\text{Spec}\hat{\oo}_{C,x}}
\end{equation}
is formally smooth (see \cite[Proposition 1.5]{DM69}). For a vector bundle $\mt E$ over $C$, we will denote by $\text{Def}_{(C,\mt E)}$ the deformation functor of the pair (for a more precise definition see \cite[Def. 3.1]{CMKV}). As in \cite[Def. 3.4]{CMKV}, the automorphism group $\text{Aut}(C,\mt E)$ (resp. $\text{Aut}(C)$) acts on $\text{Def}_{(C,\mt E)}$ (resp. $\text{Def}_C$). Using the same argument of \cite[Lemma 5.2]{CMKV}, we can see that the multiplication by scalars on $\mt E$ acts trivially on $\text{Def}_{(C,\mt E)}$. By \cite[Theorem 8.5.3]{FAG}, the forgetful morphism
\begin{equation}\label{cec}
\text{Def}_{(C,\mt E)}\rightarrow \text{Def}_C
\end{equation}
is formally smooth and the tangent space of $\text{Def}_{(C,\mt E)}$ has dimension $\text{ext}^1(\Omega_C,\oo_C)+\text{ext}^1(\mt E,\mt E)$.
Let $h:=[C\hookrightarrow Gr(V_n, r)]$ be a $k$-point of $H_n$. Let $\hat{\oo}_{H_n,h}$ be the completed local ring of $H_n$ at $h$. Clearly, the ring $\hat{\oo}_{H_n,h}$ is a universal deformation ring for the deformation functor $\text{Def}_{h}$ of the closed embedding $h$. Moreover

\begin{lem}\label{111}The natural morphism $\text{Def}_h\rightarrow \text{Def}_{(C,\mt E)}$ is formally smooth.
\end{lem}

\begin{proof}For any $k$-algebra $R$, we will set $Gr(V_n,r)_R:=Gr(V_n,r)\times_k \text{Spec}R$. We have to prove that given
	\begin{enumerate}
		\item a surjection $B\rightarrow A$ of Artin local $k$-algebras,
		\item a deformation $h_A:=[C_A\hookrightarrow Gr(V_n,r)_A]$  of $h$ over $A$,
		\item a deformation $(C_B,\mt E_B)$ of $(C,\mt E)$ over $B$, which is a lifting of $(C_A,\mt E_A)$,
	\end{enumerate}
	then there exists an extension $h_B$ over $B$ of $h_A$ which maps on $(C_B,\mt E_B)$. Since by hypothesis $H^1(C,\mt E(n))=0$, we can show that the restriction map $res:H^0(C_B,\mt E_B(n))\rightarrow H^0(C_A,\mt E_A(n))$ is surjective. Now $h_A$ only depends on the vector bundle $\mt E_A$ and on the choice of a basis for $H^0(C_A,\mt E_A(n))$. We can lift the basis, using the map $res$, to a basis $\mt B$ of $H^0(C_B,\mt E_B(n))$. The basis $\mt B$ induces a morphism
	$
	C_B\rightarrow Gr(V_n,r)_B
	$
	which is a lifting for $h_A$.
\end{proof}

The next two propositions conclude the proof of Theorem \ref{teo1}.
\begin{prop}\label{finteo1}The stack $\CVc$ is irreducible.
\end{prop}

\begin{proof}Since the morphism $loc:\text{Def}_C\rightarrow \prod_{x\in\Sigma}\text{Def}_{\text{Spec}\hat{\oo}_{C,x}}$ is formally smooth, Lemma \ref{111} implies that the morphism $Def_h\rightarrow \prod_{x\in\Sigma}\text{Def}_{\text{Spec}\hat{\oo}_{C,x}}$ is formally smooth. In particular, any semistable curve with a properly balanced vector bundle can be deformed to a smooth curve with a vector bundle. In other words, the open substack $\Vc$ is dense in $\CVc$; hence $\CVc$ is irreducible if and only if $\Vc$ is irreducible. Since the morphism of Artin functors $\text{Def}_{(C,\mt E)}\to\text{Def}_C$ is formally smooth for any nodal curve, the morphism $\Vc\to\Mg$ is smooth, in particular is open. So the proposition follows from the fact that $\Mg$ is irreducible and the fibre of $\Vc\to\Mg$ is integral (by \cite[Corollary A.5]{Ho10}).
\end{proof}

\begin{prop}\label{finteo2}The scheme $H_n$ and the stack $\CVc$ are smooth of dimension respectively $P(n)^2+(r^2+3)(g-1)$ and $(r^2+3)(g-1)$.
\end{prop}
\begin{proof}As observed in the previous proof, the map $\Vc\to\Mg$ is smooth. Moreover, its geometric fibre have dimension $r^2(g-1)$ (see \cite[Corollary A.5]{Ho10} and discussion below). In particular, 
\begin{align}
\dim\CVc&= r^2(g-1)+\dim\Mg=(r^2+3)(g-1),\nonumber\\
\dim H_n&= \dim GL(V_n)+\dim \CVc=P(n)^2+(r^2+3)(g-1)\nonumber.
\end{align}
We set $Gr:=Gr(V_n,r)$. Arguing as in \cite[Proposition 3.1.3.]{Sc}, we see that for any $k$-point $h:=[C\hookrightarrow Gr]\in H_n$ the co-normal sheaf $\mathcal I_C/\mathcal I^2_C$ is locally free and we have an exact sequence:
	$$
	0\lra\mathcal I_C/\mathcal I^2_C\lra\Omega^1_{Gr}|_C\lra\Omega^1_{C}\lra 0.
	$$
	Applying the functor $\text{Hom}_{\oo_C}(-,\oo_C)$, we obtain the following exact sequence of vector spaces
	$$
	0\longrightarrow \text{Hom}_{\oo_C}(\Omega^1_C,\oo_C)\longrightarrow H^0(C,T_{Gr}|_C)\longrightarrow \text{Hom}_{\oo_C}(\mathcal I_C/\mathcal I^2_C,\oo_C)\longrightarrow \text{Ext}^1_{\oo_C}(\Omega^1_C,\oo_C)\longrightarrow 0
	$$
	Now $\text{Hom}_{\oo_C}(\mathcal I_C/\mathcal I^2_C,\oo_C)$ is the tangent space of $H_n$ at $h$. We can prove that its dimension is $P(n)^2+(r^2+3)(g-1)$ by using the sequence above as in the proof of \emph{loc. cit.} This implies that $H_n$ is smooth of dimension $P(n)^2+(r^2+3)(g-1)$. The assertion for the stack $\CVc$ follows immediately from Proposition \ref{qts}.
\end{proof}

We are now going to construct a miniversal deformation ring for $\text{Def}_{(C,\mt E)}$ by taking a slice of $H_n$.

\begin{lem}\label{locstrc} Let $h:=[C\hookrightarrow Gr(V_n,r)]$ a $k$-point of $H_n$ and let $\mt E$ be the restriction to $C$  of the universal vector bundle. Assume that $\text{Aut}(C,\mt E)$ is smooth and linearly reductive. Then the followings hold.
	\begin{enumerate}[(i)]
		\item There exists a slice for $H_n$. More precisely, there exists a locally closed Aut$(C,\mt E)$-invariant subset $U$ of $H_n$, with $h\in U$, such that the natural morphism
		$$
		U\times_{Aut(C,\mt E)}GL(V_n)\rightarrow H_n
		$$
		is \'etale and affine and moreover the induced morphism of stacks
		$$
		[ U /Aut(C,\mt E)]\rightarrow \overline{\mt U}_n
		$$
		is affine and \'etale.
		\item The completed local ring $\hat{\oo}_{U,h}$ of $U$ at $h$ is a miniversal deformation ring for $\text{Def}_{(C,\mt E)}$.
	\end{enumerate}
\end{lem}

\begin{proof}The part (i) follows from \cite[Theorem 3]{Al10}.
	We will prove the second one following the strategy of \cite[Lemma 6.4]{CMKV}. We will set $F\subset Def_h$ as the functor pro-represented by $\hat{\oo}_{U,h}$, $G:=GL(V_n)$ and $N:=\text{Aut}(C,\mt E)$. Since $\text{Def}_h\rightarrow \text{Def}_{(C,\mt E)}$ is formally smooth, it is enough to prove that the restriction to $F(A)$ of $\text{Def}_h(A)\rightarrow \text{Def}_{(C,\mt E)}(A)$ is surjective for any local Artin $k$-algebra $A$ and bijective when $A=k[\epsilon]$.
	Let $\mathfrak{g}$ (resp. $\mathfrak{n}$) be the deformation functor pro-represented by the completed local ring of $G$ (resp. $N$) at the identity. There is a natural map $\mathfrak{g}/\mathfrak{n}\rightarrow\text{Def}_h$ given by the derivative of the orbit map. More precisely, for a local Artin $k$-algebra $A$:
	$$
	\begin{array}{ccc}
	\mathfrak{g}/\mathfrak{n}(A) & \rightarrow &\text{Def}_h(A)\\
	\left[g\right]&\mapsto & g.v^{triv}
	\end{array}
	$$
	where $v^{triv}$ is the trivial deformation over Spec$A$. First of all we will construct  a morphism $\text{Def}_h\rightarrow \mathfrak{g}/\mathfrak{n}$ such that the derivative of the orbit map defines a section. The construction is the following: up to \'etale base change, the morphism $U\times_N G\rightarrow H_n$ of part $(i)$, admits a section locally on $h$. The morphism, obtained composing this section with the morphism $U\times_NG\rightarrow G/N$, which sends a class $[(u,g]]$ to $[g]$, induces a morphism of Artin functors
	$$
	\text{Def}_h\rightarrow \mathfrak{g}/\mathfrak{h}
	$$
	with the desired property. By construction, if $A$ is a local Artin $k$-algebra then the inverse image of $0\in \mathfrak{g}/\mathfrak{n}(A)$ is $F(A)$. If $v\in \text{Def}_h(A)$ maps to some element $[g]\in\mathfrak{g}/\mathfrak{n}(A)$ then $g^{-1}v\in F(A)$. Because both $v$ and $g^{-1}v$ map to the same element of $\text{Def}_{(C,\mt E)}$, we can conclude that $F(A)\rightarrow \text{Def}_{(C,\mt E)}(A)$ is surjective.\\
	It remains to prove the injectivity of $F(k[\epsilon])\rightarrow \text{Def}_{(C,\mt E)}(k[\epsilon])$. We consider the following complex of $k$-vector spaces
	$$
	0\rightarrow \mathfrak{g}/\mathfrak{n}\rightarrow Def_h(k[\epsilon])\rightarrow Def_{(C,\mt E)}(k[\epsilon])\rightarrow 0
	$$
	where the first map is the derivative of the orbit map. We claim that this is an exact sequence, which would prove the injectivity of $F(k[\epsilon])\rightarrow \text{Def}_{(C,\mt E)}(k[\epsilon])$ by the definition of $F$. The only non obvious thing to check is the exactness in the middle. Suppose that $h_{k[\epsilon]}\in \text{Def}_h(k[\epsilon])$ is trivial in $\text{Def}_{(C,\mt E)}(k[\epsilon])$, i.e. if $q_{\epsilon}:V_n\otimes\oo_{C_{\epsilon}}\rightarrow\mt E_{\epsilon}$ represents the embedding $h_{k[\epsilon]}$, then there exists an isomorphism with the trivial deformation on $k[\epsilon]$: $\varphi:C_{\epsilon}\cong C[\epsilon]$ and $\psi:\varphi_*\mt E_{\epsilon}\cong \mt E[\epsilon]$. Consider the morphism
	$$
	g_{\epsilon}:=\psi\circ\varphi_*q_{\epsilon}:V_n\otimes\oo_{C[\epsilon]}\rightarrow\mt E[\epsilon]
	$$
	which represents the same class $h_{k[\epsilon]}$. By definition of $H_n$, the push-forward of $g_{\epsilon}$ on $k[\epsilon]$ is an isomorphism
	$$
	V_n\otimes k[\epsilon]\cong H^0(C,\mt E(n))\otimes k[\epsilon]
	$$
	and it defines uniquely the class $h_{k[\epsilon]}$.  We can choose basis for $V_n$ and $H^0(C,\mt E(n))$ such that $g_{\epsilon}$ differs from the trivial deformation of $\text{Def}_{h}(k[\epsilon])$ by an invertible matrix $g\equiv Id$ mod $\epsilon$, which concludes the proof.
\end{proof}

\subsection{The Schmitt compactification $\CU$.}\label{Sc}

In this section we will summarize how Schmitt in \cite{Sc}, generalizing a result of Nagaraj-Seshadri in \cite{NS}, constructs via GIT an irreducible normal projective variety $\CU$, which is a good moduli space (for the definition see Appendix \ref{App}) of an open substack of $\CVc$.\vspace{0.2cm}

First we recall the Seshadri's definition of slope-(semi)stable sheaf for a stable curve in the case of the canonical polarization.
\begin{defin}
Let $C$ be a stable curve and let $C_1,\ldots,C_s$ be its irreducible components. We will say that a sheaf $\mt E$ is \emph{P-(semi)stable} if it is torsion free of uniform rank $r$ and for any subsheaf $\mt F$ we have
$$
\frac{\chi(\mt F)}{\sum s_i\omega_{C_i}}\underset{(\leq)}<\frac{\chi(\mt E)}{r\omega_C},
$$
where $s_i$ is the rank of $\mt F$ at $C_i$. We will say that a P-semistable sheaf is \emph{strictly P-semistable} if it is not P-stable.

A P-semistable sheaf has a Jordan-Holder filtration with P-stable factors. Two P-semistable sheaves are \emph{equivalent} if they have the same Jordan-Holder factors. Two equivalence classes are said to be \emph{aut-equivalent} if they differ by an automorphism of the curve. 
\end{defin}
Consider the stack $\mt TF_{r,d,g}$ of torsion free sheaves of uniform rank $r$ and Euler characteristic $d+r(1-g)$ on stable curves of genus $g$. Pandharipande has proved in \cite{P96} that there exists an open substack $\mt TF^{ss}_{r,d,g}$, which admits a projective irreducible variety as moduli space. More precisely, this variety is a good moduli space for the aut-equivalence classes of P-semistable sheaves over stables curves (see \cite[Theorem 9.1.1]{P96}). This is the reason why we prefer the "P" instead of "slope" in the definition above.\vspace{0.2cm}

We will denote by $\CVc^{P(s)s}$ the substack of $\CVc$ of pairs $(C,\mt E)$ such that the sheaf $\pi_*\mt E$ over the stabilized curve $C^{st}$  is P-(semi)stable. Sometimes we will simply say that the pair $(C,\mt E)$ (or $(C^{st},\pi_*\mt E)$) is P-(semi)stable. As we will see in the next proposition, the set of these pairs is bounded.

\begin{prop}\label{qc}The substack $\CVc^{P(s)s}\subset \CVc$ is open and quasi-compact.
\end{prop}

\begin{proof}Let $(C\to S, \mt E)$ be an object in $\CVc$. Then $\pi_*(\mt E)$ is a torsion free sheaf over the family of stable curves $C^{st}\to S$ and it is flat over $S$, by Proposition \ref{NSlemma4}. Then, the locus of the points $s\in S$ such that the fibre $\pi_*(\mt E)_s$ is P-(semi)stable is open in $S$ by \cite[Proposition 2.3.1]{HL}. This implies that $\CVc^{P(s)s}$ is open in $\CVc$.

Now we show the quasi-compactness. By Theorem \ref{teo1}, it is sufficient to prove that there exists $n$ big enough such that $\CVc^{Pss}\subset \overline{\mt U}_n$. By \cite[Proposition 1.3.3]{Sc}, it is enough to show that there exists an integer $n_\star$ such that for every $n\geq n_\star$ and for every pair $(C,\mt E)\in\CVc^{Pss}$, the following conditions hold 
\begin{enumerate}[(i)]
\item $H^1\left(C,\mt E(n)\right)=0$.
\item For any maximal destabilizing chain $R\subset C$, we have 
$H^1\left(R,\mt I_{D}\mt E_{R^c}(n)\right)=0$, where $D:=R\cup R^c$.
\item For any maximal destabilizing chain $R\subset C$, the homomorphism
$
H^0\left(R,\mt I_{D}\mt E_{R^c}(n)\right)\lra \left(\mt I_{D}\mt E_{R^c}\right)\Big/\left( \mt I^2_{D}\mt E_{R^c}\right)
$
is surjective, where $D:=R\cup R^c$.
\item For any $x\in \widehat{C}:=C \backslash C_{\text{exc}}$, the homomorphism
$
H^0\left(C,\mt I_{C_{\text{exc}}}\mt E\right)\lra \mt E_{\widehat C}\Big/\left(\mt I^2_x\mt E_{\widehat C}\right)
$
is surjective.
\item For any $p\neq q\in\widehat{C}$, the evaluation homomorphism
$
H^0\left(C,\mt I_{C_{\text{exc}}}\mt E\right)\lra \mt E_{\{p\}}\oplus \mt E_{\{q\}}
$
is surjective.
\end{enumerate}
We will show that for each above condition, there exists an integer big enough such that the condition is satisfied. By taking the maximum among these five integers, we will have the assertion.

(i). By \cite{P96}, the set of P-semistable torsion free sheaves with Euler characteristic $d+r(1-g)$ and uniform rank $r$ on stable curves of genus $g$ is bounded. So, there exists an integer $n_1$ such that for every $n\geq n_1$, we have $H^1(C,\mt E(n))=H^1(C^{st},\pi_*\mt E(n))=0$, for every pair $(C,\mt E)$ in $\CVc^{P(s)s}$. In particular, for any pair $(C,\mt E)\in \CVc^{Pss}$ and $n\geq n_1$, the bundle $\mt E(n)$ satisfies (i).

(ii). Let $R\subset C$ be a subcurve obtained as the union of some maximal destabilizing chains. We set, as usual, $\widetilde C:=R^c$ and $D:=|\widetilde C\cap R|$, $D^{st}:=\pi(D)$ with their reduced structure. Consider the exact sequence
$$
0\lra\mt  I_{D^{st}}(\pi_*\mt E)\lra \pi_*\mt E\lra (\pi_*\mt E)_{D^{st}}\lra 0.
$$
The cokernel is a torsion sheaf and $\chi((\pi_*\mt E)_{D^{st}})=h^0((\pi_*\mt E)_{D^{st}})\leq 2rN\leq 6r(g-1)$, where $N$ is the number of nodes on $C^{st}$. The second inequality comes from the fact that a stable curve of genus $g$ can have at most $3g-3$ nodes. By boundness of P-(semi)stable shaves, using the theory of relative Quot schemes, we can construct a quasi-compact scheme, which is the fine moduli space for the pairs $(X,q:\mt P\to \mt F)$, where $X$ is a stable curve of genus $g$, $q$ is a surjective morphism of sheaves on $X$, $\mt P$ is a P-semistable torsion free of rank $r$ and Euler characteristic $d+r(1-g)$ and $\mt F$ is a sheaf with constant Hilbert polynomial less or equal than $6r(g-1)$.

In particular, there exists an integer $n_2$ such that for every $n\geq n_2$, for any pair $(C,\mt E)$ in $\CVc^{Pss}$ and any collection $R$ of maximal destabilizing chains in $C_{\text{exc}}$, the sheaf $\mt I_{D^{st}}(\pi_*\mt E)\otimes\omega_{C^{st}}^n=\mt I_{D^{st}}(\pi_*\mt E(n))$ is generated by global sections and $H^1(C^{st},\mt I_{D^{st}}(\pi_*\mt E(n)))=0$. As in the proof of Corollary \ref{NSrmk}, we have
$$
\mt I_{D^{st}}(\pi_*\mt E(n))\cong \pi_*(\mt I_{D}\mt E_{\widetilde C}(n))=\pi_*(\mt I_R\mt E(n)).
$$
Observe that $H^i(C^{st}, \pi_*(\mt I_{D}\mt E_{\widetilde C}(n)))=H^i(\widetilde C,\mt I_{D}\mt E_{\widetilde C}(n))$ for $i=0,1$. In particular, for any pair $(C,\mt E)\in \CVc^{Pss}$ and $n\geq n_2$, the bundle $\mt E(n)$ satisfies (ii).

(iii). Suppose that $R$ is a maximal destabilizing chain. By (ii), for $n\geq n_1$ the sheaf $\pi_*\left(\mt I_{D}\mt E_{\widetilde C}(n)\right)$ is generated by global sections. In particular, both the horizontal arrows
$$
\xymatrix@=0.5em{
H^0\left(C^{st},\pi_*\left(\mt I_{D}\mt E_{\widetilde C}(n)\right)\right)\ar[rr]\ar@{=}[d]&&\pi_*\left(\mt I_{D}\mt E_{\widetilde C}(n)\right)_{D^{st}}\ar@{=}[d]\\
H^0\left(C,\mt I_D\mt E_{\widetilde C}\mt E(n)\right)\ar[rr]&&\left(\mt I_{D}\mt E_{\widetilde C}(n)\right)_D\ar@{=}[r]&\left(\mt I_{D}\mt E_{\widetilde C}\right)\Big/\left( \mt I^2_{D}\mt E_{\widetilde C}\right)
}
$$
are surjective. Thus, for any pair $(C,\mt E)\in \CVc^{Pss}$ and $n\geq n_2$, the bundle $\mt E(n)$ satisfies (iii).

(iv), (v). For the rest of the proof $R$ will be the exceptional curve $C_{\text{exc}}$. Set $\mt G(n):=\mt I_{D}\mt E_{\widetilde C}(n)$. Observe that $\mt G(n)=\mt G\otimes\omega_C^n$. Let $p$ and $q$ be (not necessarily distinct) points on $\widehat C=C\backslash R$. Consider the exact sequence of sheaves on $\widetilde C$
$$
0\lra \mt I_{p,q}\mt G(n)\lra\mt G(n)\lra \mt G(n)/\mt I_{p,q}\mt G(n)\lra 0,
$$
where, when $p=q$, we denote by $\mt I_{p,p}\mt G$ the sheaf $\mt I^2_{p}\mt G$. If we show that there exists $n_2$ such that for every $n\geq n_2$ and any pairs $(C,\mt E)$ in $\CVc^{Pss}$, we have $$H^1(\widetilde C,\mt I_{p,q}\mt G(n))=H^1(C^{st},\pi_*(\mt I_{p,q}\mt G(n)))=0,$$ then for any pair $(C,\mt E)\in \CVc^{Pss}$ and $n\geq n_2$, the bundle $\mt E(n)$ satisfies both the conditions (iv) and (v). We have already shown in (ii) that the pairs $(C^{st},\pi_*\mt G(n))$ are bounded. The sheaf $G/\mt I_{p,q}\mt G(n)$ is torsion and its Euler characteristic is $2r$ (it does not depend on the choice of $p$ and $q$). Arguing as in (ii), we conclude that there exists $n_3$ such that for very pair $(C,\mt E)\in\CVc^{P(s)s}$ and any $n\geq n_3$, we have $H^1(C^{st},\pi_*(\mt I_{p,q}\mt G))=0$.

Thus, if we define $n_\star:=\max\{n_1,n_2,n_3\}$, then $\CVc^{P{s}s}\subset \overline{\mt U}_n$ for every $n\geq n_\star$.
\end{proof}

In \cite{Sc}, Schmitt proved that there exists an open substack of $\CVc^{Pss}$, admitting a a normal projective variety $\CU$ as good moduli space.  Before to introduce the points of this stack, we need to recall some facts from \emph{loc.cit.}.

Let $(C,\mt E)$ be a pair in $\CVc^{Pss}$, $\mt F_\bullet:=\{\mt F_1\subset\ldots\subset\mt F_l\}$  be a Jordan-Holder filtration of $\pi_*\mt E$ and $\alpha_\bullet:=\{\alpha_1,\ldots,\alpha_l\}$ be positive rational numbers. We will associate to this datum a polynomial.

As in the proof of Proposition \ref{qc}, let $n_\star$ be an integer such that $\CVc^{Pss}\subset\overline{\mt U}_n$ for every $n\geq n_\star$. Set $V_n:=H^0(C,\mt E(n))$ and let $B_n$ be a basis for $V_n$ compatible with the filtration $H^0(\mt F_\bullet):=\{H^0(\mt F_1(n))\subset\ldots\subset H^0(\mt F_k(n))\}$. We define a one-parameter subgroup $\lambda^{\alpha_\bullet,\mt F_\bullet}_n$ of $GL(V_n)$, given, with respect to the basis $B_n$, by the weight vector
$$
\sum_{i=1}^{\dim V_n-1}\alpha_i(\underbrace{h^0(\mt F_i(n))-\dim V_n,\ldots,h^0(\mt F_i(n))-\dim V_n}_{h^0(\mt F_i(n))},\underbrace{h^0(\mt F_i(n)), \ldots,h^0(\mt F_i(n))}_{h^0(\mt F_i(n))-\dim V_n}).
$$
Fix a surjective morphism $q:H^0(C,\mt E)\otimes\oo_C\to\mt E$ and set $\mt L_n:=\det\mt E(n)$. Consider the morphism induced by $q$
$$
p:S^m\bigwedge^rV_n\rightarrow H^0(C,\mt L_n^m),
$$
where $S^m$ denote the symmetric $m$-th product. It is surjective for $m$ big enough. The action of the the one-parameter subgroup $\lambda^{\alpha_\bullet,\mt F_\bullet}_n$ on $V_n$ induces an action on $S^m\bigwedge^rV_n$. Let $P_{\alpha_\bullet,\mt F_\bullet}(m,n)$ be the minimum among the sums of the weights of the elements of the basis $S^m\bigwedge^r B_n$ of $S^m\bigwedge^rV_n$ which induce, via $p$, a basis of $H^0(C,\mt L_n^m)$. When $n,m$ are big enough, $P_{\alpha_\bullet,\mt F_\bullet}(m,n)$ is a polynomial in the variable $m$ (see \cite[Recall 1.5]{T98}).

\begin{defin}\label{defhsemtrue}\cite[Definition 2.2.10]{Sc}. A pair $(C,\mt E)$ in $\CVc^{Pss}$ is called H-(semi)stable if for every Jordan-Holder filtration $\mt F_\bullet$ of $\pi_*\mt E$ and every vector $\alpha_\bullet$ of positive rational numbers, there exists an index $n^*$, such that for any $n\geq n^*$
$$
P_{\alpha_\bullet,\mt F_\bullet}(m,n)\underset{(\leq)}{<}0\text{ as polynomial in } m.
$$
We will say that a H-semistable pair is \emph{strictly H-semistable} if it is not H-stable. We denote by $\CVc^{H(s)s}$ the open substack in $\CVc^{Pss}$ of H-(semi)stable pairs.
\end{defin}
A priori the H-semistability depends on the choice of a basis for the global sections of the vector bundle. However, it can be shown that it depends only on the curve and on the vector bundle.

The main result in \cite{Sc} is the existence of a moduli space for these pairs.
\begin{teo}\cite[Theorem 2.3.1]{Sc}\label{schmitt} The moduli stack $\CVc^{Hss}$ of H-semistable pairs admits a normal projective variety $\CU$ of dimension $(r^2+3)(g-1)+1$ as good moduli space.
\end{teo}

\begin{rmk}\label{defhsem} An H-semistable point has the following properties: let, as usual, $\pi:C\rightarrow C^{st}$ be the stabilization morphism and $(C,\mt E)$ is a pair in $\CVc$.
\begin{enumerate}[(i)]
\item Suppose that $C$ is smooth. Then $(C,\mt E)$ is $H$-(semi)stable if and only if $(C,\mt E)$ is $P$-(semi)stable (see \cite{Sceq}). In this case we will say just $(C,\mt E)$ is \emph{(semi)-stable}.
\item We have a sequence of implications: $(C,\mt E)$ P-stable $\Rightarrow$ $(C,\mt E)$ H-stable $\Rightarrow$ $(C,\mt E)$ H-semistable $\Rightarrow$ $(C,\mt E)$ P-semistable.
\end{enumerate}
\end{rmk}

Our definition of H-semistable pairs is slightly different from the Schmitt's one. In his construction Schmitt just requires that the vector bundle must be admissible, but not necessarily balanced. The next lemma proves that the vector bundles appearing in his construction are indeed also properly balanced. So our definition agrees with the one of Schmitt.

\begin{lem}If $(C^{st},\pi_*\mt E)$ is P-semistable then $\mt E$ is properly balanced.
\end{lem}

\begin{proof}By considerations above, we must prove that $\mt E$ is balanced. By Lemma \ref{balcon}(ii), we have to prove that for any connected subcurve $Z\subset C$ such that $Z^c$ is connected, we have
$$
\frac{\chi(\mt F_Z)}{\omega_Z}\leq \frac{\chi(\mt E)}{\omega_C},
$$
where $\mt F_Z$ is the subsheaf of $\mt E_Z$ of sections that vanishes on $Z\cap Z^c$. Observe that $\mt F_Z$ is also a subsheaf of $\mt E$. The hypothesis and the fact that the push-forward is left exact imply
$$
\frac{\chi(\pi_*\mt F_Z)}{\omega_{Z^{st}}}\leq \frac{\chi(\pi_*\mt E)}{\omega_{C^{st}}}=\frac{\chi(\mt E)}{\omega_C},
$$
where $Z^{st}$ is the reduced subcurve $\pi(Z)\subset C^{st}$. It is clear that $\omega_{Z^{st}}=\omega_Z$. We have an inclusion of vector spaces $H^1\left(Z^{st},\pi_*\mt F\right)\subset H^1\left(Z,\mt F\right)$. This implies $\chi(\mt F_Z)\leq\chi(\pi_*\mt F_Z)$, concluding the proof.
\end{proof}

\section{Preliminaries about line bundles on stacks.}\label{linechow}

\subsection{Picard group and Chow groups of a stack.}\label{boh}

We will recall the definitions and some properties of the Picard group and the Chow group of an Artin stack. Some parts contain overlaps with \cite[Section 2.9]{MV}. Let $\mathcal X$ be an Artin stack locally of finite type over $k$.
\begin{defin}\cite[p.64]{Mum65} A \emph{line bundle} $\mathcal L$ on $\mathcal X$ is the data consisting of a line bundle $\mathcal L(F_S)$ on $S$ for every scheme $S$ and morphism $F_S:S\rightarrow\mathcal X$ such that:
\begin{itemize}
\item For any commutative diagram
$$
\xymatrix{
S \ar[rd]_{F_S}\ar[rr]^{f} & & T\ar[ld]^{F_T}\\
& \mathcal X &
}
$$
there is an isomorphism $\phi(f):\mathcal L(F_S)\cong f^*\mathcal L(F_T)$.
\item For any commutative diagram
$$
\xymatrix{
S \ar[rd]_{F_S}\ar[r]^{f} & T\ar[d]^{F_T}\ar[r]^g & Z\ar[ld]^{F_Z}\\
& \mathcal X &
}
$$
we have the following commutative diagram of isomorphisms
$$
\xymatrix{
\mathcal L(F_S)\ar[r]^{\phi(f)}\ar[d]^{\phi(g\circ f)} & f^*\mathcal L(F_T)\ar[d]^{f^*\phi(g)}\\
(g\circ f)^*\mathcal L(F_Z)\ar[r]^{\cong} & f^* g^*\mathcal L(F_Z).
}
$$
\end{itemize}
The abelian group of isomorphism classes of line bundles on $\mathcal X$ is called the \emph{Picard group} of $\mathcal X$ and is denoted by $\text{Pic}(\mathcal X)$.
\end{defin}

\begin{rmk}The definition above is equivalent to have a locally free sheaf of rank $1$ for the site lisse-\'etale (\cite[Proposition 1.1.1.4.]{Br1}).
\end{rmk}

If $\mathcal X$ is a quotient stack $[X/G]$, where $X$ is a scheme of finite type over $k$ and $G$ a group scheme of finite type over $k$, then $\text{Pic}(\mathcal X)\cong \text{Pic}(X)^G$  (see \cite[Chap. XIII, Corollary 2.20]{ACG11}), where $\text{Pic}(X)^G$ is the group of isomorphism classes of $G$-linearized line bundles on $X$.\vspace{0.2cm}

In \cite[Section 5.3]{EG98} (see also \cite[Definition 3.5]{Edi12}), Edidin and Graham introduce the operational Chow groups of an Artin stack $\mathcal X$, as generalization of the operational Chow groups of a scheme.

\begin{defin}A \emph{Chow cohomology class} $c$ on $\mathcal X$ is the data consisting of an element $c(F_S)$ in the operational Chow group $A^*(S)=\oplus A^i(S)$ for every scheme $S$ and morphism $F_S:S\rightarrow\mathcal X$ such that for any commutative diagram
$$
\xymatrix{
S \ar[rd]_{F_S}\ar[rr]^{f} & & T\ar[ld]^{F_T}\\
& \mathcal X &
}
$$
we have $c(F_S)= f^*c(F_T)$. The abelian group consisting of all the $i$-th Chow cohomology classes on $\mathcal X$ together with the operation of sum is called the \emph{$i$-th Chow group} of $\mathcal X$ and is denoted by $A^i(\mathcal X)$.
\end{defin}

If $\mathcal X$ is a quotient stack $[X/G]$, where $X$ is a scheme of finite type over $k$ and $G$ a group scheme of finite type over $k$, then $A^i(\mathcal X)\cong A^i_G(X)$ (see \cite[Proposition 19]{EG98}), where $A^i_G(X)$ is the operational equivariant Chow group defined in \cite[Section 2.6]{EG98}. We have a homomorphism of groups $c_1:\text{Pic}(\mathcal X)\rightarrow A^1(\mathcal X)$ defined by the first Chern class.\vspace{0.2cm}

The next theorem summarizes some results on the Picard group of a smooth stack, which will be useful for our purposes.

\begin{teo}\label{picchow}Let $\mt X$ be a (not necessarily quasi-compact) smooth Artin stack over $k$. Let $\mt U\subset\mt X$ be an open substack.
\begin{itemize}
\item[\textit{(i)}] The restriction map $\Pic(\mt X)\rightarrow \Pic(\mt U)$ is surjective.
\item[\textit{(ii)}] If $\mt X\backslash\mt U$ has codimension $\geq 2$ in $\mt X$, then $\Pic(\mt X)=\Pic(\mt U)$.
\end{itemize}
Suppose that $\mt X=[X/G]$, where $G$ is an algebraic group and $X$ is a smooth quasi-projective variety with a $G$-linearized action.
\begin{itemize}
\item[\textit{(iii)}] The first Chern class map $c_1:\Pic(\mathcal X)\rightarrow A^1(\mathcal X)$ is an isomorphism.
\item[\textit{(iv)}] If $\mt X\backslash\mt U$ has codimension $1$ with irreducible components $\mt D_i$, then we have an exact sequence
$$
\bigoplus_i\mathbb{Z}\langle \oo_{\mathcal X}(\mt D_i)\rangle\lra \Pic(\mt X)\lra \Pic(\mt U)\lra 0.
$$
\end{itemize}
\end{teo}

\begin{proof}The first two points are proved in \cite[Lemma 7.3]{BH12}. The third one follows from \cite[Corollary 1]{EG98}. Let $U$ be the open subscheme of $X$ such that $\mt U=[U/G]$. Set $n:=\dim\mt X$. By \cite[Theorem 1]{EG98}, we can identify $\Pic(\mt X)$ and $\Pic(\mt U)$ with the equivariant Chow groups $A^G_{n-1}(X)$ and $A^G_{n-1}(U)$, respectively. By this identification and by \cite[Proposition 5]{EG98}, we have an exact sequence:
$$
A^G_{n-1}(X\backslash U)\xrightarrow{f} \Pic(\mt X)\to \Pic(\mt U)\to 0
$$
Since $\mt X\backslash \mt U$ has dimension $n-1$, $A^G_{n-1}(X\backslash U)$ is freely generated by the divisors $\mt D_i$. So, we have that $\text{Im}f$ is generated by $\bigoplus_i\mathbb{Z}\langle \oo_{\mathcal X}(\mt D_i)\rangle$, concluding the proof.
\end{proof}

\subsection{Determinant of cohomology and Deligne pairing.}\label{dcdp}

There exist two methods to produce line bundles on a stack parametrizing nodal curves with some extra-structure (as our stacks): the determinant of cohomology and the Deligne pairing. We refer to \cite{KM76}, \cite{Del87} for a complete treatment of the determinant of cohomology and Deligne pairing, respectively. Here, we will recall the main properties of these construction, following the presentation given in \cite[Chap. XIII, Sections 4 and 5]{ACG11} and the resume in \cite[Section 2.13]{MV}. 

Let $p:C\rightarrow S$ be a family of nodal curves. Given a coherent sheaf $\mt F$ on $C$ flat over $S$, the \emph{determinant of cohomology} of $\mathcal F$ is a line bundle $d_{p}(\mt F)\in \Pic(S)$ defined as it follows. Locally on $S$, there exists a relatively ample divisor $D$ for the family $C\to S$, which gives us an exact sequence
$$
0\to p_*\mt F\to p_*\mt F(D)\to p_*\mt F(D)_D\to R^1p_*\mt F\to 0,
$$
such that $(V_D^0\rightarrow V_D^1):=(p_*\mt F(D)\to p_*\mt F(D)_{D})$ is a complex of vector bundles. We set 
$$
d_{p}(\mt F):=\det V_D^0\otimes (\det V_D^1)^{-1}.
$$
The definition does not depend on the choice of the divisor $D$ (see \cite[pp. 356-357]{ACG11}), in particular, this defines a line bundle globally on $S$. The proof of the next theorem can be found in \cite[Chap. XIII, Section 4]{ACG11}.
\begin{teo}\label{detcoh}Let $p:C\rightarrow S$ be a family of nodal curves and let $\mathcal F$ be a coherent sheaf on $C$ flat on $S$.
\begin{enumerate}[(i)]
\item If $S$ is smooth, the first Chern class of $d_p(\mt F)$ is equal to
$$
c_1(d_{p}(\mt F))=c_1(p_!(\mt F)):=c_1(p_*(\mt F))-c_1(R^1p_*(\mt F)).
$$
\item Given a cartesian diagram
$$
\xymatrix{
C\times_S T\ar[r]^g\ar[d]_q & C\ar[d]^p\\
T\ar[r]^f & S
}
$$
we have a canonical isomorphism
$$
f^*d_p(\mt F)\cong d_q(g^*\mt F).
$$
\end{enumerate}
\end{teo}

\begin{proof} Locally on $S$, $c_1(d_{p}(\mt F))=c_1(\det V_D^0)-c_1 ((\det V_D^1))=c_1(p_!(\mt F))$. The latter equality comes from the fact that $c_1$ takes values in the Grothendieck group of bounded complexes of coherent sheaves on $S$. Then (i) holds. The second assertion can be deduced from the independence of the choice of the divisor $D$ in \cite[pp. 356-357]{ACG11} and the fact that $g^{-1}(D)$ is relatively ample with respect to $q$.
\end{proof}

Given two line bundles $\mt M$ and $\mt L$ over a family of nodal curves $p:C\rightarrow S$, the \emph{Deligne pairing} of $\mt M$ and $\mt L$ is a line bundle $\langle\mt M,\mt L\rangle_p\in Pic(S)$ which can be defined as
$$
\langle\mt M,\mt L\rangle_p:=d_p(\mt M\otimes\mt L)\otimes d_p(\mt M)^{-1}\otimes d_p(\mt L)^{-1}\otimes d_p(\oo_{C}).
$$
The proof of the next theorem can be found in \cite[Chap. XIII, Section 5]{ACG11}.
\begin{teo}\label{delpair}Let $p:C\rightarrow S$ be a family of nodal curves.
\begin{enumerate}[(i)]
\item If $S$ is smooth, the first Chern class of $\langle\mt M,\mt L\rangle_p$ is equal to
$$
c_1(\langle\mt M,\mt L\rangle_p)=p_*(c_1(\mt M)\cdot c_1(\mt L)).
$$
\item Given a cartesian diagram
$$
\xymatrix{
C\times_S T\ar[r]^g\ar[d]_q & C\ar[d]^p\\
T\ar[r]^f & S
}
$$
we have a canonical isomorphism
$$
f^*\langle\mt M,\mt L\rangle_p\cong \langle g^*\mt M,g^*\mt L\rangle_q
$$
\end{enumerate}
\end{teo}

\begin{proof}For the first equality see \cite[p. 376]{ACG11}. The second one is a consequence of Theorem \ref{detcoh}(ii).
\end{proof}

\begin{rmk}\label{rmk2}By the functoriality of the determinant of cohomology and of the Deligne pairing, we can extend their definitions to the case when we have a representable, proper and flat morphism of Artin stacks such that the geometric fibers are nodal curves.
\end{rmk}

\subsection{Picard group of $\CMg$.}\label{cmg}

The universal family $\overline{\pi}:\overline{\mt M}_{g,1}\rightarrow \CMg$ is a representable, proper, flat morphism with stable curves as geometric fibers. In particular we can define the relative dualizing sheaf $\omega_{\overline{\pi}}$ on $\overline{\mathcal M}_{g,1}$ and taking the determinant of cohomology $d_{\overline{\pi}}(\omega^n_{\overline{\pi}})$ we obtain line bundles on $\CMg$. The line bundle $\Lambda:=d_{\overline{\pi}}(\omega_{\overline{\pi}})$ is called the \emph{Hodge line bundle}. 

Let $C$ be a stable curve and for every node $x$ of $C$, consider the partial normalization $C'$ at $x$. If $C'$ is connected then we say $x$ is a node of type $0$, if $C'$ is the union of two connected curves of genus $i$ and $g-i$, with $i\leq g-i$ (for some $i$), then we say that $x$ is a node of type $i$.
The boundary $\CMg\slash \Mg$ decomposes as union of irreducible divisors $\delta_i$ for $i=0,\ldots,\lfloor g/2\rfloor$, where $\delta_i$ parametrizes (as stack) the stable curves with a node of type $i$. The generic point of $\delta_0$ is an irreducible curve of genus $g$ with exactly one node, the generic point of $\delta_i$ for $i=1,\ldots,\lfloor g/2\rfloor$ is a stable curve formed by two irreducible smooth curves of genus $i$ and $g-i$ meeting in exactly one point. We set $\delta:=\sum\delta_i$. By Theorem \ref{picchow} we can associate to any $\delta_i$ a unique (up to isomorphism) line bundle $\oo(\delta_i)$. We set $\oo(\delta)=\bigotimes_i\oo(\delta_i)$.

The proof of the next results for $g\geq 3$ can be found in \cite[Theorem. 1]{AC87} based upon a result of \cite{Har83}. If $g=2$ see \cite{Vis98} for $\Pic(\mt M_2)$ and \cite[Proposition 1]{Cor07} for $\Pic(\overline{\mt M}_2)$.

\begin{teo}\label{picmg}Assume $g\geq 2$. Then
\begin{enumerate}[(i)]
\item $\Pic(\Mg)$ is freely generated by the Hodge line bundle, except for $g=2$ in which case we add the relation $\Lambda^{10}=\oo_{\mt M_2}$.
\item $\Pic(\CMg)$ is freely generated by the Hodge line bundle and the boundary divisors, except for $g=2$ in which case we add the relation $\Lambda^{10}=\oo\left(\delta_0+2\delta_1\right)$.
\end{enumerate}
\end{teo}

\subsection{Picard Group of $\Jc$.}\label{jc}

The universal family $\pi:\mt Jac_{d,g,1}\rightarrow \Jc$ is a representable, proper, flat morphism with smooth curves as geometric fibers. In particular, we can define the relative dualizing sheaf $\omega_{\pi}$ and the universal line bundle $\mt L$ on $\mt Jac_{d,g,1}$. Taking the determinant of cohomology $\Lambda(n,m):=d_{\pi}(\omega^n_{\pi}\otimes\mt L^m)$, we obtain several line bundles on $\Jc$. 

The proof of next theorem can be found in \cite[Theorem A(i) and Notation 1.5]{MV}, based upon a result of \cite{Kou91}.

\begin{teo}\label{picjac}Assume $g\geq 2$. Then $\Pic(\Jc)$ is freely generated by $\Lambda(1,0)$, $\Lambda(1,1)$ and $\Lambda(0,1)$, except for $g=2$ in which case we add the relation $\Lambda(1,0)^{10}=\oo_{\Jc}$.
\end{teo}

\subsection{Picard Groups of the fibers.}\label{fibre}

Fix now a smooth curve $C$ with a line bundle $\mathcal L$ of degree $d$. Let $\Vl$ be the stack whose objects over a scheme $S$ are the pairs $(\mt E,\varphi)$ where $\mt E$ is a vector bundle of rank $r$ on $C\times S$ and $\varphi$ is an isomorphism between the line bundles $\det\mt E$ and $\mt L\boxtimes \oo_S$. A morphism between two objects over $S$ is an isomorphism of vector bundles compatible with the isomorphism of determinants. 

Now we recall some properties of this stack. For more details, we refer to \cite{H} and references therein. It is known that $\Vl$ is an irreducible smooth Artin stack of dimension $(r^2+1)(g-1)$. We denote by $\Vl^{ss}$ the open substack of (semi)stable vector bundles. Since the set of isomorphism classes of semistable vector bundles on $C$ is bounded, the stack $\Vl^{ss}$ is quasi-compact. Consider the set of equivalence classes (defined as in Section \ref{Sc}) of semistable vector bundles over the curve $C$  with determinant isomorphic to $\mt L$. There exists a normal projective variety $U_{\mt L,C}$, which is a moduli space for this set. Observe that the stack $\Vl$ is the fiber of the determinant morphism $det:\Vc\to\Jc$ with respect to the $k$-point $(C,\mt L)$.

\begin{teo}\label{fibers}Let $C$ be a smooth curve with a line bundle $\mathcal L$. Let $\mathcal E$ be the universal vector bundle over $\pi:\Vl\times C\to\Vl$ of rank $r$ and degree $d$. Then:
\begin{enumerate}[(i)]
\item We have natural isomorphisms induced by the restriction
	$$\mathbb Z=\mathbb Z\langle d_{\pi}(\mathcal E )\rangle\cong Pic(\Vl)\cong Pic(\Vl^{ss}).$$
\item $U_{\mathcal L,C}$ is a good moduli space for $\Vl^{ss}$.
\item The good moduli morphism $\Vl^{ss}\rightarrow U_{\mt L, C}$ induces an exact sequence of groups
	$$
	0\rightarrow Pic(U_{\mathcal L,C}) \rightarrow Pic(\Vl^{ss})\rightarrow\mathbb Z/\tfrac{r}{n_{r,d}}\mathbb Z\rightarrow 0
	$$
where the second map sends $d_{\pi}(\mt E)^k$ to $k$.
\end{enumerate}
\end{teo}

\begin{proof} Part (i) is proved in \cite[Theorem 3.1 and Corollary 3.2]{H}. Part (ii) follows from \cite[Section 2]{H}. Part (iii) is proved in \cite[Theorem 3.7]{H}.
\end{proof}

\begin{rmk}\label{caso2schifo}By \cite[Corollary 3.8]{H}, the variety $U_{\mt L,C}$ is locally factorial. Moreover, except the cases when $g=r=2$ and $\deg\mt L$ is even, the closed locus of strictly semistable vector bundles has codimension $\geq 2$
	. So, by Theorem \ref{picchow}, when $(r,g,d)\neq (2,2,0)\in \mathbb Z\times\mathbb Z\times(\mathbb Z/2\mathbb Z)$ we have that $\Pic(\Vl)\cong \Pic(\mathcal Vec^s_{=\mathcal L,C})$ and, since $U_{\mt L,C}$ is locally factorial, $\Pic(U_{\mathcal L,C})\cong\Pic(U^s_{\mathcal L,C})$.
\end{rmk}

\subsection{Boundary divisors.}\label{boundiv}

The aim of this section is to study the boundary $\widetilde\delta:=\CVc\backslash\Vc$. By definition, it parameterizes properly balanced vector bundles on singular curves. Our first result is that it has codimension one in $\CVc$.

\begin{prop}\label{ncross}The boundary $\widetilde\delta\subset\CVc$ is a normal crossing divisor.
\end{prop}

\begin{proof}We have to show that the pull-back of $\widetilde\delta$ to $H_n$ is a normal crossing divisor for any $n$. Here $H_n$ are the schemes defined at page \pageref{Hn}, which are smooth by Proposition \ref{finteo2}. With abuse of notation, we denote the pull-back along $H_n$ by $\widetilde{\delta}$.
	
By the results of \S\ref{locstr}, we have that morphism of deformation functors $\text{Def}_h\to\text{Def}_C$ is formally smooth. In particular, locally at a point $h:=[C\hookrightarrow Gr(V_n,r)]$, the scheme $H_n$ looks like
$$
\Spf\, k\llbracket x_1,\ldots,x_{3g-3},y_1,\ldots,y_{\dim H_n+3-3g}\rrbracket,
$$
where the first $3g-3$ coordinates control the deformations of the curve $C$. Furthermore, if $N$ is the number of the nodes of $C$, we can choose local coordinates such that $x_1,\ldots,x_N$ correspond to smoothing the nodes. For such a choice of the coordinates, we have that the ideal defining $\widetilde\delta$, locally at $h$, is $(\prod_{i=1}^Nx_i)$, proving the assertion.
\end{proof}

We want to describe the irreducible components of $\widetilde{\delta}$. Before to introduce them, we need to do some observations. Set-theoretically, we have $$\widetilde{\delta}=\overline{\phi}_{r,d}^{-1}(\delta)=\bigcup_{i=1}^{\lfloor g/2\rfloor}\overline{\phi}_{r,d}^{-1}(\delta_i),$$
where $\overline{\phi}_{r,d}:\CVc\to\CMg$ is the forgetful morphism. In particular:
\begin{enumerate}[(a)]
	\item the generic point of an irreducible component of $\widetilde{\delta}$ over $\delta_0$ consists of an irreducible nodal curve with just one node and a properly balanced vector bundle of degree $d$ and rank $r$. 
	\item the generic point of an irreducible component over $\delta_i$, with $i\neq 0$, is a curve composed by two irreducible smooth curves of genus $i$ and $g-i$ meeting in one point and a properly balance vector bundle of degree $d$ and rank $r$.
\end{enumerate}

In the first case, any vector bundle is properly balance because there are no subcurves. In the second case, since the curve is stable any vector bundle is admissible, so we need just to check that the multidegree satisfies the basic inequality. 

Moreover, the multidegree does not change if we do not change the number of nodes of the curve. More precisely, let $(C\to S,\mt E)$ be a family of pairs in $\CVc$, with $S$ connected. Assume that $C\to S$ is obtained by gluing nodally two families of smooth curves $C_1\to S$, $C_2\to S$ of lower genus, such that each fibre defines a generic pair (as in (b)) in $\delta_i$ for some $i\neq 0$. Then, the degree of the restriction of $\mt E$ to $C_1$, $C_2$ must be constant on each fibre.

In other words, once we fix the dual graph of the curve (in the sense of \cite[p. 88]{ACG11}), distinct choices of the multidegree defines distinct irreducible components of $\widetilde\delta$. With these observations, the next definition makes sense.

\begin{defin}\label{boundef}The \emph{boundary divisors} of $\CVc$ are:
	\begin{itemize}
		\item $\widetilde\delta_0:=\widetilde\delta_0^0$ is the divisor whose generic point is an irreducible curve $C$ with just one node and $\mt E$ is a vector bundle of degree $d$,
		\item if $k_{r,d,g}|2i-1$ and $0< i < g/2$:\\
		$\widetilde\delta_i^{j}$ for $0\leq j\leq r$ is the divisor whose generic point is a curve $C$ composed by two irreducible smooth curves $C_1$ and $C_2$ of genus $i$ and $g-i$ meeting in one point and $\mt E$ a vector bundle over $C$ with multidegree
		$$
		(\deg\mt E_{C_1},\deg\mt E_{C_2})=\left(d\frac{2i-1}{2g-2}-\frac{r}{2}+j,d\frac{2(g-i)-1}{2g-2}+\frac{r}{2}-j\right),
		$$
		\item if $k_{r,d,g}\nmid 2i-1$ and $0< i < g/2$:\\
		$\widetilde\delta_i^{j}$ for $0\leq j\leq r-1$ is the divisor whose generic point is a curve $C$ composed by two irreducible smooth curves $C_1$ and $C_2$ of genus $i$ and $g-i$ meeting in one point and $\mt E$ a vector bundle over $C$ with multidegree
		$$
		(\deg\mt E_{C_1},\deg\mt E_{C_2})=\left(\left\lceil d\frac{2i-1}{2g-2}-\frac{r}{2}\right\rceil +j,\left\lfloor d\frac{2(g-i)-1}{2g-2}+\frac{r}{2}\right\rfloor-j\right),
		$$
		\item if $g$ is even and $d+r$ is even:\\
		$\widetilde\delta_{\frac{g}{2}}^{j}$ for $0\leq j\leq \lfloor \frac{r}{2}\rfloor$ is the divisor whose generic point is a curve $C$ composed by two irreducible smooth curves $C_1$ and $C_2$ of genus $g/2$ meeting in one point and $\mt E$ a vector bundle over $C$ with multidegree
		$$
		(\deg\mt E_{C_1},\deg\mt E_{C_2})=\left(\frac{d-r}{2}+j,\frac{d+r}{2}-j\right),
		$$
		\item if $g$ is even and $d+r$ is odd:\\
		$\widetilde\delta_{\frac{g}{2}}^{j}$ for $0\leq j\leq \lfloor \frac{r-1}{2}\rfloor$ is the divisor whose generic point is a curve $C$ composed by two irreducible smooth curves $C_1$ and $C_2$ of genus $g/2$ meeting in one point and $\mt E$ a vector bundle over $C$ with multidegree
		$$
		(\deg\mt E_{C_1},\deg\mt E_{C_2})=\left(\left\lceil \frac{d-r}{2}\right\rceil +j,\left\lfloor \frac{d+r}{2}\right\rfloor-j\right).
		$$
	\end{itemize}
	If $0<i<g/2$ and $k_{r,d,g}|2i-1$ (resp. $g$ and $d+r$ even) we will call $\widetilde\delta_i^0$ and $\widetilde\delta_i^r$ (resp. $\widetilde\delta_{\frac{g}{2}}^0)$ \emph{extremal boundary divisors}. We will call \emph{non-extremal boundary divisors} the boundary divisors which are not extremal.
	
	By Theorem \ref{picchow}, we can associate to $\widetilde\delta_i^j$ a line bundle on $\overline{\mt U}_n$ for any $n$, which glue together to a line bundle $\oo(\widetilde\delta_i^j)$ on $\CVc$, we will call them \emph{boundary line bundles}. Moreover, if $\widetilde\delta_i^j$ is a (non)-extremal divisor, we will call $\oo(\widetilde\delta_i^j)$ \emph{(non)-extremal boundary line bundle}.
\end{defin}

It turns out that the boundary of $\CVc$ is the union of the above boundary divisors.
\begin{prop}\label{boundary}\noindent
\begin{enumerate}[(i)]
\item The irreducible components of $\widetilde\delta$ are $\widetilde\delta_i^j$ for $0\leq i\leq g/2$ and $j\in J_i$ where
$$J_i=
\begin{cases}
\{0\} &\mbox{if } i=0,\\
\{0,\ldots,r \} & \mbox{if } k_{r,d,g}|2i-1 \mbox{ and } 0<i<g/2,\\
\{0,\ldots,r-1\} & \mbox{if } k_{r,d,g}\nmid 2i-1 \mbox{ and }  0<i<g/2,\\
\{0,\ldots,\lfloor r/2\rfloor \}& \mbox{if g even, d+r even and } i=g/2.\\
\{0,\ldots,\lfloor (r-1)/2\rfloor \}& \mbox{if g even, d+r odd and } i=g/2.
\end{cases}
$$
\item Let $\overline{\phi}_{r,d}:\CVc\rightarrow\CMg$ be the forgetful map. For $0\leq i\leq g/2$, we have
$$
\overline{\phi}_{r,d}^*\oo(\delta_i)=\oo\left(\sum_{j\in J_i}\widetilde\delta_i^j\right).
$$
\end{enumerate}
\end{prop}

\begin{proof}(i). Let $\widetilde\delta^*$ be the locus of $\widetilde\delta$ of curves with exactly one node. As in \cite[Corollary 1.9]{DM69} we can prove that $\widetilde\delta^*$ is a dense smooth open substack in $\widetilde\delta$. The irreducible components of $\widetilde\delta$ are in bijection with the connected components of $\widetilde\delta^*$.

A direct computation of the basic inequalities for bundles on curves with just one node shows that given a pair $(C,\mt E)$ in $\widetilde\delta^*$, then the it must be contained in the generic locus of a unique divisor $\widetilde{\delta}_i^j$. In other words, we have a set-theoretical equality
$$
\overline{\phi}_{r,d}^{-1}(\delta_i)=\bigcup_{j\in J_i}\widetilde\delta_i^j,
$$
and $\delta_i^j=\delta_t^k$ if and only if $j=k$ and $i=t$. Now we are going to prove that they are irreducible. Setting $\widetilde\delta_i^{*j}:=\widetilde\delta^*\cap\widetilde\delta_i^j$, we see that $\widetilde\delta_i^j$ is irreducible if and only if $\widetilde\delta_i^{*j}$ is irreducible. It can be shown also that they are disjoint, i.e. $\widetilde\delta_i^{*j}\cap\widetilde\delta_t^{*k}\neq\emptyset$ if and only if $j=k$ and $i=t$.\\
Consider the forgetful map $\phi:\widetilde\delta_i^{*j}\rightarrow \delta_i^*$, where $\delta_i^*$ is the open substack of $\delta_i\subset\CMg$ of curves with exactly one node. In \S\ref{locstr}, we have seen that the morphism of Artin functors $\text{Def}_{(C,\mt E)}\to\text{Def}_C$ is formally smooth for any nodal curve. This implies that the map $\phi$ is smooth, in particular is open. Since $\delta_i^*$ is irreducible (see \cite[pag. 94]{DM69}), it is enough to show that the geometric fibers of $\phi$ are irreducible.\\
Assume $i\neq 0$. Let $C$ be a nodal curve with two irreducible components $C_1$ and $C_2$, of genus $i$ and $g-i$, meeting at a point $x$, this defines a geometric point $[C]\in\delta^*_i$. The fibre $\phi^*([C])$ is the moduli stack of vector bundles on $C$ of multidegree
$$
(d_1,d_2):=(\deg\mt E_{C_1},\deg\mt E_{C_2})=\left(\left\lceil d\frac{2i-1}{2g-2}-\frac{r}{2}\right\rceil+j,\left\lfloor d\frac{2(g-i)-1}{2g-2}+\frac{r}{2}\right\rfloor-j\right).
$$
The restriction over $C_1$ and $C_2$ induces a surjective smooth morphism of stacks
$$\phi^*([C])\to \mt Vec_{r,d_1,C_1}\times\mt Vec_{r,d_2,C_2},
$$
Where $Vec_{r,d_i,C_i}$ is the moduli stack of vector bundles of degree $d_i$ and rank $r$ over $C_i$. The geometric fibre isomorphic to the group $GL_r$.  The target is irreducible by \cite[Corollary A.5]{Ho10}) and so the fibre, then the same holds for $\phi^*([C])$. The same argument works for $\widetilde\delta_0$.

(ii). By part (i), for $0\leq i\leq g/2$ we have
$$
\overline{\phi}_{r,d}^*\oo(\delta_i)=\oo\left(\sum_{j\in J_i}a_i^j\widetilde\delta_i^j\right)
$$
where $a_i^j$ are integers. We have to prove that the coefficients are $1$. We can reduce to prove it locally on $\widetilde\delta$. The generic element of $\widetilde\delta$ is a pair $(C,\mt E)$ such that $C$ is stable with exactly one node and $\Aut(C,\mt E)=\mathbb{G}_m$. By Lemma \ref{locstrc}, locally at such $(C,\mt E)$, $\overline{\phi}_{r,d}$ looks like
$$
\left[\Spf\, k\llbracket x_1,\ldots,x_{3g-3},y_1,\ldots,y_{r^2(g-1)+1}\rrbracket/\mathbb{G}_m\right]\rightarrow \left[\Spf\, k\llbracket x_1,\ldots,x_{3g-3}\rrbracket/\Aut(C)\right].
$$
We can choose local coordinates such that $x_1$ corresponds to smoothing the unique node of $C$. For such a choice of the coordinates, we have that the equation of $\delta_i$ locally on $C$ is given by $(x_1=0)$ and the equation of $\widetilde\delta_i^j$ locally on $(C,\mt E)$ is given by $(x_1=0)$. Since $\overline{\phi}_{r,d}^*(x_1)=x_1$, the theorem follows.
\end{proof}

With an abuse of notation, we set $\widetilde\delta_i^j:=\nu_{r,d}(\widetilde\delta_i^j)$ for $0\leq i\leq g/2$ and $j\in J_i$, where $\nu_{r,d}:\CVc\rightarrow \CVr$ is the rigidification map. From the above proposition, we deduce the following
\begin{cor}\label{boundarycor} The following hold:
\begin{enumerate}
\item The boundary $\widetilde\delta:=\CVr\backslash\Vc$ of $\CVr$ is a normal crossing divisor, and its irreducible components are $\widetilde\delta_i^j$ for $0\leq i\leq g/2$ and $j\in J_i$.
\item For $0\leq i\leq g/2$, $j\in J_i$ we have
$
\nu_{r,d}^*\oo(\widetilde\delta_i^j)=\oo(\widetilde\delta_i^j)
$.
\end{enumerate}
\end{cor}

\subsection{Tautological line bundles.}\label{tautbun}

In this subsection, we will produce several line bundles on the stack $\CVc$ and we will study their relations in the rational Picard group of $\CVc$. Consider the universal curve $\overline{\pi}:\overline{\mt Vec}_{r,d,g,1}\rightarrow \CVc$. The stack $\overline{\mt Vec}_{r,d,g,1}$ has two natural sheaves, the dualizing sheaf $\omega_{\overline{\pi}}$ and the universal vector bundle $\mathcal E$. As explained in \S\ref{dcdp}, we can produce the following line bundles which will be called \emph{tautological line bundles}:
$$
\begin{array}{rcl}
K_{1,0,0} &:=&\langle \omega_{\overline{\pi}},\omega_{\overline{\pi}}\rangle,\\
K_{0,1,0}&:=&\langle \omega_{\overline{\pi}},\det\mt E\rangle,\\
K_{-1,2,0}&:=&\langle \det\mt E,\det\mt E\rangle,\\
\Lambda(m,n,l) & :=& d_{\overline{\pi}}(\omega_{\overline{\pi}}^m\otimes (\det \mt E)^n\otimes \mathcal E^l).
\end{array}
$$
With an abuse of notation, we will denote by the same symbols their restriction to any open substack of $\CVc$. By Theorems \ref{detcoh} and \ref{delpair}, we can compute the first Chern classes of the tautological line bundles:
$$
\begin{array}{rcccl}
k_{1,0,0}&:=&c_1(K_{1,0,0})&=&\overline{\pi}_*\left(c_1(\omega_{\overline{\pi}})^2\right),\\
k_{0,1,0}&:=&c_1(K_{0,1,0})&=&\overline{\pi}_*\left(c_1(\omega_{\overline{\pi}})\cdot c_1(\mt E)\right)\\
k_{-1,2,0}&:=&c_1(K_{-1,2,0})&=&\overline{\pi}_*\left(c_1(\mt E)^2\right)\\
\lambda(m,n,l)&:=&c_1(\Lambda(m,n,l))&=&c_1\left(\overline{\pi}_!\left(\omega_{\overline{\pi}}^m\otimes(\det\mt E)^n\otimes \mt E^l\right)\right)
\end{array}
$$
\begin{teo}\label{relations}The tautological line bundles on $\CVc$ satisfy the following relations in the rational Picard group $Pic(\CVc)\otimes \mathbb Q$.
\begin{enumerate}[(i)]
\item $K_{1,0,0}=\Lambda(1,0,0)^{12}\otimes \oo(-\widetilde\delta).$
\item $K_{0,1,0}=\Lambda(1,0,1)\otimes \Lambda(0,0,1)^{-1}=\Lambda(1,1,0)\otimes \Lambda(0,1,0)^{-1}.$
\item $K_{-1,2,0}=\Lambda(0,1,0)\otimes \Lambda(1,1,0)\otimes\Lambda(1,0,0)^{-2}.$
\item For $(m,n,l)$ integers we have:\begin{eqnarray*}
\Lambda(m,n,l) &=&\Lambda(1,0,0)^{r^l(6m^2-6m+1-n^2-l)-2r^{l-1}nl-r^{l-2}l(l-1)}\otimes\\
& & \otimes\Lambda(0,1,0)^{r^l\left(-mn+{n+1\choose 2}\right)+r^{l-1}l\left(n -m\right)+r^{l-2}{l\choose 2}}\otimes\\
& &\otimes\Lambda(1,1,0)^{r^l\left(mn+{n\choose 2}\right)+r^{l-1}l\left(m+n\right)+r^{l-2}{l\choose 2}}\otimes\\
&&\otimes\Lambda(0,0,1)^{r^{l-1}l}\otimes\oo\left(-r^l{m\choose 2}\widetilde\delta\right).
\end{eqnarray*}
\end{enumerate}
\end{teo}

\begin{proof}As we will see in the Lemma \ref{redsemistable}, we can reduce to proving the equalities on the quasi-compact open substack $\CVc^{Pss}$. We follow the same strategy in the proof of \cite[Theorem 5.2]{MV}. The first Chern class map is an isomorphism by Theorem \ref{picchow}. Thus it is enough to prove the above relations in the rational Chow group $A^1\left( \CVc^{Pss}\right)\otimes\mathbb Q$. Applying the Grothendieck-Riemann-Roch Theorem to the universal curve $\overline{\pi}:\overline{\mt Vec}_{r,d,g,1}\rightarrow \CVc$, we get:
\begin{equation}\label{1}
\ch\left(\overline{\pi}_!\left(\omega_{\overline{\pi}}^m\otimes(\det\mt E)^n\otimes\mt E^l\right)\right)=\overline{\pi}_*\left(\ch\left(\omega_{\overline{\pi}}^m\otimes(\det\mt E)^n\otimes\mt E^l\right)\cdot \Td^{\vee}\left(\Omega_{\overline{\pi}}\right)\right)
\end{equation}
where $\ch$ is the Chern character, $\Td^{\vee}$ is the dual Todd class and $\Omega_{\overline{\pi}}$ is the sheaf of relative Kahler differentials. Using Theorem \ref{detcoh}, the degree one part of the left hand side becomes
\begin{equation}\label{2}
\ch\left(\overline{\pi}_!\left(\omega_{\overline{\pi}}^m\otimes(\det\mt E)^n\otimes\mt E^l\right)\right)_1=c_1\left(\overline{\pi}_!\left(\omega_{\overline{\pi}}^m\otimes(\det\mt E)^n\otimes\mt E^l\right)\right)=c_1\left(\Lambda(m,n,l)\right)=\lambda(m,n,l).
\end{equation}
In order to compute the right hand side, we will use the fact that $c_1\left(\Omega_{\overline{\pi}}\right)=c_1\left(\omega_{\overline{\pi}}\right)$ and $\overline{\pi}_*\left(c_2\left(\Omega_{\overline{\pi}}\right)\right)=\widetilde\delta$ (see \cite[p. 383]{ACG11}. Using this, the first three terms of the dual Todd class of $\Omega_{\overline{\pi}}$ are equal to
\begin{equation}\label{3}
\text{Td}^{\vee}\left(\Omega_{\overline{\pi}}\right)=1-\frac{c_1\left(\Omega_{\overline{\pi}}\right)}{2}+\frac{c_1\left(\Omega_{\overline{\pi}}\right)^2+c_2\left(\Omega_{\overline{\pi}}\right)}{12}+\ldots=1-\frac{c_1\left(\omega_{\overline{\pi}}\right)}{2}+\frac{c_1\left(\omega_{\overline{\pi}}\right)^2+c_2\left(\Omega_{\overline{\pi}}\right)}{12}+\ldots
\end{equation}
By the multiplicativity of the Chern character, we get
\begin{multline}\label{4}
\text{ch}\left(\omega_{\overline{\pi}}^m\otimes(	\det\mt E)^n\otimes\mt E^l\right)=\ch\left(\omega_{\overline{\pi}}\right)^m\ch(\det\mt E)^n\ch\left(\mt E\right)^l=\\
\shoveleft=\left(1+c_1\left(\omega_{\overline{\pi}}\right)+\frac{c_1\left(\omega_{\overline{\pi}}\right)^2}{2}+\ldots\right)^m\cdot\left(1+c_1\left(\mt E\right)+\frac{c_1\left(\mt E\right)^2}{2}+\ldots\right)^n\cdot\\
\shoveright{\cdot\left(r+c_1\left(\mt E\right)+\frac{c_1\left(\mt E\right)^2-2c_2\left(\mt E\right)}{2}+\ldots\right)^l=}\\
\shoveleft=\left(1+mc_1\left(\omega_{\overline{\pi}}\right)+\frac{m^2}{2}c_1\left(\omega_{\overline{\pi}}\right)^2+\ldots\right)\cdot\left(1+nc_1\left(\mt E\right)+\frac{n^2}{2}c_1\left(\mt E\right)^2+\ldots\right)\cdot\\
\shoveright{\cdot\left(r^l+lr^{l-1}c_1\left(\mt E\right)+\frac{lr^{l-2}}{2}\left((r+l-1)c_1\left(\mt E\right)^2-2rc_2\left(\mt E\right))\right)+\ldots\right)=}\\
\shoveleft{=r^l+\left[rmc_1\left(\omega_{\overline{\pi}}\right)+(rn+l)c_1\left(\mt E\right)\right]r^{l-1}+\left[r^l\frac{m^2}{2}c_1\left(\omega_{\overline{\pi}}\right)^2+r^{l-1}m\left(rn+l\right)c_1\left(\omega_{\overline{\pi}}\right)c_1\left(\mt E\right)+\right.}\\
\left.+ \frac{r^{l-2}}{2}\left(r^2n^2+lr(2n+1)+l(l-1)\right)c_1\left(\mt E\right)^2-lr^{l-1}c_2\left(\mt E\right)\right].
\end{multline}
Combining (\ref{3}) and (\ref{4}), we can compute the degree one part of the right hand side of (\ref{1}):
\begin{multline}
\left[\overline{\pi}_*\left(\text{ch}\left(\omega_{\overline{\pi}}^m\otimes(\det\mt E)^n\otimes\mt E^l\right)\cdot \text{Td}^{\vee}\left(\Omega_{\overline{\pi}}\right)\right)\right]_1=\overline{\pi}_*\left(\left[\text{ch}\left(\omega_{\overline{\pi}}^m\otimes(\det\mt E)^n\otimes\mt E^l\right)\cdot \text{Td}^{\vee}\left(\Omega_{\overline{\pi}}\right)\right]_2\right)=\\
\shoveleft{=\overline{\pi}_*\left(\frac{r^l}{12}(6m^2-6m+1)c_1(\omega_{\overline{\pi}})^2+\frac{r^{l-1}}{2}(rn+l)(2m-1)c_1\left(\omega_{\overline{\pi}}\right)c_1\left(\mt E\right)+\right.}\\
\shoveright{\left.+\frac{r^{l-2}}{2}\left(r^2n^2+lr(2n+1)+l(l-1)\right)c_1\left(\mt E\right)^2-lr^{l-1}c_2\left(\mt E\right)+\frac{r^l}{12}c_2\left(\Omega_{\overline{\pi}}\right)\right)=}\\
\shoveleft{=\frac{r^l}{12}(6m^2-6m+1)k_{1,0,0}+\frac{r^{l-1}}{2}(rn+l)(2m-1)k_{0,1,0}+}\\
+\frac{r^{l-2}}{2}\left(r^2n^2+lr(2n+1)+l(l-1)\right)k_{-1,2,0}-lr^{l-1}\overline{\pi}_*c_2\left(\mt E\right)+\frac{r^l}{12}\widetilde\delta.
\end{multline}
Combining with (\ref{2}), we have:
\begin{multline}\label{5}
\lambda(m,n,l)=\frac{r^l}{12}(6m^2-6m+1)k_{1,0,0}+\frac{r^{l-1}}{2}(rn+l)(2m-1)k_{0,1,0}+\\
+\frac{r^{l-2}}{2}\left(r^2n^2+lr(2n+1)+l(l-1)\right)k_{-1,2,0}-lr^{l-1}\overline{\pi}_*c_2\left(\mt E\right)+\frac{r^l}{12}\widetilde\delta.
\end{multline}
As a special case of the above relation, we get
\begin{equation}\label{6}
\lambda(1,0,0)=\frac{k_{1,0,0}}{12}+\frac{\widetilde\delta}{12}.
\end{equation}
If we replace (\ref{6}) in (\ref{5}), then we have
\begin{multline}\label{7}
\lambda(m,n,l)=r^l(6m^2-6m+1)\lambda(1,0,0)+\frac{r^{l-1}}{2}(rn+l)(2m-1)k_{0,1,0}+\\
+\frac{r^{l-2}}{2}\left(r^2n^2+lr(2n+1)+l(l-1)\right)k_{-1,2,0}-lr^{l-1}\overline{\pi}_*c_2\left(\mt E\right)-r^l{m\choose 2}\widetilde\delta.
\end{multline}
Moreover from (\ref{7}) we obtain:
\begin{equation}\label{8}
\begin{cases}
\lambda(0,1,0)=\lambda(1,0,0)-\frac{k_{0,1,0}}{2}+\frac{k_{-1,2,0}}{2}\\
\lambda(1,1,0)=\lambda(1,0,0)+\frac{k_{0,1,0}}{2}+\frac{k_{-1,2,0}}{2}\\
\lambda(0,0,1)=r\lambda(1,0,0)-\frac{k_{0,1,0}}{2}+\frac{k_{-1,2,0}}{2}-\overline{\pi}_*c_2\left(\mt E\right)\\
\lambda(1,0,1)=r\lambda(1,0,0)+\frac{k_{0,1,0}}{2}+\frac{k_{-1,2,0}}{2}-\overline{\pi}_*c_2\left(\mt E\right)
\end{cases}
\end{equation}
which gives
\begin{equation}\label{9}
\begin{cases}
k_{0,1,0}=\lambda(1,0,1)-\lambda(0,0,1)=\lambda(1,1,0)-\lambda(0,1,0)\\
k_{-1,2,0}=-2\lambda(1,0,0)+\lambda(0,1,0)+\lambda(1,1,0)\\
\overline{\pi}_*c_2\left(\mt E\right)=(r-1)\lambda(1,0,0)+\lambda(0,1,0)-\lambda(0,0,1).
\end{cases}
\end{equation}
Substituting in (\ref{7}), we finally obtain
\begin{eqnarray}
r^{2-l}\lambda(m,n,l)&= &\Big(r^2(6m^2-6m+1-n^2-l)-2rnl-l(l-1)\Big)\lambda(1,0,0)+\\
\nonumber & &+\left(r^2\left(-mn+{n+1\choose 2}\right)+rl\left(n -m\right)+{l\choose 2}\right)\lambda(0,1,0)+\\
\nonumber& &+\left(r^2\left(mn+{n\choose 2}\right)+rl\left(m+n\right)+{l\choose 2}\right)\lambda(1,1,0)+\\
\nonumber & &+\,rl\lambda(0,0,1)-r^2{m\choose 2}\widetilde\delta.
\end{eqnarray}
\end{proof}

\begin{rmk}As we will see in the next section the integral Picard group of $\Pic(\CVc)$ is torsion free for $g\geq 3$. In particular, the relations of Theorem \ref{relations} hold also for $\Pic(\CVc)$.

\end{rmk}
\section{The Picard groups of $\CVc$ and $\CVr$.}\label{robba}
The aim of this section is to prove the Theorems \ref{pic} and \ref{picred}. We will prove them in several steps. Since all the results in this section have been already proved in the rank one case in \cite{MV}, for the rest of the paper we will assume $r\geq 2$.

\subsection{Independence of the boundary divisors.}\label{indipendece}
The aim of this subsection is to prove the following
\begin{teo}\label{indbou}Assume that $g\geq 3$. We have an exact sequence of groups
$$
0\longrightarrow\bigoplus_{i=0,\ldots,\lfloor g/2\rfloor}\oplus_{j\in J_i}\langle\oo(\widetilde\delta_i^j)\rangle\longrightarrow Pic(\CVc)\longrightarrow Pic(\Vc)\longrightarrow 0
$$
where the right map is the natural restriction and the left map is the natural inclusion.
\end{teo}
For the rest of this subsection, with the only exceptions of Proposition \ref{T} and Lemma \ref{sHs}, we will always assume that $g\geq 3$. We recall now a result from \cite{T95}.
\begin{prop}\label{T}\cite[Proposition 1.2]{T95}. Let $C$ be a nodal curve of genus greater than one without rational components and let $\mt E$ be a balanced vector bundle over $C$ with rank $r$ and degree $d$. Let $C_1,\ldots,C_s$ be its irreducible components. If $\mt E_{C_i}$ is semistable for any $i$ then $\mt E$ is P-semistable. Moreover if the basic inequalities are all strict and all the $\mt E_{C_i}$ are semistable and at least one is stable then $\mt E$ is P-stable.
\end{prop}

The next result is probably well-known to the experts. Since we did not find an exhaustive reference for all the cases, we included a sketchy proof.

\begin{lem}\label{gvb}The generic vector bundle of rank $r$ and degree $d$ over smooth curve of genus $g\geq 1$ is semistable. 
Furthermore, if $g>1$ the generic semistable vector bundle is stable. If $g=1$, the stable locus is not empty and coincides with the semistable locus if and only if $n_{r,d}=1$. In general, the generic semistable vector bundle over an elliptic curve of degree $d$ and rank $r$ is direct sum of $n_{r,d}$ distinct stable vector bundles of degree $d/n_{r,d}$ and rank $r/n_{r,d}$.
\end{lem}

\begin{proof}The assertions on $g\geq 2$ follows from the work of Seshadri \cite{Se82}. Let focus on $g=1$. The semistable locus is non-empty by \cite{Tu} and is open by \cite[Proposition 2.3.1]{HL}.  The sentence about the stable locus is proved in \cite[Appendix A]{Tu}. Consider the good moduli map $p:\mt Vec^{ss}_{r,d,C}\to U_{r,d,C}$, where the domain is the moduli stack of semistable vector bundles over an elliptic curve $C$ and the target is the associated moduli space, whose points are in bijection with polystable vector bundles. By \emph{loc. cit.}, a vector bundle is polystable if and only if it is direct sums of $n_{r,d}$ stable bundles of rank $r/n_{r,d}$ and degree $d/n_{r,d}$. We will show that the map $p$ induces a bijection between the set of $k$-points, along the open subset of $U_{r,d,C}$ of those bundles which are direct sums of distinct bundles.

Since the map $p$ sends a semistable vector bundle $\mt E$ to the S-equivalent polystable one $gr(\mt E)$, it is enough to show that
\begin{equation}\label{poly}
gr(\mt E)= \bigoplus_{i=1}^{n_{r,d}}\mt E_i, \text{ s.t. } \mt E_i\neq\mt E_j \text{ when }i\neq j\Longrightarrow gr(\mt E)=\mt E.
\end{equation}
We will give a proof by induction on the rank. If $r=1$, it is obvious. Assume that (\ref{poly}) is true for ranks smaller than $r$. Let $\mt E$ be a semistable vector bundle of rank $r$ and degree $d$ satisfying the hypothesis of (\ref{poly}). Up to permuting the factors of $gr(\mt E)$, we can assume that there exists an exact sequence
$$
0\to\mt F\to\mt E\to \mt E_1\to 0,
$$ 
with $F$ semistable of rank $r(n_{r,d}-1)/n_{r,d}$ and degree $d(n_{r,d}-1)/n_{r,d}$. By induction, 
$
\mt F=gr(\mt F)=\bigoplus_{i=2}^{n_{r,d}}\mt E_i.
$
If we show $\text{Ext}^1(\mt E_1,\mt F)=0$, the sequence will split and, by inductive assumptions, we have the equality $gr(\mt E)=\mt E$. Observe that 
$$
0=\chi(Hom(\mt E_1,\mt F))=\text{hom}(\mt E_1,\mt F)- \text{ext}^1(\mt E_1,\mt F).
$$
We will show that $\text{Hom}(\mt E_1,\mt F)=0$, concluding the proof. Let $\varphi:\mt E_1\to\mt F$ be a non zero morphism of vector bundles. First of all, $\varphi$ must be injective, because otherwise
$$
\frac{\chi(\text{Im}\varphi)}{\text{rank}(\text{Im}\varphi)}>\frac{\chi(\mt E_1)}{\text{rank}\mt E_1}=\frac{d}{r}= \frac{\chi(\mt F)}{\text{rank}\mt F}\geq \frac{\chi(\text{Im}\varphi)}{\text{rank}(\text{Im}\varphi)}.
$$
So, $\varphi$ is injective. In particular, $\mt E_1$ becomes a destabilizing stable sheaf also for $\mt F$, but this contradicts $\mt E_i\neq\mt E_j$ for $i\neq j$. Thus, $\varphi=0$.
\end{proof}

We deduce from this

\begin{lem}\label{T2}The generic point of $\widetilde\delta_i^j$ is a curve $C$ with exactly one node and a properly balanced vector bundle $\mt E$ such that
\begin{enumerate}[(i)]
 \item if $i=0$ the vector bundle $\mt E$ is P-stable, 
\item  if $i=1$ and we denote by $C_1$ and $C_2$ are the irreducible smooth components of genus $1$ and $g-1$, the restriction $\mt E_{C_1}$ is a direct sum of distinct stable vector bundles with same rank and degree and $\mt E_{C_2}$ is a stable vector bundle.
\item if $2\leq i\leq \lfloor g/2\rfloor$ and we denote by $C_1$ and $C_2$ are the irreducible smooth components of genus $i$ and $g-i$, the restrictions $\mt E_{C_1}$ and $\mt E_{C_2}$ are stable vector bundles.
\end{enumerate}
Furthermore, $\mt E$ is $P$-stable if $\widetilde\delta_i^j$ is a non-extremal divisor and $\mt E$ is strictly semistable if $\widetilde\delta_i^j$ is an extremal divisor.
\end{lem}

\begin{proof}When $i=0$, it follows from the work of Seshadri \cite{Se82}. We fix $i\in \{1,\ldots,\lfloor g/2\rfloor\}$ and $j\in J_i$. By definition the generic point of $\widetilde\delta_i^j$ is a curve with two irreducible components $C_1$ and $C_2$ of genus $i$ and $g-i$ meeting at one point and a vector bundle $\mt E$ with multidegree
$$
(\deg\mt E_{C_1},\deg\mt E_{C_2})=\left(\left\lceil d\frac{2i-1}{2g-2}-\frac{r}{2}\right\rceil+j,\left\lfloor d\frac{2(g-i)-1}{2g-2}+\frac{r}{2}\right\rfloor-j\right).
$$
Fix a such curve $C$. As in the proof of Proposition \ref{boundary}(i), we have a surjective smooth morphism of stacks
$$
\widetilde\delta_i^{*j}\cap \overline{\phi}_{r,d}^{-1}([C])\to\mt Vec_{r,\deg\mt E_{C_1},C_1}\times \mt Vec_{r,\deg\mt E_{C_2},C_2},
$$
where $\widetilde\delta_i^{*j}$ is the open locus in $\widetilde\delta_i^j$ of curves with just one node and $\overline{\phi}_{r,d}:\CVc\to\CMg$ is the natural morphism of stacks, which forgets the bundle and stabilizes the curve. By Lemma \ref{gvb}, we have that points (ii) and (iii) hold. 

By Proposition \ref{T}, the generic point of $\widetilde\delta_i^j$ is P-semistable. Moreover if $\widetilde\delta_i^j$ is a non-extremal divisor  the basic inequalities are strict. By the second assertion of \emph{loc. cit.}, if $\widetilde\delta_i^j$ is a non-extremal divisor the generic point of $\widetilde\delta_i^j$ is P-stable. It remains to prove the assertion for the extremal divisors. Suppose that $\widetilde\delta_i^0$ is an extremal divisor, the proof for the $\widetilde\delta_i^r$ is similar. As in the proof of Lemma \ref{balcon}, we can see that
$$
\deg_{C_1}\mt E=d\frac{2i-1}{2g-2}-\frac{r}{2}\iff
\frac{\chi\left(\mt E_{C_1}\right)}{\omega_{C_1}}=\frac{\chi\left(\mt E\right)}{\omega_{C}}.
$$
In other words, $\mt E_{C_1}$ is a destabilizing quotient for $\mt E$, concluding the proof.
\end{proof}
\begin{lem}\label{redsemistable}The Picard group of $\Vc$ is isomorphic to the Picard groups of $\Vc^{(s)s}$ and $\mt U_n$ for $n$ big enough. The Picard group of $\CVc$ is isomorphic to the Picard groups of $\CVc^{Pss}$ and $\overline{\mt U}_n$ for $n$ big enough. The same results hold for the rigifications.
\end{lem}

\begin{proof}Fix a $k$-point $(C,\mt L)$ in $\Jc$. Then, by the determinant map $det:\Vc\to \Jc$, we have that
\begin{eqnarray*}
\dim\Vc &=&\dim\Jc+\dim\Vl,\\
\dim\left(\Vc\backslash\Vc^{(s)s} \right)& \leq&\dim\Jc+\dim\left(\Vl\backslash\Vl^{(s)s}\right).
\end{eqnarray*}
Thus $\cod(\Vc\backslash\Vc^{(s)s},\Vc)\geq\cod(\Vl\backslash\Vl^{(s)s},\Vl)\geq 2$ (see proof of \cite[Corollary 3.2]{H}).
By Proposition \ref{qc}, there exists $n$ big enough such that $\CVc^{Pss}\subset \overline{\mt U}_n$. In particular, $\cod(\mt U_n\backslash\Vc^{ss},\mt U_n)\geq 2$. Suppose that $\cod(\CVc\backslash\CVc^{Pss},\CVc)=1$, so
 $\CVc\backslash\CVc^{Pss}$ contains a substack of codimension $1$. By the observations above this stack must be contained in the boundary divisor $\widetilde\delta$. The generic point of any divisor $\widetilde\delta_i^j$ is P-semistable by Lemma \ref{T2}, then we have a contradiction. So  $$\cod\left(\overline{\mt U}_n\backslash\CVc^{Pss},\overline{\mt U}_n\right)\geq \cod\left(\CVc\backslash\CVc^{Pss},\CVc\right)\geq2.$$The same argument works for the rigidifications. By Theorem \ref{picchow}, the lemma follows.
\end{proof}

By Lemma \ref{redsemistable}, Theorem \ref{indbou} is equivalent to proving that there exists $n_*\gg 0$ such that for $n\geq n_*$ we have an exact sequence of groups
$$
0\lra\bigoplus_{i=0,\ldots,\lfloor g/2\rfloor}\oplus_{j\in J_i}\langle\oo(\widetilde\delta_i^j)\rangle\lra \Pic(\overline{\mt U}_n)\lra \Pic(\mt U_n)\lra 0.
$$
By Theorem \ref{picchow}, the sequence is exact in the middle and at right. It remains to prove the left exactness.
The strategy that we will use is the same as the one of Arbarello-Cornalba for $\CMg$ in \cite{AC87} and the generalization for $\overline{\mt Jac}_{r,g}$ done by Melo-Viviani in \cite{MV}. More precisely, we will construct morphisms $B\to\overline{\mt U}_n$ from irreducible smooth projective curves $B$ and we compute the degree of the pull-backs of the boundary divisors of $\Pic(\overline{\mt U}_n)$ to $B$. We will construct liftings of the families $F_h$ (for $1 \leq h \leq (g - 2 ) /2$), $F$ and $F'$ used by Arbarello-Cornalba in \cite[pp. 156-159]{AC87}. Since $\CVc\cong\overline{\mathcal Vec}_{r,d',g}$ if $d\equiv d'$ mod $(r(2g-2))$, in this section we can assume that $0\leq d<r(2g-2)$.\\\\
\textbf{The Family $\widetilde F$}.

Consider a general pencil in the linear system $H^0(\mathbb{P}^2,\oo(2))$. It defines a rational map $\mathbb{P}^2\dashrightarrow\mathbb{P}^1$, which is regular outside of the four base points of the pencil. Blowing the base locus we get a conic bundle $\phi:X\rightarrow\mathbb{P}^1$. The four exceptional divisors $E_1$, $E_2$, $E_3$, $E_4\subset X$ are sections of $\phi$. It can be shown that the conic bundle has 3 singular fibers consisting of rational chains of length two. Fix a smooth curve $C$ of genus $g-3$ and $p_1,p_2,p_3,p_4$ points of $C$. Consider the following surface
$$
Y=\left(X\amalg(C\times\mathbb P^1)\right)/(E_i\sim \{p_i\}\times \mathbb P^1).
$$
We get a family $f:Y\rightarrow \mathbb P^1$ of stable curves of genus $g$. The general fiber of $f$ is as in Figure \ref{Figure1} where $Q$ is a smooth conic.
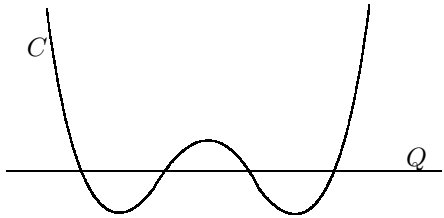
\begin{figure}[h!]
\begin{center}
\unitlength .65mm 
\linethickness{0.4pt}
\ifx\plotpoint\undefined\newsavebox{\plotpoint}\fi 
\begin{picture}(85.55,50.324)(0,110)
\qbezier(12.78,153.574)(19.905,97.574)(34.53,116.574)
\qbezier(34.53,116.574)(46.155,136.574)(56.28,116.574)
\qbezier(56.28,116.574)(71.53,96.699)(78.78,154.324)
\put(10.828,145.765){\makebox(0,0)[cc]{$C$}}
\put(4.625,120.248){\line(1,0){.21}}
\put(4.835,120.248){\line(1,0){88.715}}
\put(88.504,122.561){\makebox(0,0)[cc]{$Q$}}
\end{picture}
\end{center}
\caption{The general fiber of $f:Y\to \mathbb P^1$}\label{Figure1}
\end{figure}\\
While the 3 special components are as in Figure 2 where $R_1$ and $R_2$ are rational curves.
\begin{figure}[ht]
\begin{center}
\unitlength .65mm 
\linethickness{0.4pt}
\ifx\plotpoint\undefined\newsavebox{\plotpoint}\fi 
\begin{picture}(80.75,50.75)(0,150)
\put(10.25,182.25){\line(-1,0){.25}}
\put(10,182.25){\line(1,0){.25}}
\multiput(10.25,182.25)(.04485138004,-.03370488323){942}{\line(1,0){.04485138004}}
\multiput(35,150.75)(.04842799189,.03372210953){986}{\line(1,0){.04842799189}}
\qbezier(12.25,196)(19.375,140)(34,159)
\qbezier(34,159)(45.625,179)(55.75,159)
\qbezier(55.75,159)(71,139.125)(78.25,196.75)
\put(24.5,176.25){\makebox(0,0)[cc]{$R_1$}}
\put(62.75,175){\makebox(0,0)[cc]{$R_2$}}
\put(8,196.5){\makebox(0,0)[cc]{$C$}}
\end{picture}
\end{center}
\caption{The three special fibers of $f:Y\to \mathbb P^1$}\label{Figure2}
\end{figure}
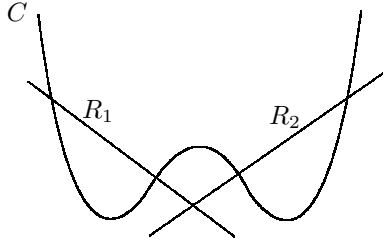\\
Choose a vector bundle of degree $d$ on $C$, pull it back to $C\times \mathbb P^1$ and call it $E$. Since $E$ is trivial on $\{p_i\}\times \mathbb P^1$, we can glue it with the trivial vector bundle of rank $r$ on $X$ obtaining a vector bundle $\mt E$ on $f:Y\rightarrow\mathbb P^1$ of rank $r$ and degree $d$.

\begin{lem}$\mt E$ is properly balanced.
\end{lem}

\begin{proof}$\mt E$ is obviously admissible because is defined over a family of stable curves. Since being properly balanced is an open condition, we can reduce to check that $\mt E$ is properly balanced on the three special fibers. By Lemma \ref{balcon}, it is enough to check the basic inequality for the subcurves $R_1\cup R_2$, $R_1$ and $R_2$. And by the assumption $0\leq d<r(2g-2)$, a direct computation shows that the inequalities hold.
\end{proof}

We call $\widetilde F$ the family $f:Y\rightarrow\mathbb P^1$ with the vector bundle $\mt E$. It is a lifting of the family $F$ defined in \cite[p. 158]{AC87}. So we can compute the degree of the pull-backs of the boundary bundles in $\Pic(\CVc)$ to the curve $\widetilde F$. Consider the commutative diagram
$$
\xymatrix{
\mathbb P^1\ar[rd]^{F}\ar[r]^{\widetilde F} & \CVc\ar[d]^{\overline{\phi}_{r,d}}\\
 &\CMg
}$$
By Proposition \ref{boundary}, we have $\deg_{\widetilde F}\oo(\widetilde\delta_0)=\deg_F\oo(\delta_0)$ and $\deg_F\oo(\delta_0)=-1$ by \cite[p. 158]{AC87}. Since $\widetilde F$ does not intersect the other boundary divisors, we have:
$$
\begin{cases}
\deg_{\widetilde F}\oo(\widetilde\delta_0)=-1, &\\
\deg_{\widetilde F}\oo(\widetilde\delta_i^j)=0 & \mbox{if } i\neq 0\mbox{ and } j\in J_i.
\end{cases}
$$\\
\textbf{The Families $\widetilde F_1'^j$ and $\widetilde F_2'^j$}(for $j\in J_1$).

We start with the same family of conics $\phi:X\rightarrow \mathbb P^1$ and the same smooth curve $C$ used for the family $\widetilde F$. Let $\Gamma$ be a smooth elliptic curve and take points $p_1\in \Gamma$ and $p_2,p_3,p_4\in C$. We construct a new surface
$$
Z=\left(X\amalg(C\times\mathbb P^1)\amalg(\Gamma\times\mathbb P^1)\right)/(E_i\sim \{p_i\}\times \mathbb P^1).
$$
We obtain a family $g:Z\rightarrow \mathbb P^1$ of stable curves of genus $g$. The general fiber is as in Figure \ref{Figure3} where $Q$ is a smooth conic. The three special fibers are as in Figure \ref{Figure4} where $R_1$ and $R_2$ are rational smooth curves.
\begin{figure}[ht]
\begin{center}
\unitlength .6mm 
\linethickness{0.4pt}
\ifx\plotpoint\undefined\newsavebox{\plotpoint}\fi 
\begin{picture}(131,70.125)(0,115)
\qbezier(18.5,176)(37.5,96.75)(53.5,158.5)
\qbezier(53.5,158.5)(72.75,214.125)(83,130.25)
\put(13,146.75){\line(1,0){1.75}}
\put(14.75,146.75){\line(1,0){112.5}}
\qbezier(123.5,176)(131,165.5)(103.5,155)
\qbezier(103.5,155)(93,147)(103.5,138)
\qbezier(103.5,138)(130.625,123.625)(121.25,114.75)
\put(127.25,173.75){\makebox(0,0)[cc]{$\Gamma$}}
\put(125,150){\makebox(0,0)[cc]{$Q$}}
\put(20.75,177){\makebox(0,0)[cc]{$C$}}
\end{picture}
\end{center}
\caption{The general fibers of $g:Z\to \mathbb P^1$.}\label{Figure3}
\end{figure}
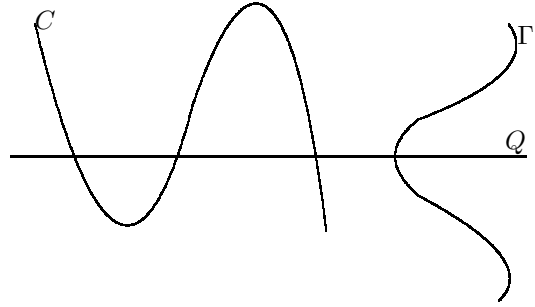
\begin{figure}[ht]
\begin{center}
\unitlength .6mm 
\linethickness{0.4pt}
\ifx\plotpoint\undefined\newsavebox{\plotpoint}\fi 
\begin{picture}(131,60.125)(0,120)
\qbezier(18.5,176)(37.5,96.75)(53.5,158.5)
\qbezier(53.5,158.5)(72.75,214.125)(83,130.25)
\qbezier(123.5,176)(131,165.5)(103.5,155)
\qbezier(103.5,155)(93,147)(103.5,138)
\qbezier(103.5,138)(130.625,123.625)(121.25,114.75)
\put(128.25,173.75){\makebox(0,0)[cc]{$\Gamma$}}
\put(20.75,177){\makebox(0,0)[cc]{$C$}}
\put(11,169.5){\line(1,0){.25}}
\multiput(11.25,169.5)(.05160550459,-.03371559633){1090}{\line(1,0){.05160550459}}
\put(51.5,133.5){\line(0,1){.25}}
\multiput(51.5,133.75)(.125,.0337370242){578}{\line(1,0){.125}}
\put(10.75,165){\makebox(0,0)[cc]{$R_1$}}
\put(122.5,147.5){\makebox(0,0)[cc]{$R_2$}}
\end{picture}
\end{center}
\caption{The three special fibers of $g:Z\to \mathbb P^1$.}\label{Figure4}
\end{figure}
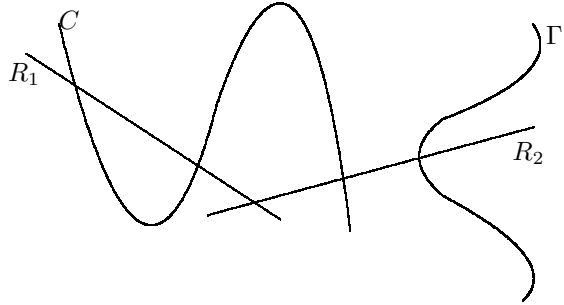\\
Let $j$ be an integer. We choose two vector bundles of degree $d-j$ and $d-3j$ on $C$, pull them back to $C\times \mathbb P^1$ and call them $G_1^j$ and $G_2^j$. We choose a vector bundle of degree $j$ on $\Gamma$, pull it back to $\Gamma\times\mathbb P^1$ and call it $M^j$. We glue the vector bundle $G_1^j$ (resp. $G_2^j$) on $C\times\mathbb P^1$, the vector bundle $M^j$ on $\Gamma\times\mathbb P^1$ and the vector bundle $\oo_X^r$ (resp. $\phi^*\oo_{\mathbb P^1}(j)\otimes\omega_{X/\mathbb P^1}^{-j}\oplus\oo_X^{r-1}$), obtaining a vector bundle $\mt G_1^j$ (resp. $\mt G_2^j$) on $Z$ of rank $r$ and degree $d$.

\begin{lem} Let $j$ be an integer such that $$\left|j-\frac{d}{2g-2}\right|\leq\frac{r}{2}.$$
Then $\mt G_1^j$ is properly balanced if $0\leq d\leq r(g-1)$ and $\mt G_2^j$ is properly balanced if $r(g-1)\leq d<r(2g-2)$.
\end{lem}

\begin{proof}As before we can check the condition on the special fibers. By Lemma \ref{balcon} we can reduce to check the inequalities for the subcurves $\Gamma,C,R_1$ and $R_2\cup\Gamma$. Suppose that $0\leq d\leq r(g-1)$ and consider $\mt G_1^j$. The inequality on $\Gamma$ follows by hypothesis. The inequality on $C$ is
$$
\left|d-j-d\frac{2g-5}{2g-2}\right|\leq\frac{3}{2}r \iff \left|j-d\frac{3}{2g-2}\right|\leq\frac{3}{2}r,
$$
and this follows by these inequalities (true by hypothesis on $j$ and $d$)
$$
\left|j-d\frac{3}{2g-2}\right|\leq\left|j-\frac{d}{2g-2}\right|+\left|\frac{d}{g-1}\right|\leq \frac{r}{2}+r.
$$
The inequality on $R_1$ is
$$
\left|\frac{d}{2g-2}\right|\leq\frac{3}{2}r,
$$
and this follows by the hypothesis on $d$. Finally the inequality on $R_2\cup \Gamma$ is
$$
\left|j-\frac{d}{g-1}\right|\leq r,
$$
and this follows by the following inequalities (true by hypothesis on $j$ and $d$)
$$
\left|j-\frac{d}{g-1}\right|\leq\left|j-\frac{d}{2g-2}\right|+\left|\frac{d}{2g-2}\right|\leq \frac{r}{2}+\frac{r}{2}.
$$
Suppose next that $r(g-1)\leq d<r(2g-2)$ and consider $\mt G_2^j$. The inequality on $\Gamma$ follows by hypothesis. On $C$, the inequality gives
$$
\left|d-3j-d\frac{2g-5}{2g-2}\right|\leq\frac{3}{2}r \iff \left|j-\frac{d}{2g-2}\right|\leq\frac{r}{2},
$$
which follows by hypothesis on $j$. The inequality on $R_1$ is
$$
\left|j-\frac{d}{2g-2}\right|\leq\frac{3}{2}r,
$$
and this follows by hypothesis on $j$. The inequality on $R_2\cup \Gamma$ is
$$
\left|2j-\frac{d}{g-1}\right|\leq r,
$$
and this follows by the inequalities (true by hypothesis on $j$)
$$
\left|2j-\frac{d}{g-1}\right|\leq 2\left|j-\frac{d}{2g-2}\right|\leq r.
$$
\end{proof}

Let $k\in J_1$. If $0\leq d\leq r(g-1)$, we call $\widetilde {F'}_1^k$ the family $g:Z\rightarrow\mathbb P^1$ with the properly balanced vector bundle $\mt G_1^{\lceil\frac{d}{2g-2}-\frac{r}{2}\rceil+k}$. If $r(g-1)\leq d<r(2g-2)$ we call $\widetilde {F'}_2^k$ the family $g:Z\rightarrow\mathbb P^1$ with the properly balanced vector bundle $\mt G_2^{\lceil\frac{d}{2g-2}-\frac{r}{2}\rceil+k}$. As before we compute the degree of boundary line bundles to the curves $\widetilde {F'}_1^k$ and $\widetilde {F'}_2^k$ (in the range of degrees where they are defined) using the fact that they are liftings of the family $F'$ in \cite[p. 158]{AC87}.
If  $0\leq d\leq r(g-1)$ then we have
$$
\begin{cases}
\deg_{\widetilde{F'}_1^k}\oo(\widetilde\delta_1^k)=-1,  &\\
\deg_{\widetilde{F'}_1^k}\oo(\widetilde\delta_1^j)=0 & \mbox{if }j\neq k,\\
\deg_{\widetilde{F'}_1^k}\oo(\widetilde\delta_i^j)=0 & \mbox{if }i> 1,\mbox{ for any }j\in J_i.
\end{cases}
$$
Indeed the first two relations follow from
$$\deg_{\widetilde{F'}_1^k}\oo\left(\sum_{j\in J_i}\widetilde\delta_1^j\right)=\deg_{F'}\oo(\delta_1)=-1$$ (see \cite[p. 158]{AC87}) and the fact that $\widetilde{F'}_1^k$ does not meet $\widetilde\delta_1^j$ for $k\neq j$. The last one follows from the fact that $\widetilde{F'}_1^k$ does not meet $\widetilde\delta_i^j$ for $i>1$. Similarly for $\widetilde{F'}_2^k$ we can show that for $r(g-1)\leq d< r(2g-2)$, we have
$$
\begin{cases}
\deg_{\widetilde{F'}_2^k}\oo(\delta_1^k)=-1,  &\\
\deg_{\widetilde{F'}_2^k}\oo(\delta_1^j)=0 & \mbox{if }j\neq k,\\
\deg_{\widetilde{F'}_2^k}\oo(\delta_i^j)=0 & \mbox{if }i> 1.
\end{cases}
$$\\
\textbf{The Families $\widetilde F_{h}^j$} (for $1\leq h\leq \frac{g-2}{2}$ and $j\in J_h$).

Consider smooth curves $C_1$, $C_2$ and $\Gamma$ of genus $h$, $g-h-1$ and
$1$, respectively, and points $x_1\in C_1$, $x_2\in C_2$ and $\gamma\in \Gamma$. Consider the surface
$Y_2$ given by the blow-up of $\Gamma\times \Gamma$
at $(\gamma, \gamma)$. Let $p_2:Y_2\to \Gamma$ be the map given by composing the blow-down
$Y_2\to \Gamma\times \Gamma$ with the second projection, and $\pi_1: C_1\times\Gamma\to \Gamma$ and $\pi_3:C_2\times\Gamma\to \Gamma$ be the projections along the second factor. 
\begin{figure}[ht]
	\begin{center}
		\unitlength .65mm 
		\linethickness{0.4pt}
		\ifx\plotpoint\undefined\newsavebox{\plotpoint}\fi 
		\begin{picture}(150,60.75)(30,65)
		\put(88.5,66.5){\framebox(43.5,48.25)[cc]{}}
		\put(22.75,66.5){\framebox(43.5,48.25)[cc]{}}
		\put(156.25,66.25){\framebox(43.5,48.25)[cc]{}}
		\put(88.75,102){\line(1,0){28.25}}
		\put(23,102){\line(1,0){43.25}}
		\put(44,104){\makebox(0,0)[cc]{$A$}}
		\put(14,100.25){\vector(0,-1){26.5}}
		\put(147.5,100){\vector(0,-1){26.5}}
		\put(109.25,63.5){\makebox(0,0)[cc]{}}
		\put(43.5,63.5){\makebox(0,0)[cc]{}}
		\put(177,63.25){\makebox(0,0)[cc]{}}
		\put(109.5,63){\makebox(0,0)[cc]{$\Gamma$}}
		\put(43.75,63){\makebox(0,0)[cc]{$\Gamma$}}
		\put(177.25,62.75){\makebox(0,0)[cc]{$\Gamma$}}
		\put(18.25,86.5){\makebox(0,0)[cc]{$\pi_1$}}
		\put(151.75,86.25){\makebox(0,0)[cc]{$\pi_3$}}
		\put(61.75,86.5){\makebox(0,0)[cc]{}}
		\put(128.75,87.25){\makebox(0,0)[cc]{}}
		\put(63,87.25){\makebox(0,0)[cc]{}}
		\put(196.5,87){\makebox(0,0)[cc]{}}
		\put(63,86.5){\makebox(0,0)[cc]{$C_1$}}
		\multiput(88.5,66.5)(.03372093023,.03740310078){1290}{\line(0,1){.03740310078}}
		\put(123.75,98){\line(0,-1){31.5}}
		\qbezier(112.25,106.5)(112.75,97)(126.25,90.5)
		\put(156.25,78.25){\line(1,0){43.75}}
		\put(177,81.75){\makebox(0,0)[cc]{$B$}}
		\put(94.25,104.5){\makebox(0,0)[cc]{$S$}}
		\put(83.5,86.75){\makebox(0,0)[cc]{$\Gamma$}}
		\put(101.25,86.75){\makebox(0,0)[cc]{$\Delta$}}
		\put(119.5,78){\makebox(0,0)[cc]{$T$}}
		\put(113.5,108.75){\makebox(0,0)[cc]{$E$}}
		\put(196.75,87.25){\makebox(0,0)[cc]{$C_2$}}
		\end{picture}
	\end{center}
	\caption{The surface $(C_1\times \Gamma) \amalg Y_2\amalg(C_2\times \Gamma)$.}\label{Figure 5}
\end{figure}
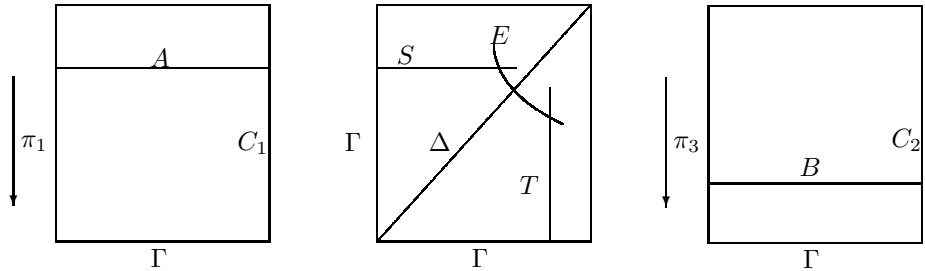

Where, as in \cite[p. 156]{AC87} (and \cite{MV}), we set:
$$\begin{aligned}
&A=\{x_1\}\times \Gamma, \\
&B=\{x_2\}\times \Gamma, \\
&E= \text{ exceptional divisor of the blow-up of } \Gamma\times \Gamma \text{ at } (\gamma,\gamma),\\
&\Delta= \text{ proper transform of the diagonal in } Y_2, \\
&S= \text{ proper transform of } \{\gamma\}\times \Gamma \text{ in } Y_2,\\
& T= \text{ proper transform of } \Gamma \times \{\gamma\}\text{ in } Y_2.
\end{aligned}
$$

Starting from $(C_1\times \Gamma) \amalg Y_2\amalg(C_2\times \Gamma)$ (see Figure \ref{Figure 5}), we construct a surface $X$, by identifying $S$ with $A$ and $\Delta$ with $B$. It comes
equipped with a projection $f:X\to \Gamma$. The fibers over all the points $\gamma'\neq \gamma$ are
shown in Figure \ref{Figure6}, while the fiber over the point $\gamma$ is shown in Figure \ref{Figure7}.
\begin{figure}[h!]
	\begin{center}
		\unitlength .65mm 
		\linethickness{0.4pt}
		\ifx\plotpoint\undefined\newsavebox{\plotpoint}\fi 
		\begin{picture}(122.75,50.5)(0,145)
		\qbezier(49.25,173.5)(76.875,158.125)(102,164.25)
		\qbezier(87,153.5)(95.125,175.625)(122.75,183.25)
		\qbezier(17.5,153.25)(52.125,160.125)(62.25,184.5)
		\put(17.75,157){\makebox(0,0)[cc]{$C_1$}}
		\put(30.75,153.75){\makebox(0,0)[cc]{$h$}}
		\put(72.5,167){\makebox(0,0)[cc]{$\Gamma$}}
		\put(64.25,162.5){\makebox(0,0)[cc]{$1$}}
		\put(113.25,183.75){\makebox(0,0)[cc]{$C_2$}}
		\put(115,173.75){\makebox(0,0)[cc]{$g-h-1$}}
		\end{picture}
	\end{center}
	\caption{The general fiber of $f:X\to \Gamma$.}\label{Figure6}
\end{figure}
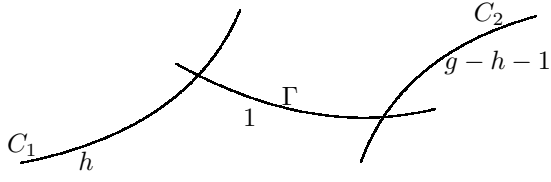
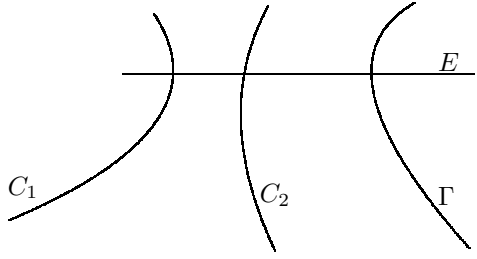
\begin{figure}[h!]
	\begin{center}
		
		\unitlength .65mm 
		\linethickness{0.4pt}
		\ifx\plotpoint\undefined\newsavebox{\plotpoint}\fi 
		\begin{picture}(93.5,55)(0,135)
		\put(21.5,170.5){\line(1,0){72}}
		\qbezier(-1.75,140.5)(43.375,159.375)(28,182.75)
		\qbezier(51.25,184.25)(39.5,162.25)(52.75,134.25)
		\qbezier(81.25,185)(59.125,172.125)(92.5,134.75)
		\put(88.25,173){\makebox(0,0)[cc]{$E$}}
		\put(1.25,147.25){\makebox(0,0)[cc]{$C_1$}}
		\put(52.75,145.75){\makebox(0,0)[cc]{$C_2$}}
		\put(87.75,145.5){\makebox(0,0)[cc]{$\Gamma$}}
		\end{picture}
	\end{center}
	\caption{The special fiber of $f:X\to \Gamma$.}\label{Figure7}
\end{figure}

Consider the line bundles $\oo_{Y_2}$, $\oo_{Y_2}(\Delta)$, $\oo_{Y_2}(E)$ over the surface $Y_2$. \begin{table}[ht]
\begin{center}
\begin{tabular}{|c|c|c|c|c|}
\hline
    & $deg_E$ & $deg_T$ & restriction to $\Delta$ & restriction to $S$ \\
\hline
  $\oo_{Y_2}$ & $0$ & $0$ & $\oo_{\Gamma}$ & $\oo_{\Gamma}$ \\
  $\oo_{Y_2}(\Delta)$ & $1$ & $0$ & $\oo_{\Gamma}(-\gamma)$ & $\oo_{\Gamma}$ \\
  $\oo_{Y_2}(E)$ & $-1$ & $1$ & $\oo_{\Gamma}(\gamma)$ & $\oo_{\Gamma}(\gamma)$ \\
  \hline
\end{tabular}\caption{}\label{table1}
\end{center}
\end{table}
Let $j$, $k$, $t$ be integers. Consider a vector bundle on $C_1$ of rank $r-1$ and degree $j$, we pull it back to $C_1\times \Gamma$ and call it $H^j$. Similarly take a vector bundle on $C_2$ of rank $r-1$ and degree $k$, we pull it back to $C_2\times \Gamma$ and call it $P^k$.
Consider the following vector bundles
$$
\begin{array}{lccclcl}
  M_{C_1\times \Gamma}^{j,k,t}&:=&H^j&\oplus&\pi_1^*\oo_{\Gamma}(t\gamma) & \text{on} & C_1\times \Gamma ,\\
  M_{C_2\times \Gamma}^{j,k,t}&:=&P^k&\oplus&\pi_3^*\oo_{\Gamma}\big((j+k+t-d)\gamma\big) &\text{on} & C_2\times \Gamma,\\
  M_{Y_2}^{j,k,t}&:=&\oo_{Y_2}^{r-1}&\oplus&\oo_{Y_2}\big((d-j-k)\Delta+ tE\big) &\text{on}& Y_2.
\end{array}
$$
By Table \ref{table1} (see \cite[p. 16-17]{MV}), we have $M_{C_1\times \Gamma}^{j,k,t}|_A\cong M^{j,k,t}_{Y_2}|_S \text{ and } M_{C_2\times \Gamma}^{j,k,t}|_B\cong M_{Y_2}^{j,k,t}|_{\Delta}$. So we can glue the vector bundles in a vector bundle $\mt M_{h}^{j,k,t}$ on the family $f:X\rightarrow \Gamma$. Moreover, by Table \ref{table1}, on the special fiber we have
$$\begin{array}{lclcl}
\deg_{C_1}(\mt M_{h}^{j,k,t}|_{f^{-1}(\gamma)})&=&\deg_{\pi_1^{-1}(\gamma)}(M^{j,k,t}_{C_1\times\Gamma})&=&j,\\
\deg_{C_2}(\mt M^{j,k,t}_{h}|_{f^{-1}(\gamma)})&=&\deg_{\pi_3^{-1}(\gamma)}(M^{j,k,t}_{C_2\times\Gamma})&=&k,\\
\deg_{\Gamma}(\mt M^{j,k,t}_{h}|_{f^{-1}(\gamma)})&=&\deg_{T}(M^{j,k,t}_{Y_2})&=&t,\\
\deg_E(\mt M^{j,k,t}_{h}|_{f^{-1}(\gamma)})&=&\deg_E(M^{j,k,t}_{Y_2})&=&d-j-k-t.
\end{array}$$
In particular $\mt M_{h}^{j,k,t}$ has degree $d$.

\begin{lem}\label{jkt} If $j$, $k$, $t$ satisfy:
$$\left| j-d\frac{2h-1}{2g-2}\right|\leq\frac{r}{2},\quad\left|k-d\frac{2g-2h-3}{2g-2}\right|\leq\frac{r}{2},\quad \left|t-\frac{d}{2g-2}\right|\leq\frac{r}{2},$$
then $\mt M_h^{j,k,t}$ is properly balanced.
\end{lem}

\begin{proof}We can reduce to check the condition on the special fiber. By Lemma \ref{balcon}, it is enough to check the inequalities on $C_1$, $C_2$ and $\Gamma$; this follows from the numerical assumptions.
\end{proof}

For any $1\leq h\leq \frac{g-2}{2}$ choose $j(h)$, resp. $t(h)$, satisfying the first, resp. third, inequality of Lemma \ref{jkt} (observe that such numbers are not unique in general). For every $k\in J_{h+1}$ we call $\widetilde {F}_h^k$ the family $f:X\rightarrow\Gamma$ with the properly balanced vector bundle $$\mt M_{h}^{j(h),\lfloor d\frac{2g-2h-3}{2g-2} +\frac{r}{2}\rfloor-k, t(h)}.$$ As before, we compute the degree of the boundary line bundles to the curves $\widetilde {F}_{h}^k$ using the fact that they are liftings of families $F_h$ of \cite[p. 156]{AC87}. We get
$$
\begin{cases}
\deg_{\widetilde{F}_{h}^k}\oo(\widetilde\delta_{h+1}^k)=-1,\\
\deg_{\widetilde{F}_{h}^k}\oo(\widetilde\delta_{h+1}^j)=0 & \mbox{if }j\neq k,\\
\deg_{\widetilde{F}_{h}^k}\oo(\delta_i^j)=0 & \mbox{if } h+1<i,\mbox{ for any }j\in J_i.
\end{cases}
$$
Indeed, the first two relations follow by 
$$
\deg_{\widetilde{F}_{h}^k}\oo\left(\sum_{j\in J_{h+1}}\widetilde\delta_{h+1}^j\right)=\deg_{F_h}\oo(\delta_{h+1})=-1
$$
(see \cite[p. 157]{AC87}) and the fact that $\widetilde{F}_h^k$ does not meet $\widetilde\delta_{h+1}^j$ for $j\neq k$. The last follows from the fact that $\widetilde{F}_h^k$ does not meet $\widetilde\delta_i^j$ for $i>h+1$.\\\\
\begin{dimo} \emph{\ref{indbou}.} We know that there exists $n_*$ such that $\CVc$ and $\overline{\mt U}_{n}$ have the same Picard groups for $n\geq n_*$. We can suppose $n_*$ big enough such that the families constructed  before define curves in $\overline{\mathcal U}_n$ for $n\geq n_*$. Suppose that there exists a linear relation
$$
\oo\left(\sum_i\sum_{j\in J_i}a_i^j\widetilde\delta_i^j\right)\cong\oo\in \Pic(\overline{\mt U}_n),
$$
where $a_i^j$ are integers. Pulling back to the curve $\widetilde F\rightarrow \overline{\mt U}_n$ we deduce $a^0_0=0$. Pulling back to the curves $\widetilde{F'}_1^j\rightarrow \overline{\mt U}_n$ or $\widetilde{F'}_2^j\rightarrow \overline{\mt U}_n$ (in the range of degrees where they are defined) we deduce $a_1^j=0$ for any $j\in J_1$. Pulling back to the curve $\widetilde{F}_{h}^j\rightarrow \overline{\mt U}_n$ we deduce $a_{h+1}^j=0$ for any $j\in J_{h+1}$ and $1\leq h\leq \frac{g-2}{2}$. This concludes the proof.
\end{dimo}\\

We have a similar result for the rigidified stack $\CVr$.
\begin{cor}\label{boundrig}
We have an exact sequences of groups
$$
0\longrightarrow\bigoplus_{i=0,\ldots,\lfloor g/2\rfloor}\oplus_{j\in J_i}\langle\oo(\widetilde\delta_i^j)\rangle\longrightarrow \Pic(\CVr)\longrightarrow \Pic(\Vr)\longrightarrow 0,
$$
where the right map is the natural restriction and the left map is the natural inclusion.
\end{cor}

\begin{proof}As before the only thing to prove is the independence of the boundary line bundles in $\Pic(\CVr)$. The same strategy (using the same families) used for $\CVc$ works for the rigidified case.
\end{proof}

\begin{rmk}\label{hstable}
As observed before, the boundary line bundles are independent on the Picard groups of $\CVc$, $\CVc^{Pss}$, $\CVr$, $\CVr^{Pss}$. A priori we do not know if $\widetilde\delta_i^j$ is a divisor of $\CVc^{Hss}$ for any $i$ and $j\in J_i$, because it can be difficult to check when a point $(C,\mt E)$ is H-semistable if $C$ is singular. But as explained in Remark \ref{defhsem}, if $(C,\mt E)$ is P-stable then it is also H-stable. By Proposition \ref{T}, we know that if $\widetilde\delta_i^j$ is a non-extremal boundary divisor the generic point of $\widetilde\delta_i^j$ is P-stable, in particular it is H-stable. 
\end{rmk}
The end of the section is devoted to prove that also the extremal boundary divisors are in $\CVc^{Hss}$. More precisely the generic points of the extremal divisors in $\CVc$ are strictly $H$-semistable. We will use the following criterion to prove strict H-semistability.

\begin{lem}\label{sHs}Assume that $g\geq 2$. Let $(C,\mt E)\in \CVc$ such that $C$ is the union of two irreducible smooth curves $C_1$ and $C_2$ of genus $1\leq g_{C_1}\leq g_{C_2}$ meeting at
$N$ points $p_1,\ldots,p_N$. Suppose that $\mt E_{C_1}$ is a direct sum of stable vector bundles with the same rank $q$ and same degree $e$ such that $e/q$ is equal to the slope of $\mt E_{C_1}$ and $\mt E_{C_2}$ is a stable vector bundle. If $\mt E$ has multidegree
$$
(deg_{C_1}\mt E,deg_{C_2}\mt E)=\left(d\frac{\omega_{C_1}}{\omega_C}-N\frac{r}{2},d\frac{\omega_{C_2}}{\omega_C}+N\frac{r}{2}\right)\in \mathbb Z^2
$$
then $(C,\mt E)$ is strictly P-semistable and strictly H-semistable.
\end{lem}

\begin{proof}By Proposition \ref{T}, $\mt E$ is P-semistable. We observe that the multidegree condition is equivalent to $\omega_C\chi(\mt E_{C_1})=\omega_{C_1}\chi(\mt E)$, so $\mt E$ is strictly P-semistable. Suppose that $\mt M$ is a destabilizing subsheaf of $\mt E$ of multirank $(m_1,m_2)$. Consider the exact sequence 
$$
\xymatrix{
0\ar[r] & \mt E_{C_2}(-\sum_1^N p_i)\ar[r] &\mt E\ar[r] &\mt E_{C_1}\ar[r] & 0\\
0\ar[r] & \mt  M_2\ar[r]\ar@{^{(}->}[u] &\mt M\ar[r]\ar@{^{(}->}[u] &\mt M_1\ar[r]\ar@{^{(}->}[u] & 0
}
$$
From this we have
$$
\chi(\mt M)=\chi(\mt M_1)+\chi(\mt M_2)\leq \frac{m_1}{r}\chi (\mt E_{C_1})+\frac{m_2}{r}\chi (\mt E_{C_2}(-\sum p_i))=\frac{m_1\omega_{C_1}+m_2\omega_{C_2}}{r\omega_C}\chi (\mt E).
$$
By hypothesis, $\mt E_{C_2}$ stable. So we have two possibilities: $\mt M_2$ is $0$ or $\mt E_{C_2}(-\sum_1^N p_i)$, because otherwise the inequality above is strict. Suppose that $\mt M_2=0$. Then $\mt M=\mt M_1$ which implies that $\mt M\subset \mt E_{C_1}(-\sum p_i)$ so the inequality above is strict. Thus we have just one possibility: if $\mt M$ is destabilizing sheaf then $\mt E_{C_2}(-\sum_1^N p_i)\subset \mt M$.\vspace{0.1cm}

Now we show that $(C,\mt E)$ is strictly H-semistable. We remark that, in our case, $(C,\mt E)=(C^{st},\pi_*\mt E)$. By hypothesis, $\mt E_{C_1}=\bigoplus_{i=0}^k\mt G_i$, where $\mt G_0=0$ and $\mt G_i$ stable bundle of rank $q$ and same slope of $\mt E_{C_1}$. Let $\mt F_\bullet$ be a Jordan-Holder filtration of $\mt E$. Then $\mt F_j/\mt E_{C_2}(-\sum_i^Np_i)\cong\bigoplus_{i=0}^j\mt G_i$. By definition of H-(semi)stability (see Definition \ref{defhsemtrue}), for every vector $\alpha_\bullet$ of positive rational numbers, there exists an index $n^*$, such that for any $n\geq n^*$
$$
P_{\alpha_\bullet,\mt F_\bullet}(m,n)\underset{(\leq)}{<}0\text{ as polynomial in } m.
$$
We now explicitly compute this polynomial. Fix $\alpha_\bullet:=(\alpha_1,\ldots,\alpha_k)$ positive integers. Let $B_n$ be a basis of $V_n:=H^(C,\mt E(n))$ compatible with the filtration $0\subset H^0(\mt F_{0}(n))\subset \ldots h^0(\mt F_{k-1}(n))\subset h^0(\mt F_{k}(n))=V_n$. Without loss of generality, we can assume $h^0(\mt F_j(n))=\chi(\mt F_j(n))$ for any $j$.

Following the definition of the polynomial $P_{\alpha_\bullet,\mt F_\bullet}(m,n)$ (see the discussion before Definition \ref{defhsemtrue}), we define a one-parameter subgroup $\lambda$ of $GL(V_n)$, given, with respect to the basis $B_n$, by the weight vector
$$
\sum_{i=1}^{\dim V_n-1}\alpha_i(\underbrace{\chi(\mt F_i(n))-\chi(\mt E(n)),\ldots,\chi(\mt F_i(n))-\chi(\mt E(n))}_{\chi(\mt F_i(n))},\underbrace{\chi(\mt F_i(n)), \ldots,\chi(\mt F_i(n))}_{\chi(\mt F_i(n))-\chi(\mt E(n))}).
$$

If we set $\widetilde Z_j:=\langle v_{\chi(\mt F_i(n))+1},\ldots,v_{\chi(\mt F_i(n))}\rangle$ for $j=1,\ldots,k$ and $\widetilde Z_0:=H^0(\mt F_0(n))$ we have
$$
\bigwedge^rV_n=\bigoplus_{\rho_0,\ldots,\rho_{k}|\sum\rho_j=r}W_{\rho_0,\ldots,\rho_{k}}, \quad\text{           where   }\quad
W_{\rho_0,\ldots,\rho_{k}}:=\bigwedge^{\rho_0}\widetilde Z_0\otimes\ldots\otimes\bigwedge^{\rho_{k}}\widetilde Z_{k}.
$$
An element of the basis $\bigwedge^rB_n$ contained in $W_{\rho_1,\ldots,\rho_{k}}$ has weight $w_{\rho_0,\ldots,\rho_{k}}(n)=\rho_0\gamma_0(n)+\ldots+\rho_{k}\gamma_{k}(n)$. Where $\gamma_j(n)$ is the weight of an element of $B_n$ inside $\widetilde Z_j$, i.e.
$$
\gamma_j(n)=\sum_{i=0}^{k-1}\alpha_i\chi(\mt F_i(n))-\sum_{i=j}^{k}\alpha_i\chi (\mt E(n)),\quad \text{where}\quad\alpha_k=0.
$$
As in \cite[p. 186-187]{Sc} the space of minimal weights, which produces sections which do not vanish on $C_1$, is $W^1_{min}:=W_{0,q,\ldots,q}$. The associated weight is 
$$
w_{1,min}(n):=\sum_{i=0}^{k-1}\alpha_i\left(\chi(\mt F_i(n))r-\chi(\mt E(n))iq\right).
$$
Moreover a general section of $W^1_{min}$ does not vanish at $p_i$. By \cite[Corollary 2.2.5]{Sc}, the space $S^mW^1_{min}$ generates $
H^0\left(C_1,\left(\mt L_{n|C_1}\right)^m\right)
$, so that the elements of $S^m\bigwedge^rB_n$ inside $S^mW^1_{min}$ will contribute with weight
$$
K_1(n,m):=m(m(\deg\mt E_{C_1}+nr\omega_{C_1})+1-g_{C_1})w_{1,min}(n),
$$
to a basis of $H^0\left(C_1,\left(\mt L_{n|C_1}\right)^m\right)$. On the other hand the space of minimal weights which produces sections which do not vanish on $C_2$ is $W^2_{min}:=W_{r,0,\ldots,0}$. The associated weight is 
$$w_{2,min}(n):=\sum_{i=0}^{k-1}\alpha_i(\chi(\mt F_i(n))-\chi(\mt E(n)))r.$$
A general section of $W^2_{min}$ vanishes at $p_i$ with order $r$. By \cite[Corollary 2.2.5]{Sc}, the space $S^mW^2_{min}$ generates  $H^0\left(C_1,\left(\mt L_{n|C_2}\left(-r\sum p_i\right)\right)^m\right)$, in particular the elements of $S^m\bigwedge^rB_n$ inside $S^mW^2_{min}$ will contribute with weight
$$
K_2(n,m):=m(m(\deg\mt E_{C_2}-rN+nr\omega_{C_2})+1-g_{C_2})w_{2,min}(n),
$$
to a basis of $H^0\left(C_2,\left(\mt L_{n|C_2}\right)^m\right)$. It remains to find the elements in $S^m\bigwedge^rB_n$ which produce sections of minimal weight in $H^0\left(C_2,\left(\mt L_{n|C_2}\right)^m\right)$ vanishing with order less than $mr$ on $p_i$ for $i=1,\ldots,N$. By a direct computation, we can see that the space of minimal weights which gives us sections with vanishing order $r-s$ at $p_i$ such that $tq\leq s\leq (t+1)q$ (where $0\leq t\leq k-1$) is
$$
\mathbb O_{r-s}:=W_{r-s,\underbrace{q\ldots,q}_t,s-tq,0,\ldots,0}.
$$
The associated weight is
$$
w_{2,p_i}^{r-s}(n):=\sum_{i=0}^{k-1}\alpha_i\left(\chi(\mt F_i(n))-\chi(\mt E(n))\right)r+\sum_{i=0}^{t}(s-iq)\alpha_i\chi(\mt E(n)).
$$
For any $0\leq \nu\leq mr-1$ and $1\leq i\leq N$, we must find an element of minimal weight in $S^m\bigwedge^rB_n$ which produces a section in $H^0\left(C_2,\left(\mt L_{n|C_2}\right)^m\right)$ vanishing with order $\nu$ at $p_i$. Observe first that we can reduce to check it on the subspace
$$
S^m\mathbb O=\bigoplus_{m_0,\ldots,m_r|\sum m_i=m}S^{m_0}\mathbb O_0\otimes\ldots\otimes S^{m_r}\mathbb O_r.
$$
A section in $S^{m_0}\mathbb O_0\otimes\ldots\otimes S^{m_r}\mathbb O_r$ vanishes with order at least $\nu=m_1+2m_2+\ldots+rm_r$ at $p_i$ and we can find some with exactly that order. As explained in \cite[p. 191-192]{Sc}, an element of $S^m\bigwedge^rB_n$ of minimal weight, such that it produces a section of order $\nu$ at $p_i$, lies in
$$
S^j\mathbb O_{t-1}\otimes S^{m-j}\mathbb O_t,
$$
where $\nu=mt-j$ and $1\leq j\leq m$. So the mininum among the sums of the weights of the elements in $S^m\bigwedge^r B_n$ which give us a basis of 
$$
H^0\left(C_2,\left(\mt L_{n|C_2}\right)^m\right)/H^0\left(C_2,\left(\mt L_{n|C_2}(-r\sum p_i)\right)^m\right)
$$
is
\begin{eqnarray*}
D_2(n,m)&=&N\left(m^2(w_{2,p_i}^1(n)+\ldots+w_{2,p_i}^r(n))+\frac{m(m+1)}{2}(w_{2,p_i}^0(n)-w_{2,p_i}^r(n))\right)=\\
 &=&N\sum_{i=0}^{k-1}\alpha_i\left(m^2\left(\chi(\mt F_i(n))r^2-\chi(\mt E(n))\left(r^2-\frac{(r-iq-1)(r-iq)}{2}\right)  \right)\right. +\\
& & \left.+ \frac{m(m+1)}{2}(r-iq)\chi(\mt E(n))\right).
\end{eqnarray*}
Then a basis for $H^0\left(C_2,\left(\mt L_{n|C_2}\right)^m\right)$ will have minimal weight $K_2(n,m)+D_2(n,m)$.

As in \cite[p.192-194]{Sc}, we obtain that $(C,\mt E)$ is H-semistable if and only if for any choice of $\alpha_\bullet$, there exists $n^*$ such that for $n\geq n^*$
$$
P_{\alpha_\bullet,\mt F_\bullet}(m,n)=K_1(n,m)+K_2(n,m)+D_2(n,m)-mNw_{1,min}(n)\leq 0,
$$
as polynomial in $m$. A direct computation shows that $P_{\alpha_\bullet,\mt F_\bullet}(m,n)\leq 0$ as polynomial in $m$. So $(C,\mt E)$ is H-semistable. Furthermore, if we take $\alpha_\bullet=(1,0,\ldots,0)$, by the same computation we can check that $P_{\alpha_\bullet,\mt F_\bullet}(m,n)\equiv 0$, then $(C,\mt E)$ is not H-stable.
\end{proof}

\begin{prop}\label{hextr}The generic point of an extremal boundary divisor is strictly P-semistable and strictly H-semistable.
\end{prop}

\begin{proof}Fix $i\in\{1,\ldots,\lfloor g/2\rfloor\}$ such that $\widetilde\delta_i^0$ is an extremal divisor. By Lemma \ref{T2}, the generic point of the extremal boundary $\widetilde\delta_i^0$ is a curve $C$ with two irreducible smooth components $C_1$ and $C_2$ of genus $i$ and $g-i$ and a vector bundle $\mt E$ such that $\mt E_{C_1}$ is a stable vector bundle (or a direct sum of stable vector bundles with the same slope as $\mt E_{C_1}$ if $i=1$) and $\mt E_{C_2}$ is stable vector bundle. By  Lemma \ref{sHs}, the generic point of $\widetilde\delta_i^0$ is strictly P-semistable and stricly H-semistable.\\
Suppose now that $i\neq g/2$ and consider the extremal boundary divisor $\widetilde\delta_i^r$. Take a point $(C,\mt E)\in \widetilde\delta_i^0$ as above. Consider the destabilizing subsheaf $\mt E_{C_2}(-p)\subset \mt E$, where $p$ is the unique node of $C$. Fix a basis of $W_n:=H^0(C,\mt E(n)_{C_2}(-p))$ and complete to a basis $\mt V:=\{v_1,\ldots,v_{dim V_n}\}$ of $V_n=H^0(C,\mt E(n))$. We define the one-parameter subgroup $\lambda$ of $SL(V_n)$ given with respect to the basis $\mt V$ by the weight vector
$$
(\underbrace{\dim W_n-\dim V_n,\ldots,\dim W_n-\dim V_n}_{\dim W_n},\underbrace{\dim W_n, \ldots,\dim W_n}_{\dim W_n-\dim V_n}).
$$
We have seen in the proof of Lemma \ref{sHs} that the pair $(C,\mt E)$ is strictly H-semistable and the polynomial attached to $\lambda$ vanishes. In particular, the limit with respect to $\lambda$ is strictly H-semistable. The limit will be a pair $(C',\mt E')$ such that $C'$ is a semistable model for $C$ and $\mt E'$ a properly balanced vector bundles such that the push-forward in the stabilization is the P-semistable sheaf $$\mt E_{C_2}(-p)\oplus\mt E_{C_1}.$$ 
By Corollary \ref{NSrmk}, $\mt E'_{C_1}\cong \mt E_{C_1}$ and $\mt E'_{C_2}\cong \mt E_{C_2}(-p)$. In particular, $\mt E$ has multidegree
$$
\left(\deg\mt E'_{C_1},\deg\mt E'_R,\deg\mt E'_{C_2}\right)=\left(\deg_{C_1}\mt E,r,\deg_{C_2}\mt E-r\right)=\left(d\frac{2i-1}{2g-2}-\frac{r}{2}, r,d\frac{2(g-i)-1}{2g-2}-\frac{r}{2}\right).
$$
Smoothing all the nodal points on the rational chain $R$ except the meeting point $q$ between $R$ and $C_2$, we obtain a generic point $(C'',\mt E'')$ in $\widetilde\delta_i^r$. It is H-semistable by the openness of the semistable locus. Let $W''_n$ be a basis for $H^0\left(C'',\mt E''_{C_1}(n)(-q)\right)$ and complete to a basis $\mt V''_n$ of $H^0(C'',\mt E''(n))$. Let $\lambda''$ be the one parameter subgroup defined by the weight vector (with respect to the basis $B$)
$$
(\underbrace{\dim W''_n-\dim V''_n,\ldots,\dim W''_n-\dim V''_n}_{\dim W''_n},\underbrace{\dim W''_n, \ldots,\dim W''_n}_{\dim W''_n-\dim V''_n}).
$$
As in the proof of Lemma \ref{sHs}, a direct computation shows that the polynomial attached to $\lambda''$ vanishes, so $(C'',\mt E'')$ is strictly H-semistable, then it is also strictly P-semistable, concluding the proof.
\end{proof}
 
Using this, we obtain

\begin{cor}\label{boundrigH}The followings hold:
\begin{enumerate}[(i)]
\item the Picard group of $\CVc$ is isomorphic to the Picard group of $\CVc^{Hss}$,
\item the kernels of the restriction to the stable loci 
$$
Pic(\CVc^{Hss})\longrightarrow Pic(\CVc^{Hs})\longrightarrow 0,\quad Pic(\CVc^{Pss})\longrightarrow Pic(\CVc^{Ps})\longrightarrow 0
$$
are freely generated by the extremal boundary divisors.
\end{enumerate}
The same facts hold for the rigidifications.
\end{cor}

\begin{proof}(i). It is enough to show $\text{cod}:=\text{cod}(\CVc\backslash\CVc^{Hss},\CVc)\geq 2$. As in the proof of Lemma \ref{redsemistable}, if we have $\text{cod}=1,$ then the intersection of $\CVc\backslash\CVc^{Hss}$ with the generic point of the boundary $\widetilde\delta$ is not-empty. By Remark \ref{hstable} and Proposition \ref{hextr}, the generic point of the boundary is H-semistable, then $\text{cod}$ must be $\geq 2$.

(ii). By Remark \ref{hstable} and Proposition \ref{hextr}, we know that the 1-codimensional loci of strictly H-semistable (resp. P-semistable) pairs are the extremal boundary divisors. So, the second assertion follows from Theorems \ref{picchow}(iv) and \ref{indbou}. 

The same arguments give the proof for the rigidifications.
\end{proof}

\subsection{Picard group of $\Vc$.}\label{liscio}
In this section we will prove Theorem \ref{pic}. First of all, we will prove the following

\begin{teo}\label{picsmooth}Assume that $g\geq 2$. For any smooth curve $C$ with a line bundle $\mathcal L$ of degree $d$, we have an exact sequence.
$$
0\lra Pic(\Jc)\lra Pic(\Vc^{ss})\lra Pic(\Vl^{ss})\lra 0,
$$
where the left map is the pull-back via the determinant map $det:\Vc^{ss}\to \Jc$ and the right map is the pull-back on the geometric fiber of $det$ over $(C,\mt L)\in\Jc$.
\end{teo}

Let $\jc$ (resp. $\jr$) be the substack of $\Jc$ (resp. $\Jr$) which parameterizes the pairs $(C,\mt L)$ such that $\Aut(C,\mt L)=\mathbb G_m$.

\begin{lem}\label{no-aut}The stack $\jc$ (resp. $\jr$) is an open substack of $\Jc$ (resp. $\Jr$), such that the complementary closed substack in $\Jc$ (resp. $\Jr$) has codimension at least two.

Furthermore, $\jr$ is a smooth irreducible variety, which parameterizes isomorphism classes of line bundles of degree $d$ over a curve $C$, such that $\Aut(C,\mt L)=\mathbb G_m$.

\end{lem}

\begin{proof}We will prove the lemma for $\jr$, the assertion for $\jc$ will follow directly. Consider the cartesian diagram
$$
\xymatrix{
\mt I_{\Jr}\ar[r]^\Xi\ar[d]^\Xi&\Jr\ar[d]^\Delta\\
\Jr\ar[r]^\Delta&\Jr\times\Jr
}
$$
Where $\Delta$ is the diagonal morphism. The stack $\mt I_{\Jr}$ is called inertia stack. A fibre of the morphism $\Xi$ over the a point $(C,\mt L)$ corresponds to the group $\text{Aut}(C,\mt L)/\mathbb G_m$. In particular, the locus in $\Jc$ where $\Xi$ is an isomorphism is exactly $\jr$ and so it is an open.

It remains to show the part about the codimension. First we recall some facts about curves with non-trivial automorphisms. The closed locus $\Jr^{Aut}$ in $\Jr$ of curves with non-trivial automorphisms has codimension $g-2$ and it has a unique irreducible component $\mt J\mt H_{d,g}$ of maximal dimension  corresponding to the hyperelliptic curves (see \cite[Remark 2.4]{GV}). Moreover, the closed locus $\mt{JH}_{d,g}^{extra}\subset\mt{JH}_{d,g}$ of hyperelliptic curves with extra-automorphisms has codimension $2g-3$ in $\Jr$ and it has a unique irreducible component of maximal dimension corresponding to the curves with an extra-involution (for details see \cite[Proposition 2.1]{GV}). By definition, $\Jr\backslash\jr\subset \Jr^{Aut}$. By the above facts, it is enough to check the dimension of $\Jr^*:=(\Jr\backslash\jr)\cap\mt J\mt H_g\subset\mt J\mt H_g$. We remark that $\Jr^*$ is the locus of pairs $(C,\mt L)$ such that $C$ is hyperelliptic and $\Aut(C,\mt L)\neq\mathbb G_m$.

If $g\geq 4$, $\mt{JH}_{d,g}$ has codimension $\geq2$, then the lemma follows. 

If $g=3$, $\mt J\mt H_{d,3}$ is an irreducible divisor. It is enough to show that $\mt J_{d,3}^*\neq\mt{JH}_{d,3}$. Let $C$ be an hyperelliptic curve with no extra-automorphisms and $i$ be the hyperelliptic involution. Assume that $\text{Aut}(C,\mt L)\neq \mathbb G_m$ ($\Leftrightarrow i^*\mt L\cong\mt L$). Then for any other $(C,\mt M)$ such that $i^*\mt M\cong \mt M$, then $\mt M=\mt L\otimes T$ with $\mt T^2\cong \oo_C$. In particular, $\mt J_{d,3}^*\neq\mt{JH}_{d,3}$, because the latter contains the entire Jacobian $J^d(C)$.

If $g=2$, then all the curves are hyperelliptic, $\dim\mt J_{d,2}=5$ and $\mt{JH}_{d,2}^{extra}$ has codimension $1$. Consider the forgetful morphism $\mt J_{d,2}^*\rightarrow\mt M_2$.
The fiber at $C$, when is not empty, is the closed subscheme of the Jacobian $J^d(C)$ where the action of $\Aut(C)$ is not free. If $C$ is a curve without extra-automorphisms then the fiber has dimension $0$: it is the finite group of $2$-torsion points in the Jacobian. In particular, if the open locus of these curves is dense in $\mt J_{d,2}^*$, then $\dim\mt J_{d,2}^*\leq\dim \mt M_2=3$ and the lemma follows. Otherwise, $\mt J_{d,2}^*$ can have an irreducible component of maximal dimension which maps in the divisor $\mt {H}^{extra}_{d,2}\subset\mt M_2$ of curves with an extra-involution. In this case $\dim\mt J_{d,2}^*<\dim \mt {H}^{extra}_{d,2}+\dim J^d(C)=4$, which concludes the proof.

The last assertion of the lemma comes from the fact that $\Jr$ admits an irreducible variety $J_{d,g}$ as good moduli space. They are isomorphic along the locus of objects where the only automorphisms are the multiplications by a scalar, i.e. $\jr$. The smoothness of $\jr$ follows from the smoothness of $\Jr$.
\end{proof}

We denote by $\vc$ (resp. $\vr$) the open substack of $\Vc^{ss}$ (resp. $\Vr^{ss}$) of pairs $(C,\mt E)$ such that $\Aut(C,\det\mt E)=\mathbb G_m$. 

\begin{cor}\label{iso-rest}We have an isomorphisms of Picard groups $$\Pic(\Jc)\cong\Pic(\jc),\quad\Pic(\Vc)\cong\Pic(\vc).$$ The same holds for the rigidifications.
\end{cor}

\begin{proof}Observe that $\vc$ is the inverse image of $\jc$ along the determinant morphism $det:\Vc\to\Jc$. In particular, we have
$$
\dim\left(\Vc\backslash\vc\right)\leq\dim\left(\Jc\backslash\jc\right)+\dim\Vl^{ss},
$$
for some pair $(C,\mt L)\in \Jc\backslash\jc$. By Lemma \ref{no-aut},
$$
\begin{array}{lll}
 \text{cod}\left(\Vc\backslash\vc,\Vc\right)&=&\dim\Jc+\dim\Vl^{ss}-\dim\left(\Vc\backslash\vc\right) \geq\\
 &\geq &\text{cod}\left(\Jc\backslash\jc,\Jc\right)\geq 2.
\end{array}
$$
By Theorem \ref{picchow}(ii), we have the isomorphism between $\Vc$ and $\vc$.
\end{proof}

By the above corollary, Theorem \ref{picsmooth} is equivalent to prove the exactness of
\begin{equation}\label{seq0}
0\lra  \Pic(\jc)\lra \Pic(\vc)\lra \Pic(\mt V ec^{ss}_{=\mt L,C})\lra 0.
\end{equation}
The first step is to prove the following

\begin{lem}The sequence (\ref{seq0}) is left exact.
\end{lem}

\begin{proof}Observe that $det:\Vc\to\Jc$ admits sections. For instance, $(C,\mt L)\to (C,\mt L\oplus\oo_C^r)$. In particular, the homomorphism of groups $det^*:\Pic(\Jc)\to\Pic(\Vc)$ is injective. Then the lemma follows from Corollary \ref{iso-rest}.
\end{proof}

The morphism $det:\vc\longrightarrow \jc$ is a morphism of smooth Artin stacks locally of finite type, with smooth fibre of the same dimension (see \S\ref{fibre}). By miraculous flatness, the map $det$ is a flat. Let $\Lambda$ be a line bundle over $\vc$, which is obviously flat over $\jc$ by the flatness of the map $det$.

\begin{lem}Suppose that $\Lambda$ is trivial over any geometric fiber. Then $det_*\Lambda$ is a line bundle on $\jc$ and the natural map $det^*det_*\Lambda\longrightarrow \Lambda$ is an isomorphism.
\end{lem}

\begin{proof}
Consider the cartesian diagram
$$
\xymatrix{
\vc\ar[d] &V_H\ar[l] \ar[d]\\
\jc & H\ar[l] }
$$
where the bottom row is an atlas for $\jc$. We can reduce to control the isomorphism locally on $V_H\to H$. Suppose that the following conditions hold
\begin{enumerate}[(i)]
\item $H$ is an integral scheme,
\item the stack $V_H$ has a good moduli scheme $U_H$,
\item $U_H$ is proper over $H$ with geometrically integral fibers.
\end{enumerate}
Then, by Seesaw Principle (see Corollary \ref{seesaw}), we have the assertion. So it is enough to find an atlas $H$ such that the conditions (i), (ii) and (iii) are satisfied.

We fix some notations: since the stack $\vc$ is quasi-compact, there exists $n$ big enough such that $\vc\subset \overline{\mt U}_n=[H_n/GL(V_n)]$. The twist by a power of the canonical bundle $\omega^n$, induces an isomorphism of stacks $\Vc\cong \mt Vec_{r,d+rn(2g-2),g}$. In particular, if $\vc\subset \overline{\mt U}_n$, then $\mt Vec^o_{r,d+rn(2g-2),g}\subset\overline{\mt U}_0$. So, we can assume $d$ big enough such that $\Vc^o\subset\overline{\mt U}_0$.\vspace{0.2cm}

For $r\geq 2$, we denote by $Q$ the $G$-stable open subset of $H_0$ such that $\vc=[Q/G]\subset \overline{\mt U}_0=[H_0/G]$, where $G:=GL(V_0)$. We will use a different notation when $r=1$, we replace the symbol $Q$ with $H$ and the symbol $G$ with $\Gamma$. Denote by $Z(\Gamma)$ (resp. $Z(G)$) the center of $\Gamma$ (resp. of $G$) and set $\widetilde{G}=G/Z(G)$, $\widetilde{\Gamma}=\Gamma /Z(\Gamma)$. Note that $Z(G)\cong Z(\Gamma)\cong\mathbb G_m$. As usual we set $\mt BZ(\Gamma):=[\Spec\, k/Z(\Gamma)]$. Since $\jc$ is integral then $H$ is integral, satisfying the condition (i). 
We have the following cartesian diagrams
$$
\xymatrix@!0@R=3pc@C=5pc{
 & Q\ar[d] & & Q\times_{\jr} H\ar[ll]\ar[d] & & Q\times_{\jc} H\ar[ll]_{\pi}\ar[d]\\
& \vc\ar[dl]\ar@{-}[d] & & \left[\left(Q\times_{\jr}H\right)\Big/G\right]\ar[ll]\ar[dl]^q\ar[ddd] & & \left[\left(Q\times_{\jc}H\right)\Big/G\right]\cong V_H\ar[ddd]\ar[ll]_>>>>{p}\\
\vr\ar[d]&\ar@{-}[d] & \left[\left(Q\times_{\jr}H\right)\Big/\widetilde G\right]\ar[ll]\ar[d]\\
U^o_{r,d,g}\ar[dd] &\ar[d] & U_H\ar[dd]\ar[ll]\\
& \jc\ar[dl]&\ar[l] & H\times \mt BZ(\Gamma)\ar[dl]\ar@{-}[l] & &H\ar[ll]\\
\jr & & H\ar[ll]\ar@{=}[urrr]
}
$$
where $U^o_{r,d,g}$ is the open subscheme in the Schmitt compactification $\overline{U}_{r,d,g}$ (see Theorem \ref{schmitt}) of pairs $(C,\mt E)$ such that $C$ is smooth, $\mt E$ is polystable and $\Aut(C,\det\mt E)=\mathbb G_m$. Note that $U_H$ is proper over $H$, because $U_{r,d,g}^o\rightarrow \jr$ is proper. In particular, the geometric fiber over a $k$-point of $H$ which maps to $(C,\mt L)$ in $ \jr$ is the irreducible projective moduli variety $U_{\mt L,C}$ of semistable vector bundles on $C$ with determinant isomorphic to $\mt L$ (see \S\ref{fibre}). So, $U_H\to H$ satisfies the property (ii).

It remains to prove that $U_H$ is a good moduli space for $V_H$. Since $V_H$ is a quotient stack, it is enough to show that $U_H$ is a good $G$-quotient of $Q\times_{\jc}H$. Consider the commutative diagram
$$
\xymatrix{
Q\times_{\jc}H\ar[r]^\pi\ar[rd]^\beta & Q\times_{\jr} H\ar[r]\ar[d]^\alpha & U_H\ar[ld]\\
& H &}
$$
\underline{Claim:} the horizontal maps make $U_H$ a categorical $G$-quotient of $Q\times_{\Jc^0}H$.\\ Suppose that the claim holds. Then $U_H$ is a good $G$-quotient also for $Q\times_{\jc}H$, because the horizontal maps are affine (see \cite[1.12]{Mum94}), and we have done.

We conclude the proof by proving the claim. The idea for this part comes from \cite[Section 2]{H}. Since the map $Q\rightarrow \jr$ is $G$-invariant then $Q\times_{\jr}H\rightarrow H$ is $G$-invariant. In particular, we can study the action of $Z(G)$ over the fibers of $\alpha$. Fix a geometric point $h$ on $H$ and suppose that its image in $\jc$ is the pair $(C,\mathcal L)$. Then  the fiber of $\beta$ (resp. of $\alpha$) over $h$ is the fine moduli space of the triples $(\mathcal E,B,\phi)$ (resp. of the pairs $(\mathcal E,B)$), where $\mathcal E$ is a semistable vector bundle on $C$, $B$ a basis of $H^0(C,\mathcal E)$ and $\phi$ is an isomorphism between the line bundles $\det\mt E$ and $\mathcal{L}$. If $g\in Z(G)$ we have $g.(\mathcal E,B,\phi)=(\mathcal E,gB,\phi)$ and $g.(\mathcal E,B)=(\mathcal E,gB)$. Observe that the isomorphism $g.Id_{\mathcal E}$ gives us an isomorphism between the pairs $(\mathcal E, B)$ and $(\mathcal E, gB)$ and between the triples $(\mathcal E, B,\phi)$ and $(\mathcal E, gB, g^r\phi)$. So $g.(\mathcal E,B,\phi)=(\mathcal E,B,g^{-r}\phi)$ and $g.(\mathcal E,B)=(\mathcal E,B)$. On the other hand, $\pi:Q\times_{\jc}H\rightarrow Q\times_{\jr}H$ is a principal $Z(\Gamma)$-bundle and the group $Z(\Gamma)$ acts as it follows: if $\gamma\in Z(\Gamma)$ we have $\gamma.(\mathcal E,B,\phi)=(\mathcal E,B,\gamma\phi)$ and $\gamma.(\mathcal E,B)=(\mathcal E,B)$.\\
This implies that the groups $Z(G)/\mu_r$ (where $\mu_r$ is the finite group of the $r$-roots of unity) and $Z(\Gamma)$ induce the same action on $Q\times_{\jc}H$. Since $\pi:Q\times_{\jc}H\rightarrow Q\times_{\jr}H$ is a principal $Z(\Gamma)$-bundle, any $G$-invariant morphism from $Q\times_{\jc}H$ to a scheme factorizes uniquely through $Q\times_{\jr}H$ and so uniquely through $U_H$ concluding the proof of the claim.
\end{proof}

The next lemma conclude the proof of Theorem \ref{picsmooth}.

\begin{lem}Let $\Lambda$ be a line bundle on $\vc$. Then 
$\Lambda$ is trivial on one geometric fiber of $det$ if and only if $\Lambda$ is trivial on any geometric fiber of $det$.
\end{lem}

\begin{proof} Consider the determinant map $det:\vc\rightarrow \jc$. Let $T$ be the set of points $h$ (in the sense of \cite[Chap. 5]{LMB}) in $\jc$ such that the restriction $\Lambda_h:=\Lambda_{det^*h}$ is the trivial line bundle. 
By Theorem \ref{fibers}(iii), the inclusion
$$
\Pic(U_{\mt L,C})\hookrightarrow\Pic(\Vl^{ss})\cong\mathbb Z
$$
is of finite index. The variety $U_{\mt L,C}$ is projective, in particular any non-trivial line bundle on it is ample or anti-ample. This implies that $\chi(\Lambda_h^n)$, as polynomial in the variable $n$, is constant if and only if $\Lambda_h$ is trivial. So $T$ is equal to the set of points $h$ such that the polynomial $\chi(\Lambda^n_h)$ is constant. Consider the atlas $H\rightarrow \jc$ defined in the proof of the previous lemma. The line bundle $\Lambda$ is flat over $\jc$ so the function
$$
\chi_n:H\rightarrow \mathbb{Z}:\,h=(C,\mathcal L,B)\mapsto\chi(\Lambda^n_h)
$$
is locally constant for any $n$, then constant because $H$ is connected. Therefore, the condition $\chi_n=\chi_m$ for any $n,\, m\in\mathbb Z$ is either always satisfied or never satisfied, which concludes the proof.
\end{proof}

Before to present the proof of Theorem \ref{pic}, we need to compute the homomorphism of Picard groups induced by the determinant $det:\Vc\to\Jc$. Following the notations of \S\ref{jc}, over the stack $\Jc$ we have the line bundles $\Lambda(n,m):=d_{\pi_1}(\omega_{\pi_1}^n\otimes\mt L^n)$, where $\pi_1:\mt Jac_{r,d,g,1}\to\Jc$ is the universal curve. 

On the other hand, when $r>1$, on $\CVc$ we have the tautological line bundles $\Lambda(n,m,l):=d_{\pi_r}(\omega_{\pi_r}^n\otimes(\det\mt E)^m\otimes\mt E^l)$ (see \S\ref{tautbun}), where $\pi_r:\mt Vec_{r,d,g,1}\to\Vc$ is the universal curve. The next lemma shows that the two notations are compatible.

\begin{lem}\label{computations}We have 
\begin{enumerate}[(i)]
\item $det^*\Lambda(n,m)=\Lambda(n,m,0)$,
\item the pull-back of $\Lambda(0,0,1)$ to $\Vl$, for $(C,\mt L)\in\Jc$, generates $\Pic(\Vl)$.
\end{enumerate}
\end{lem}

\begin{proof}We remark that both the diagrams
$$
\xymatrix{
\mt Vec_{r,d,g,1}\ar[r]^{\widetilde{det}}\ar[d]^{\pi_r} & \mt Jac_{r,d,g,1}\ar[d]^{\pi_1}\\
\Vc\ar[r]^{det} &\Jc
}\quad
\xymatrix{
\Vl\times C\ar[r]^{\widetilde{i}}\ar[d]^{\pi} & \mt Vec_{r,d,g,1}\ar[d]^{\pi_r}\\
\Vl\ar[r]^{i} &\Vc
}
$$
are cartesian. Since $det$ sends $(X,\mt E)$ to $(X,\det\mt E)$, we have $
\widetilde{det}^*\left(\omega_{\pi_1}^n\otimes\mt L^n\right)=(\omega_{\pi_r}^n\otimes(\det\mt E)^n)
$. Then (i) follows by Theorem \ref{detcoh}(ii). In a similar way, since $i$ sends $(\mt E,\varphi)$ to $(C,\mt E)$, we have that $\widetilde{i}^*\mt E$ is the universal vector bundle for the stack $\Vl$. Then by Theorem \ref{detcoh}(ii), we see that $i^*d_{\pi_r}(\mt E)$ is exactly the generator described in Theorem \ref{fibers}.
\end{proof}

We are finally ready for\vspace{0.1cm}\\
\begin{dimo} \emph{\ref{pic}}. \label{proofpic}The proof consists in combining some facts proved through the paper.
	
\ref{pic}(i). By Corollary \ref{redsemistable}, it is enough to show that the (i) holds for $\Vc^{ss}$. By Theorem \ref{picjac} and Lemma \ref{computations}, the exact sequence of Theorem \ref{picsmooth} can be rewritten as it follows:
$$
0\to\mathbb Z\langle \Lambda(1,0,0)\rangle\oplus \mathbb Z\langle \Lambda(0,1,0)\rangle
\oplus \mathbb Z\langle \Lambda(1,1,0)\rangle\to\Pic(\Vc^{ss})\to \mathbb Z\langle \Lambda(0,0,1)\rangle\to 0.
$$
So the point holds also for $\Vc^{ss}$.

\ref{pic}(ii). The assertion for $\CVc$ follows from Theorem \ref{pic}(i) and Theorem \ref{indbou}. Then the assertion for the P-semistable, resp. H-semistable, locus follows from Corollary, \ref{redsemistable}, resp. \ref{boundrigH}(i).

\ref{pic}(iii). The assertion follows from the previous points and Corollary \ref{boundrigH}(ii).
\end{dimo}

\subsection{Comparing the Picard groups of $\CVc$ and $\CVr$.}\label{comparing}
Assume that $g\geq 2$. Consider the rigidification map $\nu_{r,d}:\Vc\rightarrow\Vr$ and the sheaf of abelian groups $\mathbb G_m$. The Leray spectral sequence 
\begin{equation}\label{leray}
H^p(\Vr,R^q\nu_{r,d*}\mathbb{G}_m) \Rightarrow H^{p+q}(\Vc,\mathbb G_m)
\end{equation}
induces an exact sequence in low degrees
$$
0\rightarrow H^1(\Vr,\nu_{r,d*}\mathbb G_m)\rightarrow  H^1(\Vc,\mathbb G_m)\rightarrow H^0(\Vr,R^1\nu_{r,d*}\mathbb G_m)\rightarrow  H^2(\Vr,\nu_{r,d*}\mathbb G_m).
$$
Since the morphism $\nu_{r,d}$ is a $\mathbb G_m$-gerbe, locally \'etale it looks like $\mt B\mathbb G_m\times U\to U$. Using this, it can be checked that $\nu_{r,d*}\mathbb{G}_m=\mathbb G _m$ and that the sheaf $R^1\nu_{r,d*}\mathbb G_m$ is the constant sheaf $H^1(\mt B\mathbb G_m,\mathbb G_m)\cong \Pic(\mt B\mathbb G_m)\cong \mathbb Z$. Via standard cocycle computation we see that the exact sequence becomes
\begin{equation}\label{lss}
0\longrightarrow \Pic(\Vr)\longrightarrow \Pic(\Vc)\xrightarrow{res} \mathbb{Z}\xrightarrow{obs} H^2(\Vr,\mathbb G_m),
\end{equation}
where $res$ is the restriction on the fibers (it coincides with the weight map defined in \cite[Def. 4.1]{H07}), $obs$ is the map which sends the identity to the $\mathbb G_m$-gerbe class $[\nu_{r,d}]\in H^2(\Vr,\mathbb G_m)$  associated to $\nu_{r,d}:\Vc\rightarrow\Vr$  (see \cite[IV, $\S$3.4-5]{Gi}).

\begin{lem}We have that:
$$
\begin{cases}
res(\Lambda(1,0,0)) =0,\\
res(\Lambda(0,0,1))=d+r(1-g),\\
res(\Lambda(0,1,0))=r(d+1-g),\\
res(\Lambda(1,1,0))=r(d-1+g).
\end{cases}
$$
\end{lem}

\begin{proof}Using the functoriality of the determinant of cohomology, we get that the fiber of $\Lambda(1,0,0)=d_{\pi}(\omega_{\pi})$ over a point $(C,\mt E)$ is canonically isomorphic to
$\det H^0(C,\omega_C)\otimes \det^{-1}  H^1(C,\omega_C)$. Since $\mathbb G_m$ acts trivially on $H^0(C,\omega_C)$ and on
$H^1(C, \omega_C)$, we get that $res(\Lambda(1,0,0))=0$.\\
Similarly, the fiber of $\Lambda(0,0,1)$ over a point $(C,\mt E)$ is canonically isomorphic to
$\det H^0(C,\mt E)\otimes \det^{-1}  H^1(C,\mt E)$. Since $\mathbb G_m$ acts with weight one on the vector spaces
$H^0(C, \mt E)$ and $H^1(C,\mt E)$, Riemann-Roch gives that
$$res(\Lambda(0,0,1))=h^0(C,\mt E)-h^1(C,\mt E)=\chi(\mt E)=d+r(1-g).$$
The fiber of $\Lambda(0,1,0)$ over a point $(C,\mt E)$ is canonically isomorphic to
$\det H^0(C,\det\mt E)\otimes \det^{-1}  H^1(C,\det\mt E)$. Now $\mathbb G_m$ acts with weight $r$ on the vector spaces
$H^0(C, \det\mt E)$ and $H^1(C,\det\mt E)$, so that Riemann-Roch gives
$$res(\Lambda(0,1,0))=r\cdot h^0(C,\det\mt E)-r\cdot h^1(C,\det\mt E)=r\cdot\chi(\det\mt E)=r(d+1-g).$$
Finally, the fiber of $\Lambda(1,1,0)$ over a point $(C,\mt E)$ is canonically isomorphic
to $\det H^0(C,\omega_C\otimes\det\mt E)\otimes \det^{-1}  H^1(C,\omega_C\otimes\det\mt E)$. Since $\mathbb G_m$ acts with weight $r$ on the vector spaces $H^0(C, \omega_C\otimes\det\mt E)$ and $H^1(C,\omega_C\otimes\det\mt E)$, Riemann-Roch gives that
$$res(\Lambda(1,1,0))=r\cdot h^0(C,\omega_C\otimes\det\mt E)-r\cdot h^1(C,\omega_C\otimes\det\mt E)=
r\cdot\chi(\omega_C\otimes\det\mt E)=r(d-1+g). $$
\end{proof}

Combining the Lemma above with Theorem \ref{pic}(i) and the exact sequence (\ref{lss}), we obtain 

\begin{cor}\label{eccocequasi}
\noindent
\begin{enumerate}[(i)]
\item The image of $\Pic(\Vc)$ via the morphism $res$ of (\ref{lss}) is the subgroup of $\mathbb Z$ generated by $$n_{r,d}\cdot v_{r,d,g}=\big(d+r(1-g), r(d+1-g), r(d-1+g)\big).$$
\item The Picard group of $\Vr$ is (freely) generated by the line bundles $\Lambda(1,0,0)$, $\Xi$ and $\Theta$ (when $g\geq 3$).
\end{enumerate}
\end{cor}

With all these facts, we are now ready for\vspace{0.1cm}\\
\begin{dimo} \emph{\ref{picred}}.\label{proofpicred} The proof consists in combining some facts proved through the paper.
	
\ref{picred}(i). By Corollary \ref{eccocequasi}(ii), the point is true for the stack $\Vr$. The assertion about the (semi)stable locus comes from the rigidified version of Corollary \ref{redsemistable}.

\ref{picred}(ii). The assertion for $\CVr$ follows from the previous point and Corollary \ref{boundrig}. Then the assertion for the P-semistable, resp. H-semistable, locus follows from the rigidified version of Corollary \ref{redsemistable}, resp. \ref{boundrigH}(i).

\ref{picred}(iii). The assertion follows from the previous points and the rigidified version of Corollary \ref{boundrigH}(ii).
\end{dimo}

\begin{rmk}\label{kou}Let $U_{r,d,g}$ be the coarse moduli space of aut-equivalence classes of semistable vector bundles on smooth curves. Suppose that $g\geq 3$. Kouvidakis in \cite{Kou93} gives a description of the Picard group of the open subset $U^\star_{r,d,g}$ of curves without non-trivial automorphisms. As observed in Section $3$ of \emph{loc. cit.}, such locus is locally factorial. Since the locus of strictly semistable vector bundles has codimension at least $2$, we can restrict to study the open subset $U^{\star}_{r,d,g}\subset U_{r,d,g}$ of stable vector bundles on curves without non-trivial automorphisms. The good moduli morphism $\Psi_{r,d}:\Vr^{ss}\lra U_{r,d,g}$ is an isomorphism over $U^\star_{r,d,g}$. In other words, we have an isomorphism $\Vr^{\star}:=\Psi_{r,d}^{-1}(U^\star_{r,d,g})\cong U^{\star}_{r,d,g}$. Therefore, we get a natural surjective homomorphism
$$
\psi:\Pic(\Vr^{ss})\cong\Pic(\Vr^{s})\twoheadrightarrow \Pic(\Vr^\star)\cong\Pic(U^\star_{r,d,g}),
$$
where the first two homomorphisms are the restriction maps. When $g\geq 4$ the codimension of $\Vr^s\backslash\Vr^\star$ is at least two (see \cite[Remark 2.4]{GV}). Then the map $\psi$ is an isomorphism by Theorem \ref{picchow}. If $g=3$, the locus $\mt V^{(s)s}_{r,d,3}\backslash\mt V^\star_{r,d,3}$ is a divisor in $\mt V^{(s)s}_{r,d,3}$ (see \cite[Remark 2.4]{GV}). More precisely, it is the pull-back of the hyperelliptic (irreducible) divisor in $\mt M_3$. As line bundle, it is isomorphic to $\Lambda^9$ in the Picard group of $\mt M_3$ (see \cite[Chap. 3, Sec. E]{HM98}). Therefore, by Theorem \ref{picred}(i), we get that $\Pic(U^\star_{r,d,3})$ is the quotient of $\Pic(\mt V^{(s)s}_{r,d,3})$ by the relation $\Lambda(1,0,0)^9$. In particular, (when $g\geq 3$) the line bundle $\Theta^s\otimes\Xi^{t}\otimes \Lambda (1,0,0)^u$, where $(s,t,u)\in\mathbb Z^3$, on $U^\star_{r,d,g}$ has the same properties of the canonical line bundle $\mathscr L_{m,a}$ in \cite[Theorem 1]{Kou93}, where $m=s\cdot \frac{v_{1,d,g}}{v_{r,d,g}}$ and $a=-s(\alpha+\beta)-t\cdot k_{1,d,g}$.
\end{rmk}

By Theorem \ref{schmitt}, $\CVc^{Hss}$ admits a projective variety as good moduli space. This means, in particular, that the stacks $\CVc^{Hss}$ and $\CVr^{Hss}$ are of finite type and universally closed. Since any vector bundle contains the multiplication by scalars as automorphisms, $\CVc^{Hss}$ is not separated. The next Proposition tells us exactly when the rigidification $\CVr^{Hss}$ is separated.

\begin{prop}\label{poincare}The following conditions are equivalent:
\begin{enumerate}[(i)]
\item $n_{r,d}\cdot v_{r,d,g}=1$, i.e. $n_{r,d}=1$ and $v_{r,d,g}=1$.
\item There exists a universal vector bundle on the universal curve of an open substack of $\CVr$.
\item There exists a universal vector bundle on the universal curve of $\CVr$.
\item The stack $\CVr^{Hss}$ is proper.
\item All H-semistable points are H-stable.
\item $\CVr^{Hss}$ is a Deligne-Mumford stack. 
\item The stack $\CVr^{Pss}$ is proper.
\item All P-semistable points are P-stable.
\item $\CVr^{Pss}$ is a Deligne-Mumford stack.
\end{enumerate}
\end{prop}

\begin{proof}The strategy of the proof is the following$$$$
$$
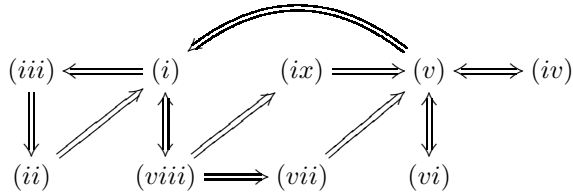
\xymatrix{
(iii)\ar@{=>}[d]    &   (i)\ar@{<=>}[d]\ar@{=>}[l]   &(ix)\ar@{=>}[r] &(v)\ar@{=>}@/_2pc/[ll]\ar@{<=>}[r] \ar@{<=>}[d]&(iv)\\
(ii)\ar@{=>}[ru]&  (viii)\ar@{=>}[r]\ar@{=>}[ur] &(vii)\ar@{=>}[ur]& (vi)
}
$$

$(i) \Rightarrow (iii)$. By Corollary \ref{eccocequasi}(i) and the exact sequence (\ref{lss}), the weight of a line bundle on $\CVc$ must be a multiple of $ n_{r,d}\cdot v_{r,d,g}$. In particular the condition $(i)$ is equivalent to have a line bundle $\mt L$ of weight $1$ on $\CVc$. Let $\left(\overline{\pi}:\overline{\mt Vec}_{r,d,g,1}\rightarrow \CVc,\mt E\right)$ be the universal pair, we can see that $\mt E\otimes \overline{\pi}^*\mt L^{-1}$ descends to a vector bundle on $\CVr$ with the universal property.

$(iii) \Rightarrow (ii)$ Obvious.

$(ii)\Rightarrow (i)$ Suppose that there exists a universal pair $(\mt S_1\rightarrow \mt S,\mt F)$ on some open substack $\mt S$ of $\CVr$. We can suppose that all the points $(C,\mt E)$ in $\mt S$ are such that $\Aut(C,\mt E)=\mathbb G_m$. Let $\nu_{r,d}:\mt T:=\nu_{r,d}^{-1}\mt S\rightarrow \mt S$ be the restriction of the rigidification map and $(\overline{\pi}:\mt T_1\rightarrow \mt T,\mt E)$ the universal pair on $\mt T\subset \CVc$. Then $$\overline{\pi}_*\left(Hom\left(\nu_{r,d}^*\mt F,\mt E\right)\right)$$ is a line bundle of weight $1$ on $\mt T$ and, by smoothness of $\CVc$, we can extend it to a line bundle of weight $1$ on $\CVc$.

$(iv) \Leftrightarrow (v)$. If all H-semistable points are H-stable, then  by \cite[Corollary 2.5]{Mum94} the action of $GL(V_n)$ on the H-semistable locus of $H_n$ is proper, i.e. the morphism $PGL\times H_n^{Hss}\rightarrow H_n^{Hss}\times H_{n}^{Hss}:(A,h)\mapsto (h,A.h)$ is proper (for $n$ big enough). The cartesian diagram
$$
\xymatrix{
PGL\times H_n^{Hss}\ar[r]\ar[d] & H_n^{Hss}\times H_n^{Hss}\ar[d]\\
\CVr^{Hss}\ar[r] &\CVr^{Hss}\times \CVr^{Hss}
}
$$
implies that the diagonal is proper, i.e. the stack is separated. We have already seen that it is always universally closed and of finite type, so it is proper. Conversely, if the diagonal is proper the automorphism group of any point must be finite, in particular there are no strictly H-semistable points.

$(v)\Leftrightarrow (vi)$. By \cite[Theorem 8.1]{LMB}, $\CVr^{Hss}$ is Deligne-Mumford if and only if the diagonal is unramified, which is also equivalent to the fact that the automorphism group of any point is a finite group (because we are working in characteristic $0$). As before, this happens if and only if all semistable points are stable.

$(v),(viii) \Rightarrow (i)$. It is known that, on smooth curves, $n_{r,d}=1$ if and only if all semistable vector bundles are stable. So we can suppose that $n_{r,d}=1$, so that $v_{r,d,g}=(2g-2, d+1-g, d+r(1-g))=(2g-2,d+r(1-g))$. If $v_{r,d,g}\neq 1$ we have $k_{r,d,g}<2g-2$, we can construct a nodal curve $C$ of genus $g$, composed by two irreducible smooth curves $C_1$ and $C_2$ meeting at $N$ points, such that $\omega_{C_1}=k_{r,d,g}$. In particular $(d_1,d_2):=(d\frac{\omega_{C_1}}{\omega_C}-N\frac{r}{2},d\frac{\omega_{C_2}}{\omega_C}+N\frac{r}{2})$ are integers. So we can construct a vector bundle $\mt E$ on $C$ with multidegree $(d_1,d_2)$ and rank $r$ satisfying the hypothesis of Lemma \ref{sHs}. This implies that the pair $(C,\mt E)$ must be strictly P-semistable and strictly H-semistable.

$(i)\Rightarrow (viii)$. Suppose that there exists a point $(C,\mt E)$ in $\CVr$ such that $(C^{st},\pi_*\mt E)$ is strictly P-semistable. If $C$ is smooth then $n_{r,d}\neq 1$ and we have done. Suppose that $n_{r,d}=1$ and $C$ singular. By hypothesis there exists a destabilizing subsheaf $\mt F\subset \pi_*\mt E$, such that
$$
\frac{\chi(\mt F)}{\sum s_i\omega_{C_i}}=\frac{\chi(\mt E)}{r\omega_C}.
$$
The equality can exist if and only if $(\chi(\mt E),r\omega_C)=(d+r(1-g),r(2g-2))\neq 1$. We have supposed that $d$ and $r$ are coprime, so $(d+r(1-g),r(2g-2))=(d+r(1-g), 2g-2)=(2g-2,d+1-g,d+r(1-g))=v_{r,d,g}$, which concludes the proof.

$(viii)\Rightarrow (vii),(ix)$. By hypothesis $\CVr^{Pss}=\CVr^{Ps}=\CVr^{Hss}=\CVr^{Hs}$, so $(vi)$ and $(viii)$ hold by what proved above.

$(vii), (ix)\Rightarrow (v)$. Suppose that $(v)$ does not hold, then there exists a strictly H-semistable point with automorphism group of positive dimension. Thus $\CVr^{Hss}$, and in particular $\CVr^{Pss}$, cannot be neither proper nor Deligne-Mumford.
\end{proof}

\appendix
\renewcommand*{\thesection}{\Alph{section}}

\section{The genus two case.}\label{genus2}
\setcounter{equation}{0}
In this appendix we will extend the Theorems \ref{pic} and \ref{picred} to the genus two case. The main results are the following
\begin{teoa}\label{pic2}Suppose that $r\geq 2$.
\begin{enumerate}[(i)]
\item The Picard groups of $\Vtc$ and $\Vtc^{ss}$ are generated by $\Lambda(1,0,0)$, $\Lambda(1,1,0)$, $\Lambda(0,1,0)$ and $\Lambda(0,0,1)$ with the unique relation 
\begin{equation}\label{relsm}
\Lambda(1,0,0)^{10}=\oo.
\end{equation}
\item The Picard groups of $\CVtc$ and $\CVtc^{Pss}$ are generated by $\Lambda(1,0,0)$, $\Lambda(1,1,0)$, $\Lambda(0,1,0)$, $\Lambda(0,0,1)$ and the boundary line bundles with the unique relation 
\begin{equation}\label{relcp}
\Lambda(1,0,0)^{10}=\oo\left(\widetilde\delta_0+2\sum_{j\in J_1}\widetilde\delta^j_1\right).
\end{equation}
\end{enumerate}
\end{teoa}
Let $v_{r,d,2}$ and $n_{r,d}$ be the numbers defined in the Notations \ref{notations}. Let $\alpha$ and $\beta$ be (not necessarily unique) integers such that $\alpha(d-1)+\beta(d+1)=-\frac{1}{n_{r,d}}\cdot\frac{v_{1,d,2}}{v_{r,d,2}}(d-r)$. We set
$$\Xi:=\Lambda(0,1,0)^{\frac{d+1}{v_{1,d,2}}}\otimes\Lambda(1,1,0)^{-\frac{d-1}{v_{1,d,2}}},\quad
\Theta:=\Lambda(0,0,1)^{\frac{r}{n_{r,d}}\cdot \frac{v_{1,d,2}}{v_{r,d,2}}}\otimes \Lambda(0,1,0)^{\alpha}\otimes \Lambda (1,1,0)^{\beta}.$$

\begin{teoa}\label{pic2red}Suppose that $r\geq2$.
\begin{enumerate}[(i)]
\item The Picard groups of $\Vr$ and $\Vr^{ss}$ are generated by $\Lambda(1,0,0)$, $\Xi$ and $\Theta$, with the unique relation (\ref{relsm}).
\item The Picard groups of $\CVtr$ and $\CVtr^{Pss}$ are generated by $\Lambda(1,0,0)$, $\Xi$, $\Theta$ and the boundary line bundles with the unique relation (\ref{relcp}).
\end{enumerate}
\end{teoa}
Unfortunately, at the moment we can not say if Theorems \ref{pic} and \ref{picred} are still true in genus two for the other open substacks in the assertions.

\begin{rmka}\label{g22}
Observe that, using Proposition \ref{T}, we can prove that Lemma \ref{redsemistable} holds also in genus two case. More precisely, we have that $\Pic(\CVtc)\cong\Pic(\CVtc^{Pss})\cong\Pic(\overline{\mt U}_n)$ and $\Pic(\Vtc)\cong\Pic(\Vtc^{ss})\cong\Pic(\mt U_n)$ for $n$ big enough.
\end{rmka}

We have analogous isomorphisms for the rigidified moduli stacks.\vspace{0.2cm}\\
\begin{dimo}\ref{pic2}(i) and \ref{pic2red}(i).
By the precedent observation, it is enough to prove the theorems for the semistable locus. Let $(C,\mt L)$ be a $k$-point of $\mt Jac_{d,2}$. We recall that Theorem \ref{picsmooth} says that the complex of groups
$$
0\lra \Pic(\mt Jac_{d,2})\lra \Pic(\Vtc^{ss})\lra \Pic(\Vl^{ss})\lra 0
$$
is exact. By Theorem \ref{fibers}, the cokernel is freely generated by the restriction of the line bundle $\Lambda(0,0,1)$ on the fiber $\Vl$. In particular the Picard groups of $\Vtc$ and $\Vtc^{ss}$ decomposes in the following way
$$
\Pic(\mt Jac_{d,2})\oplus\langle\Lambda(0,0,1)\rangle.
$$
By Theorem \ref{jc} and Lemma \ref{computations}, Theorem \ref{pic2}(i) follows. By Corollary \ref{eccocequasi}, Theorem \ref{pic2red}(i) also holds.
\end{dimo}\\

Now we are going to prove the Theorems \ref{pic2}(ii) and \ref{pic2red}(ii). First of all, by Theorems \ref{pic2}(i) and \ref{picchow}, we know that the Picard group of $\CVtc$ is generated by $\Lambda(1,0,0)$, $\Lambda(1,1,0)$, $\Lambda(0,1,0)$, $\Lambda(0,0,1)$ and the boundary line bundles. Consider the forgetful map $\overline{\phi}_{r,d}:\CVtc\to\overline{\mt M}_2$. By Theorem \ref{picmg}, the Picard group of $\overline{\mt M}_2$ is generated by the line bundles $\delta_0$, $\delta_1$ and the Hodge line bundle $\Lambda$, with the unique relation $\Lambda^{10}=\oo(\delta_0+2\delta_1)$. By pull-back along $\overline{\phi}_{r,d}$ we obtain the relation (\ref{relcp}). So for proving Theorem \ref{pic2}(ii), it remains to show that we do not have other relations on $\Pic(\CVtc)$.\vspace{0.2cm}

Suppose there exists another relation, i.e.
\begin{equation}\label{relmore}
\Lambda(1,0,0)^a\otimes\Lambda(1,1,0)^b\otimes\Lambda(0,1,0)^c\otimes\Lambda(0,0,1)^d\otimes\oo\left(e_0\widetilde\delta_0+\sum_{j\in J_1} e^j_1\widetilde\delta^j_1\right)=\oo
\end{equation}
where $a,b,c,d,e_o,e_1^j\in\mathbb Z$. By Theorem \ref{pic2}(i), the integers $b,c,d$ must be $0$ and $a$ must be a multiple of $10$. We set $a=10t$. Combining the equalities (\ref{relcp}) and (\ref{relmore}) we obtain:
\begin{equation}\label{rel}
\oo\left((e_0-t)\widetilde\delta_0+\sum_{j\in J_1}(e_1^j-2t)\widetilde\delta_1^j\right)=\oo
\end{equation}
where the integers $(e_0-t),(e_1^j-2t)$ cannot be all equal to $0$, because we have assumed that the two relations are independent. In other words the existence of two independent relations is equivalent to show that does not exist any relation among the boundary line bundles. We will show this arguing as in \S\ref{indipendece}. Observe that, arguing in the same way, we can arrive at same conclusions for the rigidified moduli stack $\CVtr$. \\\\
\textbf{The Family $\widetilde G$.}\\
Consider a double covering $Y'$ of $\mathbb P^2$ ramified along a smooth sextic $D$. Consider on it a general pencil of hyperplane sections. By blowing up $Y'$ at the base locus of the pencil we obtain a family $\varphi:Y\rightarrow \mathbb P^1$ of irreducible stable curves of genus two with at most one node. Moreover the two exceptional divisors $E_1$, $E_2\subset Y$ are sections of $\varphi$ through the smooth locus of $\varphi$. The vector bundle $\mt E:=\oo_Y\left(d E_1\right)\oplus\oo_Y^{r-1}$ is properly balanced of degree $d$. We call $G$ (resp. $\widetilde G$) the family of curves $\varphi:Y\rightarrow\mathbb P^1$ (resp. the family $\varphi$ with the vector bundle $\mt E$). We claim that
$$
\begin{cases}
\deg_{\widetilde G}\oo(\widetilde \delta_0)=30,\\
\deg_{\widetilde G}\oo(\widetilde \delta_1^j)=0 &\mbox{for any }j\in J_1.
\end{cases}
$$
The second result comes from the fact that all fibers of $\varphi$ are irreducible. We recall that, as \S\ref{indipendece}: $\deg_{\widetilde G}\oo(\widetilde \delta_0)=\deg_G\oo(\delta_0)$. So our problem is reduced to check the degree on $\overline{\mt M}_2$. Observe also that $Y$ is smooth and the generic fiber of $\varphi$ is a smooth curve. Since any fiber of $\varphi:Y\to\mathbb P^{-1}$ can have at most one node and the total space $Y$ is smooth, by \cite[Lemma 1]{AC87}, $\deg_{\widetilde G}\oo(\widetilde \delta_0)$ is equal to the number of singular fibers of $\varphi$. We can count them using the morphism $\varphi_D:D\rightarrow \mathbb P^1$, induced by the pencil restricted to the sextic $D$. By the generality of the pencil, we can assume that over any point of $\mathbb P^1$ there is at most one ramification point and that its ramification index at this point is $2$. So $\deg_{\widetilde G}\oo(\widetilde \delta_0)$ is equal to the degree of the ramification divisor in $D$. Using the Riemann-Hurwitz formula for the degree six morphism $\varphi_{D}$ we obtain the first equality.\\\\
\textbf{The Families $\widetilde G_1^j$.}\\
Consider a general pencil of cubics in $\mathbb P^2$. Blowing up the nine base points of the pencil, we obtain a family of irreducible stable elliptic curves $\phi:X\rightarrow\mathbb P^1$. The nine exceptional divisors $E_1,\ldots, E_9\subset X$ are sections of $\phi$ trough the smooth locus of $\phi$. The family will have twelve singular fibers consisting of irreducible nodal elliptic curves. Fix a smooth elliptic curve $\Gamma$ and a point $\gamma\in\Gamma$. We construct a surface $Y$ by setting
$$
Y=\left(X\coprod\left(\Gamma\times\mathbb P^1\right)\right)/\left(E_1\sim \{\gamma\}\times\Gamma\right)
$$
We get a family $f:X\rightarrow\mathbb P^1$ of stable curves of genus two. The general fiber is as in Figure \ref{Figure8} where $C$ is a smooth elliptic curve. While the twelve special fibers are as in Figure \ref{Figure9} where $C$ is a nodal irreducible elliptic curve.
\begin{figure}[h!]
\begin{center}
\unitlength .65mm 
\linethickness{0.4pt}
\ifx\plotpoint\undefined\newsavebox{\plotpoint}\fi 
\begin{picture}(122.75,50.5)(0,145)
\qbezier(47,153.5)(85.125,175.625)(122.75,183.25)
\qbezier(0,183)(25,185)(90,153)
\put(25,183.75){\makebox(0,0)[cc]{$C$}}
\put(113.25,183.75){\makebox(0,0)[cc]{$\Gamma$}}
\end{picture}
\end{center}
\caption{The general fiber of $f:X\to \mathbb P^1$.}\label{Figure8}
\end{figure}
\begin{figure}[h!]
\begin{center}
\unitlength .65mm 
\linethickness{0.4pt}
\ifx\plotpoint\undefined\newsavebox{\plotpoint}\fi 
\begin{picture}(122.75,50.5)(0,145)
\qbezier(77,153.5)(95.125,175.625)(122.75,183.25)
\qbezier(50,183)(25,180)(100,163)
\qbezier(0,153)(75,180)(50,183)
\put(5,160.75){\makebox(0,0)[cc]{$C$}}
\put(113.25,183.75){\makebox(0,0)[cc]{$\Gamma$}}
\end{picture}
\end{center}
\caption{The special fibers of $f:X\to \mathbb P^1$.}\label{Figure9}
\end{figure}
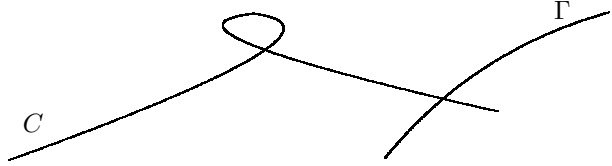
Choose a vector bundle $M^j$ of degree $\left\lceil \frac{d-r}{2}\right\rceil+j$ on $\Gamma$, pull it back to $\Gamma\times\mathbb P^1$ and call it again $M^j$. Since $M^j$ is trivial on $\{\gamma\}\times\mathbb P^1$, we can glue it with the vector bundle 
$$
\oo_X\left(\left(\left\lfloor\frac{d+r}{2}\right\rfloor-j\right)E_2\right)\oplus\oo_X^{r-1}
$$
on $X$ obtaining a vector bundle $\mt E^j$ on $f:X\to\mathbb P^1$ of rank $r$ and degree $d$. The next lemma follows easily
\begin{lema}The vector bundle $\mt E^j$ is a properly balanced for $j\in J_1=\{0,\ldots,\lfloor r/2\rfloor\}$.
\end{lema}
We call $G_1$ (resp. $\widetilde G_1^j$) the family of curves $f:X\to\mathbb P^1$ (resp. the family $f$ with the vector bundle $\mt E^j$). Moreover $\widetilde G_1^j$ does not intersect $\widetilde\delta_1^k$ for $j\neq k$. In particular $\deg_{\widetilde G_1^j}\oo(\widetilde \delta_1^j)=\deg_{G_1}\oo(\delta_1)$. By \cite[Lemma 1]{AC87}, the divisor $\oo(\delta_1)$ restricted to the family $G_1$ is isomorphic to the tensor product between the normal bundle of $E_1$ in $X$ and the normal bundle of $\gamma\times\mathbb P^1$ in $\Gamma\times\mathbb P^1$, i.e. $N_{E_1/X}\otimes N_{\{\gamma\}\times\mathbb P^1/\Gamma\times\mathbb P^1}$. The first factor has degree $-1$, while the second is trivial. Putting all together, we get
$$
\begin{cases}
\deg_{\widetilde G_1^k}\oo(\widetilde\delta_1^k)=-1,\\
\deg_{\widetilde G_1^k}\oo(\widetilde\delta_1^j)=0 &\mbox{if } j\neq k.
\end{cases}
$$

Now we can finally conclude the proof of Theorems \ref{pic2} and \ref{pic2red}.\vspace{0.2cm}\\
\begin{dimo} \ref{pic2}(ii) and \ref{pic2red}(ii). Suppose there exists a non-trivial relation $\oo(a_0\widetilde\delta_0+\sum a_1^j\widetilde\delta_1^j)=\oo$. If we restrict this equality on $\widetilde G$ we have $a_0=0$. Pulling back to $G_1^j$ we deduce $a_1^j=0$ for any $j\in J_1$. This concludes the proof of \ref{pic2}(ii). Repeating the same arguments for the rigidified moduli stack $\CVtr$ we prove Theorem \ref{pic2red}(ii).
\end{dimo}

\section{Base change cohomology for stacks admitting a good moduli space.}\label{App}
We will prove that the classical results of base change cohomology for proper schemes continue to hold again for (not necessarily proper) stacks, which admit a proper scheme as good moduli space (in the sense of Alper). The propositions and proofs are essentially equal to ones in \cite[Appendix A]{Br2}, but we rewrite them, because our hypotheses are different.\\
In this section, $\mathcal X$ will be an Artin stack of finite type over a scheme $S$, and a sheaf $\mathcal F$ will be a sheaf for the site lisse-\'etale defined in \cite[Sec. 12]{LMB} (see also \cite[Appendix A]{Br1}).
Recall first the definition of good moduli space.
\begin{defina}\cite[def 4.1]{Al} Let $S$ be a scheme, $\mathcal X$ be an Artin Stack over $S$ and $X$ an algebraic space over $S$ . We call an $S$-morphism $\pi:\mathcal X\rightarrow X$ a good moduli space if
\begin{itemize}
	\item $\pi$ is quasi-compact,
	\item $\pi_*$ is exact,
	\item The structural morphism $\oo_{ X}\rightarrow\pi_*\oo_{\mathcal X}$ is an isomorphism.
\end{itemize}
\end{defina}

\begin{rmka}\label{goodquotient} Let $\mt X$ be a quotient stack of a quasi-compact $k$-scheme $X$ by a smooth affine linearly reductive group scheme $G$. Suppose that $\mt L$ is a $G$-linearization on $X$. By\cite[Theorem 13.6 and Remark 13.7]{Al}, the GIT good quotient $X^{ss}_{\mt L}\sslash_{\mt L} G$ is a good moduli space for the open substack $\left[X^{ss}_{\mt L}/ G\right]$.
	
Conversely, suppose that there exists a $G$-stable open $U\subset X$ such that the open substack $\left[U/G\right]$ admits a good moduli space $Y$. By \cite[Theorem 11.14]{Al},  there exists a $G$-linearized line bundle $\mt L$ over $X$ such that $U$ is contained in $X^{ss}_{\mt L}$, $[U/G]$ is saturated with respect to the morphism $[X^{ss}_{\mt L}/G]\to X^{ss}_{\mt L}\sslash_{\mt L}G$ and $Y$ is the GIT good quotient $U\sslash_{\mt L}G$.
\end{rmka}

Before stating the main result of this Appendix, we need to recall the following
\begin{lema}$($\cite[Lemma 1, II]{Mum70}, see also \cite[Lemma A.1.3]{Br2}$)$.
\begin{enumerate}[(i)]
	\item Let $A$ be a ring and let $C^{\bullet}$ be a complex of $A$-modules such that $C^p\neq 0$ only if $0\leq p\leq n$. Then there exists a complex $K^{\bullet}$ of $A$-modules such that $K^p\neq 0$ only if $0\leq p\leq n$ and $K^p$ is free if $1\leq p\leq n$, and a quasi-isomorphism of complexes $K^{\bullet}\rightarrow C^{\bullet}$. Moreover, if the $C^p$ are flat, then $K^0$ will be $A$-flat too.
	\item If $A$ is noetherian and if the $H^i(C^\bullet)$ are finitely generated $A$-modules, then the $K^p$'s can be chosen to be finitely generated.
\end{enumerate}
\end{lema}

\begin{propa}\label{lasvolta}Let $\mathcal X$ be a quasi-compact Artin stack over an affine scheme (resp. noetherian affine scheme) $S=Spec(A)$. Let $\pi:\mathcal X\rightarrow X$ be a good moduli space with $X$ a separated (resp. proper) scheme over $S$. Let $\mathcal F$ be a quasi-coherent (resp. coherent) sheaf on $\mathcal{X}$ that is flat over $S$. Then there is a complex of flat $A$-modules (resp. of finite type)
$$
0\longrightarrow M^0\longrightarrow M^1\longrightarrow\dots\longrightarrow M^n\longrightarrow 0
$$
with $M^i$ free over $A$ for $1\leq i\leq n$, and isomorphisms
$$
H^i(M^{\bullet}\otimes_A A')\longrightarrow H^i(\mathcal{X}\otimes_A A',\mathcal{F}\otimes_A A')
$$
functorial in the $A$-algebra $A'$.
\end{propa}

\begin{proof} We consider the Cech complex $C^{\bullet}(\mathcal{U},\pi_*\mathcal F)$ associated to an affine covering $\mathcal U=(U_i)_{i\in I}$ of $X$. It is a finite complex of flat (by \cite[Theorem 4.16(ix)]{Al}) $A$-modules. Moreover, since $X$ is separated, then we have $H^i(C^{\bullet}(\mathcal{U},\pi_*\mathcal F))\cong H^i(X,\pi_* \mathcal F)$. If $A'$ is an $A$-algebra, then the covering $\mathcal U\otimes_A A'$ is still affine by $S$-separatedness of $X$. This implies that
$$
H^i(C^{\bullet}(\mathcal{U},\pi_*\mathcal F)\otimes_A A')\cong H^i(X\otimes_A A',(\pi_* \mathcal F)\otimes_A A').
$$
By \cite[Proposition 4.5]{Al}, we have
$$
H^i(X\otimes_A A',(\pi_* \mathcal F)\otimes_A A')\cong H^i( X\otimes_A A', \pi_*(\mathcal F\otimes_A A')).
$$
Since $\pi_*$ is exact, the Leray-spectral sequence $H^i(X\otimes_A A',R^j\pi_*(\mt F\otimes_A A'))\Rightarrow H^{i+j}(\mt X\otimes_A A',\mt (F\otimes_A A'))$ (see \cite[Theorem A.1.6.4]{Br1}) degenerates in the isomorphisms $H^i(X\otimes_A A',\pi_*(\mathcal F\otimes_A A'))\cong H^i(\mt X\otimes_A A',\mathcal F\otimes_A A')$. Putting all together:
$$
H^i(C^{\bullet}(\mathcal{U},\pi_*\mathcal F)\otimes_A A')\cong H^i(\mt X\otimes_A A',\mathcal F\otimes_A A').
$$
It can be check that such isomorphisms are functorial in the $A$-algebra $A'$. Observe that if $\mathcal F$ is coherent then also $\pi_*\mathcal F$ is coherent (see \cite[Theorem 4.16(x)]{Al}). So if $X$ is proper, then the modules $H^i(\mathcal{X},\mathcal{F})$ are finitely generated. In particular, the cohomology modules of the complex $C^{\bullet}(\mathcal{U},\pi_*\mathcal F)$ are finitely generated. We can use the previous lemma for conclude the proof.
\end{proof}

From the above results, we deduce several useful corollaries, whose proofs are the same of \cite[II.5]{Mum70}.



\begin{cora}Let $\mathcal X\rightarrow X$ be a good moduli space over a scheme $S$, with $X$ proper scheme over $S$ and $\mathcal F$ a coherent sheaf over $\mathcal X$ flat over $S$. Then we have:
\begin{enumerate}[(i)]
  \item for any $ p\geq 0$ the function $S\rightarrow \mathbb Z$ defined by $s\mapsto dim_{k(s)}H^i(\mathcal X_s,\mathcal F_s)$ is upper semicontinuous on $S$.
  \item The function $S\rightarrow \mathbb Z$  defined by $s\mapsto \chi (\mathcal{F}_s)$ is locally constant.
\end{enumerate}
\end{cora}

\begin{cora}\label{TFAE}
Let $\mathcal X\rightarrow X$ be a good moduli space over an integral scheme $S$, with $X$ proper scheme over $S$ and $\mathcal F$ a coherent sheaf over $\mt X$ flat over $S$. The following conditions are equivalent
\begin{enumerate}[(i)]
  \item $s\mapsto dim_{k(s)}H^i(\mathcal X_s,\mathcal F_s)$ is a constant function,
  \item $R^iq_*(\mathcal F)$ is locally free sheaf on $S$ and for any $s\in S$ the map
      $$
      R^iq_*(\mathcal F)\otimes k(s)\rightarrow H^i(\mathcal X_s,\mathcal F_s)
      $$
      is an isomorphism.
\end{enumerate}
If these conditions are satisfied, then we have an isomorphism
$$
R^{i-1}q_*(\mathcal F)\otimes k(s)\rightarrow H^{i-1}(\mathcal X_s,\mathcal F_s).
$$
\end{cora}

\begin{cora}Let $\mathcal X\rightarrow X$ be a good moduli space over a scheme $S$, with $X$ proper scheme over $S$ and $\mathcal F$ a coherent sheaf over $\mathcal X$ flat over $S$. Assume for some $i$ that $H^i(\mathcal X_s,\mathcal F_y)=(0)$ for any $s\in S$. Then the natural map
$$
R^{i-1}q_*(\mathcal F)\otimes_{\oo_S}k(s)\rightarrow H^{i-1}(\mathcal X_s,\mathcal F_y)
$$
is an isomorphism for any $s\in S$.
\end{cora}

\begin{cora}Let $\mathcal X\rightarrow X$ be a good moduli space over a scheme $S$, with $X$ proper scheme  and $\mathcal F$ a coherent sheaf over $\mt X$ flat over $S$. If $R^i q_*(\mathcal F)=(0)$ for $i\geq i_0$ then $H^i(\mathcal X_s,\mathcal F_s)=(0)$ for any $s\in S$ and $i\geq i_0$.
\end{cora}

\begin{cora}\label{seesaw}$[$The SeeSaw Principle$]$.\\
Let $\mathcal X\rightarrow X$ be a good moduli space over an integral scheme $S$ and $\mathcal L$ be a line bundle on $\mt X$. Suppose that $q:\mt X\to S$ is flat and that $X\rightarrow S$ is proper with integral geometric fibers. Then the locus
$$
S_1=\{s\in S|\mathcal L_s \cong \oo_{\mathscr X_s}\}
$$
is closed in $S$. Moreover if we call $q_1:\mathcal X\times_{S} S_1 \rightarrow S_1$ the restriction of $q$ on this locus, then $q_{1*}\mathcal L$ is a line bundle on $S$ and the natural morphism $q_1^*q_{1*}\mathcal{L}\cong \mathcal L$ is an isomorphism.
\end{cora}

\begin{proof}A line bundle $\mathcal M$ on a stack $\mathcal X$ with a proper integral good moduli space $X$ is trivial if and only if $h^0(\mathcal M)>0$ and $h^0(\mathcal M^{-1})>0$. The necessity is obvious. Conversely suppose that these conditions hold. Then we have two non-zero homomorphisms $s:\oo_{\mathcal X}\rightarrow \mathcal M$, $t:\oo_{\mathcal X}\rightarrow \mathcal M^{-1}$. If we dualize the second one and compose with the first one, we have a non-zero morphism $h:\oo_{\mathcal X}\rightarrow \oo_{\mathcal X}$. Now $X$ is an integral proper scheme then $H^0(X,\oo_X)=k$ so $H^0(\mathcal X,\oo_{\mathcal X})=k$. Hence $h$ is an isomorphism. This implies that also $s$ and $t$ are isomorphisms. As a consequence, we have
$$
S_1=\{s\in S|h^0(\mathcal X_s,\mathcal L_s)>0,\,h^0(\mathcal X_s,\mathcal L_s^{-1})>0\}.
$$
In particular, $S_1$ is closed by upper semicontinuity. Up to restriction we can assume $S=S_1$, so the function $s\mapsto h^0(\mathcal X_s,\mathcal L_s)=1$ is constant. By Corollary \ref{TFAE}, $q_*\mathcal L$ is a line bundle on $S$ and the natural map $q_*\mathcal L\otimes_{\oo_S} k(s)\rightarrow H^0(\mathcal X_s,\mathcal L_s)$ is an isomorphism. Consider the natural map $\pi:q^*q_*\mathcal L\rightarrow \mathcal L$. Its restriction on any fiber $\mathcal X_s$ 
$$
\oo_{\mathcal X_s}\otimes H^0(\mathcal X_s,\mathcal L_s)\rightarrow \mathcal L_s
$$
is an isomorphism. In particular $\pi$ is an isomorphism for any geometric point $x\in \mt X$. Since it is a map between line bundles, by Nakayama lemma, it is an isomorphism.
\end{proof}
\bibliographystyle{abbrv}
\bibliography{BiblioVec}

\begin{thebibliography}{10}

\bibitem{ACV}
D.~Abramovich, A.~Corti, and A.~Vistoli.
\newblock Twisted bundles and admissible covers.
\newblock {\em Comm. Algebra}, 31(8):3547--3618, 2003.
\newblock Special issue in honor of Steven L. Kleiman.

\bibitem{Al10}
J.~Alper.
\newblock On the local quotient structure of {A}rtin stacks.
\newblock {\em J. Pure Appl. Algebra}, 214(9):1576--1591, 2010.

\bibitem{Al}
J.~Alper.
\newblock Good moduli spaces for {A}rtin stacks.
\newblock {\em Ann. Inst. Fourier (Grenoble)}, 63(6):2349--2402, 2013.

\bibitem{AC87}
E.~Arbarello and M.~Cornalba.
\newblock The {P}icard groups of the moduli spaces of curves.
\newblock {\em Topology}, 26(2):153--171, 1987.

\bibitem{ACG11}
E.~Arbarello, M.~Cornalba, and P.~A. Griffiths.
\newblock {\em Geometry of algebraic curves. {V}olume {II}}, volume 268 of {\em
  Grundlehren der Mathematischen Wissenschaften [Fundamental Principles of
  Mathematical Sciences]}.
\newblock Springer, Heidelberg, 2011.
\newblock With a contribution by Joseph Daniel Harris.

\bibitem{BFMV}
G.~Bini, F.~Felici, M.~Melo, and F.~Viviani.
\newblock {\em Geometric invariant theory for polarized curves}, volume 2122 of
  {\em Lecture Notes in Mathematics}.
\newblock Springer, Cham, 2014.

\bibitem{BH12}
I.~Biswas and N.~Hoffmann.
\newblock Poincar\'e families of {$G$}-bundles on a curve.
\newblock {\em Math. Ann.}, 352(1):133--154, 2012.

\bibitem{Br1}
S.~Brochard.
\newblock Foncteur de {P}icard d'un champ alg\'ebrique.
\newblock {\em Math. Ann.}, 343(3):541--602, 2009.

\bibitem{Br2}
S.~Brochard.
\newblock Finiteness theorems for the {P}icard objects of an algebraic stack.
\newblock {\em Adv. Math.}, 229(3):1555--1585, 2012.

\bibitem{Cap94}
L.~Caporaso.
\newblock A compactification of the universal {P}icard variety over the moduli
  space of stable curves.
\newblock {\em J. Amer. Math. Soc.}, 7(3):589--660, 1994.

\bibitem{CMKV}
S.~Casalaina-Martin, J.~L. Kass, and F.~Viviani.
\newblock The local structure of compactified {J}acobians.
\newblock {\em Proc. Lond. Math. Soc. (3)}, 110(2):510--542, 2015.

\bibitem{Cor07}
M.~Cornalba.
\newblock The {P}icard group of the moduli stack of stable hyperelliptic
  curves.
\newblock {\em Atti Accad. Naz. Lincei Cl. Sci. Fis. Mat. Natur. Rend. Lincei
  (9) Mat. Appl.}, 18(1):109--115, 2007.

\bibitem{Del87}
P.~Deligne.
\newblock Le d\'eterminant de la cohomologie.
\newblock In {\em Current trends in arithmetical algebraic geometry ({A}rcata,
  {C}alif., 1985)}, volume~67 of {\em Contemp. Math.}, pages 93--177. Amer.
  Math. Soc., Providence, RI, 1987.

\bibitem{DM69}
P.~Deligne and D.~Mumford.
\newblock The irreducibility of the space of curves of given genus.
\newblock {\em Inst. Hautes \'Etudes Sci. Publ. Math.}, (36):75--109, 1969.

\bibitem{Edi12}
D.~Edidin.
\newblock Equivariant geometry and the cohomology of the moduli space of
  curves.
\newblock In {\em Handbook of moduli. {V}ol. {I}}, volume~24 of {\em Adv. Lect.
  Math. (ALM)}, pages 259--292. Int. Press, Somerville, MA, 2013.

\bibitem{EG98}
D.~Edidin and W.~Graham.
\newblock Equivariant intersection theory.
\newblock {\em Invent. Math.}, 131(3):595--634, 1998.

\bibitem{Fa96}
G.~Faltings.
\newblock Moduli-stacks for bundles on semistable curves.
\newblock {\em Math. Ann.}, 304(3):489--515, 1996.

\bibitem{FAG}
B.~Fantechi, L.~G{\"o}ttsche, L.~Illusie, S.~L. Kleiman, N.~Nitsure, and
  A.~Vistoli.
\newblock {\em Fundamental algebraic geometry}, volume 123 of {\em Mathematical
  Surveys and Monographs}.
\newblock American Mathematical Society, Providence, RI, 2005.
\newblock Grothendieck's FGA explained.

\bibitem{Gie84}
D.~Gieseker.
\newblock A degeneration of the moduli space of stable bundles.
\newblock {\em J. Differential Geom.}, 19(1):173--206, 1984.

\bibitem{Gi}
J.~Giraud.
\newblock {\em Cohomologie non ab\'elienne}.
\newblock Springer-Verlag, Berlin-New York, 1971.
\newblock Die Grundlehren der mathematischen Wissenschaften, Band 179.

\bibitem{GV}
S.~Gorchinskiy and F.~Viviani.
\newblock Picard group of moduli of hyperelliptic curves.
\newblock {\em Math. Z.}, 258(2):319--331, 2008.

\bibitem{Gr14}
M.~Grimes.
\newblock {\em Universal moduli spaces of vector bundles and the log-minimal
  model program on the moduli of curves}.
\newblock Available at http://arxiv.org/abs/1409.5734.

\bibitem{Har83}
J.~Harer.
\newblock The second homology group of the mapping class group of an orientable
  surface.
\newblock {\em Invent. Math.}, 72(2):221--239, 1983.

\bibitem{HM98}
J.~Harris and I.~Morrison.
\newblock {\em Moduli of curves}, volume 187 of {\em Graduate Texts in
  Mathematics}.
\newblock Springer-Verlag, New York, 1998.

\bibitem{H07}
N.~Hoffmann.
\newblock Rationality and {P}oincar\'e families for vector bundles with extra
  structure on a curve.
\newblock {\em Int. Math. Res. Not. IMRN}, (3):Art. ID rnm010, 30, 2007.

\bibitem{Ho10}
N.~Hoffmann.
\newblock Moduli stacks of vector bundles on curves and the {K}ing-{S}chofield
  rationality proof.
\newblock In {\em Cohomological and geometric approaches to rationality
  problems}, volume 282 of {\em Progr. Math.}, pages 133--148. Birkh\"auser
  Boston, Inc., Boston, MA, 2010.

\bibitem{H}
N.~Hoffmann.
\newblock The {P}icard group of a coarse moduli space of vector bundles in
  positive characteristic.
\newblock {\em Cent. Eur. J. Math.}, 10(4):1306--1313, 2012.

\bibitem{HL}
D.~Huybrechts and M.~Lehn.
\newblock {\em The geometry of moduli spaces of sheaves}.
\newblock Cambridge Mathematical Library. Cambridge University Press,
  Cambridge, second edition, 2010.

\bibitem{K}
I.~Kausz.
\newblock A {G}ieseker type degeneration of moduli stacks of vector bundles on
  curves.
\newblock {\em Trans. Amer. Math. Soc.}, 357(12):4897--4955 (electronic), 2005.

\bibitem{KM76}
F.~F. Knudsen and D.~Mumford.
\newblock The projectivity of the moduli space of stable curves. {I}.
  {P}reliminaries on ``det'' and ``{D}iv''.
\newblock {\em Math. Scand.}, 39(1):19--55, 1976.

\bibitem{Kou91}
A.~Kouvidakis.
\newblock The {P}icard group of the universal {P}icard varieties over the
  moduli space of curves.
\newblock {\em J. Differential Geom.}, 34(3):839--850, 1991.

\bibitem{Kou93}
A.~Kouvidakis.
\newblock On the moduli space of vector bundles on the fibers of the universal
  curve.
\newblock {\em J. Differential Geom.}, 37(3):505--522, 1993.

\bibitem{LMB}
G.~Laumon and L.~Moret-Bailly.
\newblock {\em Champs alg\'ebriques}, volume~39 of {\em Ergebnisse der
  Mathematik und ihrer Grenzgebiete. 3. Folge. A Series of Modern Surveys in
  Mathematics [Results in Mathematics and Related Areas. 3rd Series. A Series
  of Modern Surveys in Mathematics]}.
\newblock Springer-Verlag, Berlin, 2000.

\bibitem{Mel09}
M.~Melo.
\newblock Compactified {P}icard stacks over {$\overline{\mathcal M}_g$}.
\newblock {\em Math. Z.}, 263(4):939--957, 2009.

\bibitem{MV}
M.~Melo and F.~Viviani.
\newblock The {P}icard group of the compactified universal {J}acobian.
\newblock {\em Doc. Math.}, 19:457--507, 2014.

\bibitem{MR85}
N.~Mestrano and S.~Ramanan.
\newblock Poincar\'e bundles for families of curves.
\newblock {\em J. Reine Angew. Math.}, 362:169--178, 1985.

\bibitem{Mum65}
D.~Mumford.
\newblock Picard groups of moduli problems.
\newblock In {\em Arithmetical {A}lgebraic {G}eometry ({P}roc. {C}onf. {P}urdue
  {U}niv., 1963)}, pages 33--81. Harper \& Row, New York, 1965.

\bibitem{Mum70}
D.~Mumford.
\newblock {\em Abelian varieties}.
\newblock Tata Institute of Fundamental Research Studies in Mathematics, No. 5.
  Published for the Tata Institute of Fundamental Research, Bombay; Oxford
  University Press, London, 1970.

\bibitem{Mum83}
D.~Mumford.
\newblock Towards an enumerative geometry of the moduli space of curves.
\newblock In {\em Arithmetic and geometry, {V}ol. {II}}, volume~36 of {\em
  Progr. Math.}, pages 271--328. Birkh\"auser Boston, Boston, MA, 1983.

\bibitem{Mum94}
D.~Mumford, J.~Fogarty, and F.~Kirwan.
\newblock {\em Geometric invariant theory}, volume~34 of {\em Ergebnisse der
  Mathematik und ihrer Grenzgebiete (2) [Results in Mathematics and Related
  Areas (2)]}.
\newblock Springer-Verlag, Berlin, third edition, 1994.

\bibitem{NS}
D.~S. Nagaraj and C.~S. Seshadri.
\newblock Degenerations of the moduli spaces of vector bundles on curves. {II}.
  {G}eneralized {G}ieseker moduli spaces.
\newblock {\em Proc. Indian Acad. Sci. Math. Sci.}, 109(2):165--201, 1999.

\bibitem{P96}
R.~Pandharipande.
\newblock A compactification over {$\overline {M}\sb g$} of the universal
  moduli space of slope-semistable vector bundles.
\newblock {\em J. Amer. Math. Soc.}, 9(2):425--471, 1996.

\bibitem{Sceq}
A.~Schmitt.
\newblock The equivalence of {H}ilbert and {M}umford stability for vector
  bundles.
\newblock {\em Asian J. Math.}, 5(1):33--42, 2001.

\bibitem{Sc}
A.~Schmitt.
\newblock The {H}ilbert compactification of the universal moduli space of
  semistable vector bundles over smooth curves.
\newblock {\em J. Differential Geom.}, 66(2):169--209, 2004.

\bibitem{Ser06}
E.~Sernesi.
\newblock {\em Deformations of algebraic schemes}, volume 334 of {\em
  Grundlehren der Mathematischen Wissenschaften [Fundamental Principles of
  Mathematical Sciences]}.
\newblock Springer-Verlag, Berlin, 2006.

\bibitem{Se82}
C.~S. Seshadri.
\newblock {\em Fibr\'es vectoriels sur les courbes alg\'ebriques}, volume~96 of
  {\em Ast\'erisque}.
\newblock Soci\'et\'e Math\'ematique de France, Paris, 1982.
\newblock Notes written by J.-M. Drezet from a course at the {\'E}cole Normale
  Sup{\'e}rieure, June 1980.

\bibitem{stacks-project}
T.~{Stacks Project Authors}.
\newblock \textit{Stacks Project}.
\newblock \url{http://stacks.math.columbia.edu}, 2018.

\bibitem{T91b}
M.~Teixidor~i Bigas.
\newblock Brill-{N}oether theory for stable vector bundles.
\newblock {\em Duke Math. J.}, 62(2):385--400, 1991.

\bibitem{T95}
M.~Teixidor~i Bigas.
\newblock Moduli spaces of vector bundles on reducible curves.
\newblock {\em Amer. J. Math.}, 117(1):125--139, 1995.

\bibitem{T98}
M.~Teixidor~i Bigas.
\newblock Compactifications of moduli spaces of (semi)stable bundles on
  singular curves: two points of view.
\newblock {\em Collect. Math.}, 49(2-3):527--548, 1998.
\newblock Dedicated to the memory of Fernando Serrano.

\bibitem{Tu}
L.~W. Tu.
\newblock Semistable bundles over an elliptic curve.
\newblock {\em Adv. Math.}, 98(1):1--26, 1993.

\bibitem{Vis98}
A.~Vistoli.
\newblock The {C}how ring of {$\mathcal M_2$}. {A}ppendix to ''equivariant
  intersection theory''.
\newblock {\em Invent. Math.}, 131(3):595--634, 1998.

\bibitem{W}
J.~Wang.
\newblock The moduli stack of {G}-bundles.
\newblock Available at http://arxiv.org/abs/1104.4828.

\end{thebibliography}
\end{document}